\documentclass[francais]{article}
\usepackage[T1]{fontenc}
\usepackage[latin1]{inputenc}
\usepackage{graphicx}
\usepackage{amssymb}
\usepackage{amsmath}
\usepackage{amsthm}
\usepackage{color}
\usepackage[all]{xy}
\usepackage{amsfonts}
\usepackage{mathrsfs}
\usepackage{latexsym}
\usepackage{geometry} 
\usepackage{textcomp}
\usepackage[english,frenchb]{babel}

\theoremstyle{plain}
\newtheorem*{bigtheo}{Théorème}
\newtheorem{theo}{Théorème}[section]
\newtheorem{prop}[theo]{Proposition}
\newtheorem{lemm}[theo]{Lemme}
\newtheorem{coro}[theo]{Corollaire}
\theoremstyle{definition}
\newtheorem{defi}[theo]{Définition}
\theoremstyle{remark}
\newtheorem{rema}[theo]{Remarque}

\title{Formes modulaires surconvergentes, ramification et classicité}
\author{Stéphane Bijakowski}
\begin{document}
\maketitle

\begin{abstract}
Nous prouvons un résultat de classicité pour les formes modulaires surconvergentes sur les variétés de Shimura PEL de type (A) ou (C), sans hypothèse de ramification. Nous utilisons une méthode de prolongement analytique, qui généralise des résultats antérieurs dans le cas non ramifié. Nous travaillons avec le modèle rationnel de la variété de Shimura, et utilisons un plongement dans la variété de Siegel pour définir les structures entières sur l'espace rigide.
\end{abstract}

\selectlanguage{english}
\begin{abstract}
We prove in this paper a classicality result for overconvergent modular forms on PEL Shimura varieties of type (A) or (C), without any ramification hypothesis. We use an analytic continuation method, which generalizes previous results in the unramified setting. We work with the rational model of the Shimura variety, and use an embedding into the Siegel variety to define the integral structures on the rigid space.
\end{abstract}

\selectlanguage{frenchb}
\tableofcontents

\section*{Introduction}
\addcontentsline{toc}{section}{Introduction}

Les formes modulaires $p$-adiques et surconvergentes ont été introduites pour étudier des congruences entre formes modulaires. Pour prouver des propriétés sur ces nouveaux objets, il apparaît important de montrer qu'ils sont proches des formes classiques. Le résultat de Coleman (\cite{Co}) montre qu'une forme surconvergente propre pour l'opérateur de Hecke $U_p$ est classique si le poids est suffisamment grand devant la valuation de la valeur propre. Plus précisément, on a le résultat suivant.

\begin{bigtheo}
Soit $f$ une forme surconvergente de poids $k \in \mathbb{Z}$, propre pour $U_p$ de valeur propre $a_p$. Si $k > 1 + v(a_p)$, alors $f$ est classique.
\end{bigtheo}

La preuve originale de Coleman repose sur une étude de la cohomologie de la courbe modulaire rigide. Buzzard (\cite{Bu}) et Kassaei (\cite{Ka}) ont donné une nouvelle preuve de ce théorème en utilisant une méthode de prolongement analytique. Plus précisément, une forme modulaire surconvergente peut être vue comme une section d'un faisceau sur un voisinage strict du lieu ordinaire-multiplicatif (c'est-à-dire le lieu où le sous-groupe universel de la $p$-torsion de la courbe elliptique est multiplicatif). En utilisant la dynamique de l'opérateur de Hecke $U_p$, on peut étendre facilement la forme au lieu supersingulier. Sur le lieu ordinaire-étale, les auteurs arrivent à construire des séries approchant la forme désirée, et arrivent à les recoller avec la forme initiale sous la condition du théorème. Cela prouve que la forme surconvergente peut être étendue à toute la variété rigide, et donc est classique. \\

De nombreux travaux ont ensuite été faits pour des variétés de Shimura plus générales. Citons notamment les résultats de Sasaki (\cite{Sa}), Pilloni-Stroh (\cite{P-S1}), Tian-Xiao (\cite{T-X}), Johansson (\cite{Jo}) dans le cas Hilbert, ainsi que Pilloni-Stroh (\cite{P-S2}) pour les variétés de Shimura déployées. \\
L'auteur, dans un article publié avec Pilloni et Stroh (\cite{Bi}), a notamment généralisé le résultat de classicité pour les variétés de Shimura avec bonne réduction, c'est-à-dire en supposant que le nombre premier $p$ est non ramifié dans la donnée de Shimura. Ce résultat a été obtenu en généralisant la méthode du prolongement analytique. \\
Le cas où le nombre premier $p$ est ramifié pose des problèmes techniques. Il est possible, en adaptant cette méthode, de prolonger la forme surconvergente au lieu de bonne réduction. Pour conclure à la classicité, il faudrait disposer de modèles entiers des compactifications toroïdales, et démontrer un principe de Koecher. C'est ce qui a été fait par l'auteur dans le cas Hilbert (\cite{Bi_Hilbert}). Notons également que Johansson (\cite{Jo}) obtient des résultats dans ce cas. \\

Pour des variétés de Shimura plus générales, il semble très technique de construire des modèles entiers des compactifications. Les modèles rationnels des compactifications ont été construits par Pink (\cite{Pin}). Il est peut-être possible de les définir par normalisation dans un autre espace. En effet, si $X$ désigne la variété de Shimura entière, il existe un morphisme $X \to Y$, où $Y$ est une variété de Siegel, pour laquelle il est possible de construire des modèles entiers des compactifications. On peut alors définir une compactification de $X$ comme la normalisation de cet espace dans une compactification de $Y$. La difficulté technique est alors de prouver que cet espace vérifie les propriétés attendues, notamment le principe de Koecher. \\
\indent Pour éviter ces difficultés, nous avons décidé de travailler avec le modèle rationnel de la variété de Shimura, et l'espace analytique associé. Rappelons que si $K$ est une extension finie de $\mathbb{Q}_p$, et $X$ un $K$-schéma de présentation finie, alors on peut associer à $X$ un espace rigide $X^{an}$, l'analytifié de $X$, dont les $\overline{K}$-points sont les mêmes que ceux de $X$. Nous travaillons donc avec l'analytifié de la variété de Shimura. Les principales difficultés concernent les structures entières, qui étaient présentes naturellement dans le cas de bonne réduction. En particulier, il est nécessaire de définir le degré d'un sous-groupe de la variété abélienne, ainsi qu'une norme sur l'espace des formes modulaires. Si $x$ est un point de cet espace analytique, il correspond à une variété abélienne $A$ définie sur une extension finie $L$ de $\mathbb{Q}_p$, avec des structures additionnelles. D'après un théorème de réduction semi-stable de Grothendieck, on sait que quitte à étendre $L$, il existe un schéma semi-abélien $A_0$ sur $O_L$ égal à $A$ en fibre générique. En utilisant ce schéma semi-abélien, on peut donc définir les degrés pour les sous-groupes de $A$, ainsi qu'un modèle entier pour l'espace vectoriel $\omega_A$. \\
\indent Cette définition point par point des structures entières peut être globalisée de la manière suivante. Soit $X$ la variété de Shimura sur $K$ considérée et $X^{an}$ son analytifié ; alors il existe un morphisme de $X$ vers une variété de Siegel $A_g$. Soit $\overline{A_g}$ une compactification entière de $A_g$, et $\overline{A_g}^{rig}$ l'espace rigide associé. Alors on a un morphisme $X^{an} \to \overline{A_g}^{rig}$. Les structures entières définies sur $\overline{A_g}^{rig}$ peuvent donc se transporter naturellement sur $X^{an}$. \\
\indent Puisque nous n'utilisons pas les modèles entiers des variétés de Shimura, nous devons modifier notre définition des formes modulaires surconvergentes. Dans les papiers précédents, nous utilisions une forme faible des formes surconvergentes : il s'agissait de sections définies sur un voisinage strict du lieu ordinaire-multiplicatif dans l'espace rigide $X_{rig}$ associé au modèle entier de la variété de Shimura. Ici, puisque nous travaillons avec l'espace analytifié $X^{an}$, nous devons changer cette définition. Une définition forte des formes surconvergentes est alors une section définie sur un voisinage strict du lieu ordinaire-multiplicatif dans l'espace rigide $\overline{X}^{an}$, où $\overline{X}$ est une compactification rationnelle de $X$ et $\overline{X}^{an}$ son analytifié. Au final, nous obtenons le théorème de classicité suivant.

\begin{bigtheo}[Théorèmes $\ref{theogen}$ et $\ref{theogen_A}$]
Soit $p$ un nombre premier, et $X$ une variété de Shimura PEL de type $(A)$ ou $(C)$ de niveau Iwahorique en $p$. On suppose que sur $\mathbb{Q}_p$ l'algèbre de la donnée de Shimura est isomorphe à un produit d'algèbres de matrices. Soit $f$ une forme modulaire surconvergente (au sens fort) sur $X$ de poids $\kappa$. On suppose que $f$ est propre pour une famille $(U_i)$ d'opérateurs de Hecke en $p$, de valeurs propres $(a_i)$. Si le poids $\kappa$ est suffisamment grand devant la famille des $(v(a_i))$, alors $f$ est classique.
\end{bigtheo}

Dans le cas $(A)$, nous devons supposer l'existence du lieu ordinaire pour que le problème ait un sens. \\
Nous avons également un résultat de classicité pour les variétés de Shimura avec un niveau arbitraire en $p$. Remarquons que dans ce cas, la variété de Shimura ne possède pas de modèle entier, donc la situation est a priori plus compliquée que la précédente. Cependant, puisque nous travaillions avec l'espace $X^{an}$, notre résultat se généralise dans ce cas.

\begin{bigtheo}[Théorème $\ref{theogen_gen}$]
Soit $p$ un nombre premier, et $X$ une variété de Shimura PEL de type $(A)$ ou $(C)$ de niveau $\Gamma_1(p^n)$ en $p$. On suppose que sur $\mathbb{Q}_p$ l'algèbre de la donnée de Shimura est isomorphe à un produit d'algèbres de matrices. Soit $f$ une forme modulaire surconvergente (au sens fort) sur $X$ de poids $\kappa$. On suppose que $f$ est propre pour une famille $(U_i)$ d'opérateurs de Hecke en $p$, de valeurs propres $(a_i)$. Si le poids $\kappa$ est suffisamment grand devant la famille des $(v(a_i))$, alors $f$ est classique.
\end{bigtheo}

Dans les deux théorèmes précédents, les relations entre le poids et les pentes sont les mêmes, et sont analogues à celles des théorèmes dans le cas de bonne réduction (voir les théorèmes $\ref{theogen}$ et $\ref{theogen_A}$ pour plus de détails). \\
$ $\\
Parlons maintenant de l'organisation du texte. Nous traitons tout d'abord le cas des variétés de type $(C)$. Après avoir introduit la variété de Shimura et l'espace des formes modulaires, nous définissons les structures entières sur l'espace analytique sur lequel nous travaillons. Nous montrons ensuite comment les résultats précédents de l'auteur sur le prolongement analytique permettent de prouver un théorème de classicité. Nous traitons le cas des variétés de type $(A)$ dans la partie $\ref{partie_A}$, et le cas d'un niveau arbitraire en $p$ dans la partie $\ref{partie_gen}$. Enfin, dans l'appendice, nous rappelons certaines propriétés utiles sur les schémas semi-abéliens, et définissons les degrés partiels d'un schéma en groupe fini et plat avec une action d'un certain anneau. \\

L'auteur souhaite remercier son directeur de thèse Benoît Stroh pour ses conseils, remarques et encouragements. Il souhaite également remercier Pascal Boyer, Vincent Pilloni et Jacques Tilouine pour des discussions intéressantes. Il remercie enfin l'ANR Arshifo pour son soutien financier.

\section{Espace de modules et formes modulaires} \label{symplectique}

Nous étudions dans ce paragraphe le cas des variétés de Shimura de type (C), en autorisant le nombre premier $p$ à être ramifié dans la donnée de Shimura. La principale difficulté dans ce cas provient de l'absence de modèles entiers pour les compactifications de la variété de Shimura. En effet, l'espace de modules définit un schéma sur l'anneau des entiers d'une extension finie de $\mathbb{Q}_p$, et il est possible de construire des compactifications de cet espace après inversion de $p$. Construire des modèles entiers pour ces compactifications qui vérifierons des bonne propriétés est un exercice difficile (voir \cite{Ra} dans le cas Hilbert). Il est peut-être possible de définir les compactifications en prenant la normalisation d'un certain espace dans un autre, mais vérifier que cette compactification vérifie les bonnes propriétés serait au minimum long et pénible. \\
La principale difficulté posée par l'absence de modèle entier des compactifications est l'absence du principe de Koecher. Ainsi, si $X^{rig}$ est l'espace rigide associé à la variété de Shimura entière, une section du faisceau des formes modulaires sur $X^{rig}$ ne provient plus nécessairement d'une forme modulaire classique. Pour remédier à ce problème, nous allons utiliser le modèle rationnel de la variété de Shimura, et travailler avec l'espace analytique associé. Nous adaptons ensuite les résultats obtenus dans les parties précédentes à cet espace analytique.

\subsection{Données de Shimura}

Rappelons les données paramétrant les variétés de Shimura PEL de type (C) (voir \cite{Ko}). Soit $B$ une $\mathbb{Q}$-algèbre simple munie d'une involution positive $\star$. Soit $F$ le centre de $B$ et $F_0$ le sous-corps de $F$ fixé par $\star$. Le corps $F_0$ est une extension totalement réelle de $\mathbb{Q}$, soit $d$ son degré. Faisons les hypothèses suivantes :

\begin{itemize}
\item $F=F_0$.
\item Pour tout plongement $F \to \mathbb{R}$, $B \otimes_{F} \mathbb{R} \simeq $M$_n(\mathbb{R})$, et l'involution $\star$ est donnée par $A \to A^t$.
\end{itemize}

Soit également $(U_{\mathbb{Q}},\langle,\rangle)$ un $B$-module hermitien non dégénéré, l'accouplement étant alterné. Soit $G$ le groupe des automorphismes du $B$-module hermitien $U_{\mathbb{Q}}$ ; pour toute $\mathbb{Q}$-algèbre $R$, on a donc
$$G(R) = \left\{ (g,c) \in GL_{B} (U_{\mathbb{Q}} \otimes_{\mathbb{Q}} R) \times R^* , \langle gx,gy \rangle =c \langle x,y\rangle \text{ pour tout } x,y \in U_{\mathbb{Q}} \otimes_{\mathbb{Q}} R \right\} $$

Soient $\tau_1, \dots, \tau_d$ les plongements de $F$ dans $\mathbb{R}$,et $B_i=B \otimes_{F,\tau_i} \mathbb{R} \simeq $M$_n (\mathbb{R})$. Alors $G_\mathbb{R}$ est isomorphe à
$$\text{G} \left( \prod_{i=1}^d \text{Sp}_{2 g} \right)$$
où $g = \frac{1}{2nd} $dim$_\mathbb{Q} U_\mathbb{Q}$. \\
Donnons-nous également un morphisme de $\mathbb{R}$-algèbres $h : \mathbb{C} \to $End$_B U_\mathbb{R}$ tel que $\langle h(z)v,w\rangle ~=~\langle v,h(\overline{z})w\rangle$ et $(v,w) \to \langle v,h(i)w\rangle$ est définie positive. Ce morphisme définit donc une structure complexe sur $U_\mathbb{R}$ : soit $U^{1,0}_{\mathbb{R}}$ le sous-espace de $U_\mathbb{R}$ pour lequel $h(z)$ agit par la multiplication par $z$. \\
On a alors $U^{1,0}_{\mathbb{R}} \simeq \prod_{i=1}^d (\mathbb{R}^n)^g$ en tant que $B \otimes_{\mathbb{Q}} \mathbb{R} \simeq \oplus_{i=1}^d $M$_n( \mathbb{R})$-module. \\
Soient également un ordre $O_B$ de $B$ stable par $\star$, et un réseau $U$ de $U_\mathbb{Q}$ tel que l'accouplement $\langle,\rangle$ restreint à $U\times U$ soit à valeurs dans $\mathbb{Z}$. Nous ferons les hypothèses suivantes : 
\begin{itemize}
\item $B \otimes_{\mathbb{Q}} \mathbb{Q}_p$ est isomorphe à un produit d'algèbres de matrices à coefficients dans une extension finie de $\mathbb{Q}_p$.
\item $O_B$ est un ordre maximal en $p$.
\item L'accouplement $U \times U \to \mathbb{Z}$ est parfait en $p$.
\end{itemize}
$ $\\
L'algèbre $B$ est un $\mathbb{Q}$-espace vectoriel. Soit $\alpha_1, \dots, \alpha_t$ une base de cet espace vectoriel, et 
$$\text{det}_{U^{1,0}} = f(X_1, \dots, X_t) = \det (X_1 \alpha_1 + \dots + X_t \alpha_t ;U^{1,0}_{\mathbb{C}} \otimes_{\mathbb{C}} \mathbb{C} [X_1, \dots, X_t])$$
On montre (\cite{Ko}) que $f$ est un polynôme à coefficients algébriques. Le corps de nombres $E$ engendré par ses coefficients est appelé corps réflexe, et est égal à $\mathbb{Q}$ dans le cas $(C)$. \\
Soit $h$ le nombre d'idéaux premiers de $F$ au-dessus de $p$, que l'on notera $\pi_1, \dots, \pi_h$. Soient également $f_i$ le degré résiduel et $e_i$ l'indice de ramification de $\pi_i$. On a donc $(p) = \prod_{i=1}^h \pi_i^{e_i}$.  Alors $B \otimes_{\mathbb{Q}}~\mathbb{Q}_p \simeq~\prod_{i=1}^{h} $M$_n(F_i)$, où $F_i$ est la complétion de $F$ en $\pi_i$. Le corps $F_i$ est donc une extension finie de $\mathbb{Q}_p$, de degré $d_i :=e_i f_i$, d'indice de ramification $e_i$ et de degré résiduel $f_i$.

\subsection{Variété de Shimura} \label{defshi}

Définissons maintenant la variété de Shimura PEL de type (C) associée à $G$. Soit $K$ une extension finie de $\mathbb{Q}_p$ contenant les images de tous les plongements possibles $F \hookrightarrow \overline{\mathbb{Q}}_p$. Soit $N \geq 3$ un entier premier à $p$. 
\begin{defi}
Soit $X$ sur Spec$(K)$ l'espace de modules dont les $S$-points sont les classes d'isomorphismes des $(A,\lambda,\iota,\eta)$ où
\begin{itemize}
\item $A \to S$ est un schéma abélien
\item $\lambda : A \to A^t$ est une polarisation de degré premier à $p$.
\item $\iota : O_B \to $End $A$ est compatible avec les involutions $\star$ et de Rosati, et les polynômes $\det_{U^{1,0}}$ et $\det_{Lie (A)}$ sont égaux.
\item $\eta : A[N] \to U/NU$ est une similitude symplectique $O_B$-linéaire, qui se relève localement pour la topologie étale en une similitude symplectique $O_B$-linéaire
$$H_1 (A,\mathbb{A}_f^p) \to U \otimes_{\mathbb{Z}} \mathbb{A}_f^p$$
\end{itemize}
\end{defi}

\begin{prop}
L'espace $X$ est un schéma quasi-projectif sur $K$.
\end{prop}

De plus, il est possible de construire des compactifications toroïdales de l'espace $X$. Celles-ci sont construites par exemple dans \cite{Pin}. On rappelle que la construction de ces compactifications nécessite un choix combinatoire, que l'on supposera fait dans la suite. Rappelons ici le principal théorème de Pink quant aux compactifications toroïdales des variétés de Shimura, et la fonctorialité de ces constructions. On renvoie à \cite{Pin} pour les définitions précises et les constructions.

\begin{theo} [\cite{Pin} Théorème $12.4$] \label{morphcompact}
$ $
\begin{itemize}
\item Soit $D$ une donnée de Shimura, et $X$ la variété de Shimura associée ; c'est un schéma défini sur le corps réflexe $E$. Alors, à tout choix combinatoire suffisamment fin $\Sigma$ on peut associer une compactification toroïdale $\overline{X}^\Sigma$. C'est un schéma propre sur $E$, lisse si le choix combinatoire l'est également.
\item Si $\Sigma_1$ est un choix combinatoire plus fin que $\Sigma_2$, alors l'identité de $X$ s'étend de manière unique en une immersion ouverte $\overline{X}^{\Sigma_1} \hookrightarrow \overline{X}^{\Sigma_2}$.
\item Soient $D_1$ et $D_2$ deux données de Shimura, avec un morphisme $D_1 \to D_2$. Alors, on a une inclusion des corps réflexes $E_2 \subset E_1$, et un morphisme de variétés de Shimura\\
 $X_1 \to X_2 \times_{E_2} E_1$. Soit $\Sigma_i$ un choix combinatoire pour $X_i$. Si $\Sigma_1$ est suffisamment fin, alors le morphisme précédent s'étend en un morphisme $\overline{X_1}^{\Sigma_1} \to \overline{X_2}^{\Sigma_2} \times_{E_2} E_1$.
\end{itemize}
\end{theo}

Soit donc $\overline{X}$ une compactification toroïdale de $X$, associé à un choix combinatoire lisse. C'est un schéma propre et lisse sur $K$. On supposera ce choix fixé dans la suite, en ayant à l'esprit que l'on peut prendre ce choix combinatoire arbitrairement fin. Le schéma abélien universel $A \to X$ s'étend en un schéma semi-abélien $A \to \overline{X}$. Nous allons maintenant définir une structure de niveau Iwahorique sur $X$. Si $A \to X$ est le schéma abélien universel sur $X$, on a
$$A[p^\infty] = \oplus_{i=1}^{h} A[\pi_i^\infty]$$         
De plus, les groupes de Barsotti-Tate $A[\pi_i^\infty]$ sont principalement polarisés de dimension $n d_i g$. 

\begin{defi}
Soit $X_{Iw}$ l'espace de modules sur $K$ dont les $S$-points sont les $(A,\lambda,\iota,\eta,H_{i,j})$ où $(A,\lambda,\iota,\eta) \in X(S)$ et $0=H_{i,0} \subset H_{i,1} \subset \dots \subset H_{i,g}$ est un drapeau de sous-groupes finis et plats, stables par $O_B$, et totalement isotropes de $A[\pi_i]$, chaque $H_{i,j}$ étant de hauteur $nf_i j$, pour tout $1\leq i \leq h$.
\end{defi}

L'espace $X_{Iw}$ est un schéma quasi-projectif sur $K$, et on dispose du morphisme d'oubli $X_{Iw} \to X$. Soit également $\overline{X}_{Iw}$ une compactification toroïdale lisse de $X_{Iw}$ sur $K$. On suppose que les choix combinatoires sont faits de telle manière à ce que le morphisme $X_{Iw} \to X$ s'étend en $\overline{X}_{Iw} \to \overline{X}$ (cela est possible d'après le théorème $\ref{morphcompact}$). \\
Enfin, nous noterons $X^{an}$, $X_{Iw}^{an}$, $\overline{X}^{an}$ et $\overline{X}_{Iw}^{an}$ les espaces analytiques associés respectivement aux schémas $X$ et $X_{Iw}$, $\overline{X}$ et $\overline{X}_{Iw}$ (voir \cite{Be} par exemple).

\subsection{Formes modulaires}

Pour tout idéal premier $\pi_i$ divisant $p$, on rappelle que $F_{i}$ est la complétion de $F$ au-dessus de $\pi_i$. \\
Soit $A$ le schéma semi-abélien universel sur $\overline{X}$, et soit $\omega_A = e^* \Omega_{A/\overline{X}}^1$ le faisceau conormal relatif à la section unité de $A$ ; il est localement pour la topologie de Zariski isomorphe à $St \otimes_{\mathbb{Q}} \mathcal{O}_{\overline{X}}$ comme $B \otimes_{\mathbb{Q}} \mathcal{O}_{\overline{X}}$-module, où
$$St = \oplus_{i=1}^{h} (F_{i}^n )^g $$
Soit $\mathcal{T} = $Isom$_{B \otimes \mathcal{O}_{\overline{X}}} (St \otimes \mathcal{O}_{\overline{X}}, \omega_A)$. C'est un torseur sur $\overline{X}$ sous le groupe
$$M=\left( \prod_{i=1}^{h} Res_{F_{i} / \mathbb{Q}_p} GL_g  \right ) \times_{\mathbb{Q}_p} K$$

Soit $T_M$ le tore diagonal de $M$, $B_M$ le Borel supérieur de $M$, et $U_M$ son radical unipotent. Soit $X(T_M)$ le groupe des caractères de $T_M$, et $X(T_M)^+$ le cône des poids dominants pour $B_M$. Si $\kappa \in X(T_M)^+$, on note $\kappa'=- w_0 \kappa \in X(T_M)^+$, où $w_0$ est l'élément le plus long du groupe de Weyl de $M$ relativement à $T_M$. \\
Soit $\phi : \mathcal{T} \to \overline{X}$ le morphisme de projection.

\begin{defi}
Soit $\kappa \in X(T_M)^+$. Le faisceau des formes modulaires de poids $\kappa$ est $\omega^\kappa =~\phi_* O_\mathcal{T}[\kappa']$, où $\phi_* O_\mathcal{T}[\kappa']$ est le sous-faisceau de $\phi_* O_\mathcal{T}$ où $B_M=T_M U_M$ agit par $\kappa'$ sur $T_M$ et trivialement sur $U_M$.
\end{defi} 

Le faisceau $\omega^\kappa$ est un faisceau localement libre de rang fini sur $\overline{X}$. Une forme modulaire de poids $\kappa$ sur $\overline{X}$ est donc une section globale de $\omega^\kappa$, soit un élément de $H^0({\overline{X}} , \omega^\kappa)$. En utilisant la projection $\overline{X}_{Iw} \to {\overline{X}}$, on définit de même le faisceau $\omega^\kappa$ sur $\overline{X}_{Iw}$, ainsi que les formes modulaires sur $\overline{X}_{Iw}$. On notera encore $\omega^\kappa$ le faisceau analytifié sur $\overline{X}^{an}$ et $\overline{X}_{Iw}^{an}$. \\

\begin{rema}
Par équivalence de Morita, la catégorie des $M_n(A)$-modules et celle des $A$-modules sont équivalentes, pour tout anneau $A$. L'équivalence de catégorie est explicite : à un $A$-module $M$, on associe $M^n$, qui est bien muni d'une action de $M_n(A)$ ; réciproquement, à un $M_n(A)$-module $N$, on associe le $A$-module $E_{1,1} N$, où $E_{1,1}$ est la matrice avec un seul coefficient non nul en position $(1,1)$ égal à $1$. \\
De cette manière, puisque $B \otimes_{\mathbb{Q}} \mathbb{Q}_p = \prod_{i=1}^h M_n(F_i)$, et que le faisceau $\omega_A$ est isomorphe à $St \otimes \mathcal{O}_{\overline{X}}$ comme $B \otimes_{\mathbb{Q}} \mathcal{O}_{\overline{X}}$-module, l'équivalence de Morita associe à $\omega_A$ le faisceau de $(\prod_{i=1}^h F_i) \otimes_{\mathbb{Q}} \mathcal{O}_{\overline{X}}$-modules défini par $E \cdot \omega_A$, où $E$ est la projection défini par $(E_{1,1})_i \in \prod_{i=1}^h M_n(F_i)$. Ce faisceau est isomorphe à $(\oplus_{i=1}^{h} F_i^g) \otimes_{\mathbb{Q}} \mathcal{O}_{\overline{X}}$ comme $(\prod_{i=1}^h F_i) \otimes_{\mathbb{Q}} \mathcal{O}_{\overline{X}}$-module.
\end{rema}

\subsection{Opérateurs de Hecke} \label{heckedef}

Soit $1 \leq i \leq h$. Soit $C_i$ l'espace de modules sur $K$ paramétrant les $(A,\lambda,\iota,\eta,H_{j,k},L)$ avec $(A,\lambda,\iota,\eta,H_{j,k}) \in X_{Iw}$ et $L$ un sous-groupe fini et plat de $A[\pi_i]$, stable par $O_B$, totalement isotrope et supplémentaire de $H_{i,g}$ dans $A[\pi_i]$. Nous avons deux morphismes finis étales $p_1, p_2 : C_i \to X_{Iw}$ : $p_1$ est l'oubli de $L$, et $p_2$ est le quotient par $L$. Une ambigu\"ité subsiste toutefois pour la polarisation sur $A/L$. Soit $\xi_i$ un élément totalement positif de l'anneau des entiers de $F$, avec $v_{\pi_i} (\xi_i) = 1$ et $v_{\pi_j} (\xi_i) = 0$ si $j \neq i$. On définit la polarisation sur $A/L$ comme la polarisation descendue $\xi_i \cdot \lambda$. Il s'agit bien d'une polarisation de degré premier à $p$. Le morphisme $p_2$ dépend donc du choix d'un tel élément $\xi_i$, mais ce choix n'est pas important en pratique (voir \cite{Bi} remarque $2.3.2$). Remarquons que si $p$ est inerte dans $F$, on peut choisir $\xi_i = p$. \\
Soit $C_i^{an}$ l'espace analytique associé à $C_i$ ; on note encore $p_1$, $p_2 : C_i^{an} \to X_{Iw}^{an}$ les morphismes induits. L'opérateur de Hecke agissant sur $X_{Iw}^{an}$ est défini par $U_{\pi_i} (S) := p_2(p_1^{-1}(S))$ pour toute partie $S$ de $X_{Iw}^{an}$. Cette correspondance envoie les ouverts sur les ouverts, et les ouverts quasi-compacts sur les ouverts quasi-compacts. \\ 
Notons $q : A \to A/L$ l'isogénie universelle au-dessus de $C_i$. Celle-ci induit un isomorphisme $q^*:~\omega_{A/L} \to \omega_{A}$, et donc un morphisme $q^* (\kappa) : p_2^* \omega^\kappa \to~p_1^* \omega^\kappa$. Pour tout ouvert $\mathcal{U}$ de $X_{Iw}^{an}$, nous pouvons donc former le morphisme composé
\begin{displaymath}
\widetilde{U}_{\pi_i} :   H^0(U_{\pi_i}(\mathcal{U}),\omega^\kappa) \to H^0 ( p_1^{-1} (\mathcal{U}), p_2^* \omega^\kappa) \overset{q^*(\kappa)}{\to} H^0(p_1^{-1}(\mathcal{U}) , p_1^* \omega^\kappa) \overset{Tr_{p_1}}{\to}  H^0(\mathcal{U},\omega^\kappa)
\end{displaymath}

\begin{defi}
L'opérateur de Hecke agissant sur les formes modulaires est alors défini par $U_{\pi_i} = \frac{1}{p^{n_i}} \widetilde{U}_{\pi_i}$ avec $n_i=\frac{f_i g(g+1)}{2}$. 
\end{defi}

\begin{rema}
Le terme de normalisation $\frac{1}{p^{n_i}}$ sert à maximaliser l'intégrabilité de l'opérateur de Hecke, comme le montre un calcul sur les $q$-développements.
\end{rema}

A priori, cet opérateur n'est défini que sur l'espace $H^0(X_{Iw}^{an},\omega^\kappa)$, et non sur $H^0(\overline{X}_{Iw}^{an},\omega^\kappa)$. Etudions les problèmes au bord. \\
Il existe une compactification toroïdale $\overline{C}_i$ de $C_i$. D'après le théorème $\ref{morphcompact}$, il existe un choix combinatoire pour $C_i$ tel que le morphisme $p_1 : C_i \to X_{Iw}$ s'étend en un morphisme $\overline{C}_i \to \overline{X}_{Iw}$. Par le même argument, il existe un autre choix combinatoire pour $C_i$ tel que le morphisme $p_2$ s'étend aux compactifications pour ce choix combinatoire. Or, étant donné deux choix combinatoires on peut toujours en trouver un troisième plus fin que les deux premiers. Le théorème $\ref{morphcompact}$ montre donc qu'il est possible de construire une compactification toroïdale $\overline{C}_i$ telle que les morphisme $p_1$ et $p_2$ s'étendent en des morphismes $\overline{C}_i \to \overline{X}_{Iw}$. Si on note $\overline{C}_i^{an}$ l'espace rigide analytique associé à $\overline{C}_i$, on obtient des morphismes $p_1, p_2 : \overline{C}_i^{an} \to \overline{X}_{Iw}^{an}$. La même formule que précédemment permet de définir un opérateur de Hecke géométrique $U_{\pi_i}$ agissant sur les parties de $\overline{X}_{Iw}^{an}$. Néanmoins, les morphismes $p_1$ et $p_2$ n'étant plus finis étales, cette correspondance ne respecte plus la topologie, c'est-à-dire que l'image d'un ouvert n'est pas nécessairement encore un ouvert. Pour la même raison, il n'est pas possible de définir l'opérateur $U_{\pi_i}$ agissant sur $H^0(\overline{X}_{Iw}^{an},\omega^\kappa)$ par la même formule que précédemment. Pour pallier à ce problème, nous utilisons le théorème suivant.

\begin{theo}[\cite{Lu}] \label{extension}
Soit $Y$ un espace rigide lisse quasi-compact, et $Z$ un fermé Zariski de $Y$ de codimension supérieure ou égale à $1$. Alors toute fonction bornée sur $Y \backslash Z$ s'étend de manière unique en une fonction sur $Y$.
\end{theo}

Soit $f \in H^0(\overline{X}_{Iw}^{an}, \omega^\kappa)$. Alors $U_{\pi_i} f$ définit un élément de $H^0(X_{Iw}^{an} , \omega^\kappa)$. Comme l'espace $\overline{X}_{Iw}^{an}$ est quasi-compact, la forme $f$ est automatiquement bornée (c'est-à-dire qu'il existe un recouvrement admissible de $\overline{X}_{Iw}^{an}$ par des ouverts affinoïdes sur lesquels on a une trivialisation du faisceau $\omega^\kappa$, et la forme $f$ est bornée uniformément sur chacun de ces ouverts). Il n'est pas difficile de voir que l'opérateur $U_{\pi_i}$ est borné, donc que $U_{\pi_i} f$ est bornée (c'est-à-dire qu'il existe un recouvrement admissible de $\overline{X}_{Iw}^{an}$ par des ouverts affinoïdes sur lesquels on a une trivialisation du faisceau $\omega^\kappa$, et $U_{\pi_i} f$ est bornée uniformément sur chacun de ces ouverts intersectés avec $X_{Iw}^{an}$). On peut alors appliquer le théorème précédent, et la forme $U_{\pi_i} f $ s'étend en une section de $\omega^\kappa$ sur $\overline{X}_{Iw}^{an}$. \\
L'opérateur $U_{\pi_i}$ agit donc bien sur l'espace $H^0(\overline{X}_{Iw}^{an}, \omega^\kappa)$. \\

\begin{rema}
Nous définirons dans la suite une norme sur l'espace $H^0(\overline{X}_{Iw}^{an}, \omega^\kappa)$, et majorerons la norme des opérateurs $U_{\pi_i}$ (voir la partie \ref{norme}).
\end{rema}

\noindent Nous avons donc défini $h$ opérateurs agissant sur l'espace des formes modulaires.

\section{Structures entières}

\subsection{Fonction degré}

Nous souhaitons définir des fonctions degrés sur les espaces $X_{Iw}^{an}$ et $\overline{X}_{Iw}^{an}$. Les sous-groupes universels $H_{i,j}$ étant définis sur des extensions finies de $\mathbb{Q}_p$ (et non sur leur anneau des entiers), on ne peut appliquer directement les résultats de \cite{Fa}. Le problème principal est l'absence de modèle entier pour la compactification ; en effet, si le schéma $\overline{X}_{Iw}$ admettait un bon modèle entier propre, on pourrait définir la fonction degré sur l'espace rigide associé à ce modèle entier, qui serait égal à $\overline{X}_{Iw}^{an}$ par propreté. Pour remédier à ce problème, nous allons utiliser une autre variété de Shimura, pour laquelle les structures entières sont bien connues.

\begin{defi} \label{siegel}
Soit $1\leq i \leq h$. Soit $\mathcal{A}_{ndg,Iw_i}$ la variété de Siegel sur $\mathbb{Z}_p$ paramétrant 
\begin{itemize}
\item un schéma abélien $A$ de dimension $ndg$.
\item $\lambda : A \to A^t$ est une polarisation de degré premier à $p$.
\item une structure de niveau principale en $N$, c'est-à-dire un isomorphisme $A[N] \simeq (\mathbb{Z} / N\mathbb{Z})^{2ndg}$ qui respecte les formes symplectiques à un scalaire près.
\item un sous-groupe $H$ de $A[p]$, totalement isotrope et de hauteur $nf_i g$.
\end{itemize}
\end{defi}

On dispose d'un morphisme naturel $\mathcal{P}_i : X_{Iw} \to \mathcal{A}_{ndg,Iw_i} \times K$, défini par $\mathcal{P}_i (A,\lambda,\iota,\eta,H_{i,j}) = (A,\lambda,\eta,H_{i,g})$. On notera $\mathcal{A}_{ndg,Iw_i}^{an}$, l'espace analytique associé à $\mathcal{A}_{ndg,Iw_i} \times K$, et on note toujours par $\mathcal{P}_i$ le morphisme $X_{Iw}^{an} \to \mathcal{A}_{dg,Iw_i}^{an}$.\\
D'après \cite{St}, il existe une compactification $\overline{\mathcal{A}}_{ndg,Iw_i}$ de $\mathcal{A}_{ndg,Iw_i}$ définie sur $\mathbb{Z}_p$. Si on note $\overline{\mathcal{A}}_{ndg,Iw_i}^{rig}$ l'espace rigide associé à $\overline{\mathcal{A}}_{ndg,Iw_i} \times_{\mathbb{Z}_p} O_K$, et $\overline{\mathcal{A}}_{ndg,Iw_i}^{an}$ l'espace analytique associé à $\overline{\mathcal{A}}_{ndg,Iw_i} \times K$, alors ces deux espaces sont égaux car l'espace compactifié est propre. Le sous-groupe universel $H$ sur $\mathcal{A}_{ndg,Iw_i}$ s'étend en un groupe quasi-fini et plat à $\overline{\mathcal{A}}_{ndg,Iw_i}$. De plus, d'après le théorème $\ref{morphcompact}$, on peut supposer (quitte à raffiner la décomposition polyhédrale utilisée pour construire $\overline{X}_{Iw}$) que le morphisme $\mathcal{P}_i : X_{Iw} \to \mathcal{A}_{ndg,Iw_i} \times K$ s'étende en $\overline{X}_{Iw} \to \overline{\mathcal{A}}_{ndg,Iw_i} \times K$, et donc induise un morphisme $\overline{X}_{Iw}^{an} \to \overline{\mathcal{A}}_{ndg,Iw_i}^{an}$. \\
Nous allons définir le degré de $H$ sur $\overline{\mathcal{A}}_{ndg,Iw_i}^{rig}$. Si $x$ est un point de $\mathcal{A}_{ndg,Iw_i}^{rig}$, alors le schéma en groupes $H$ correspondant à $x$ est fini et plat sur l'anneau des entiers d'une extension finie de $\mathbb{Q}_p$. On peut donc définir son degré par \cite{Fa}. Dans le cas général, si $x$ est un $L$-point de $\overline{\mathcal{A}}_{ndg,Iw_i}^{rig}$, alors le groupe $H$ au-dessus de $x$ est un schéma en groupes quasi-fini et plat sur $O_L$, l'anneau des entiers de $L$. Le schéma semi-abélien $A$ au-dessus de $x$ est obtenue par la construction de Mumford (voir \cite{F-C} par exemple) en quotientant un schéma semi-abélien $\widetilde{G}$ sur $O_L$, globalement extension d'un tore $T$ et d'un schéma abélien $A_0$ sur $O_L$, par un réseau étale $Y$ (on se référera à \cite{St} partie $1$ pour plus de détails). Comme explicité en annexe (partie $\ref{semiab}$), on a de plus une suite exacte en fibre générique
$$0 \to (\widetilde{G} [p])_\eta \to (A[p])_\eta \to \frac{1}{p} Y / Y \to 0$$
où $\eta$ désigne la fibre générique. Le sous-groupe $(\widetilde{G} [p])_\eta$ s'étend en un schéma en groupes fini et plat sur $O_L$ (qui est $\widetilde{G} [p]$). 
Soit $\widetilde{H}$ l'intersection de $H$ avec $\widetilde{G}[p]$ ; c'est un schéma en groupes fini et plat sur $O_L$. Il s'agit en fait du plus grand sous-groupe de $H$ qui est fini et plat sur $O_L$. On peut donc définir son degré par \cite{Fa}.

\begin{defi}
On définit la fonction degré sur $\overline{\mathcal{A}}_{ndg,Iw_i}^{rig}$ par deg$(x) = \frac{1}{n}$deg $\widetilde{H}$, pour tout $x \in \overline{\mathcal{A}}_{ndg,Iw_i}^{rig}$.
\end{defi}

On a ainsi défini une fonction deg : $\overline{\mathcal{A}}_{ndg,Iw_i}^{rig} \to [0,f_i g]$. Cette fonction est définie sur $\overline{\mathcal{A}}_{ndg,Iw_i}^{rig} = \overline{\mathcal{A}}_{ndg,Iw_i}^{an}$.

\begin{defi}
On définit la fonction Deg$_i : \overline{X}_{Iw}^{an} \to [0,f_ig]$ par Deg$_i (x) = $ deg $\mathcal{P}_i(x)$. On définit également la fonction degré Deg $: \overline{X}_{Iw}^{an} \to \prod_{i=1}^h [ 0 , f_i g]$ par $x \to ($Deg$_i (x))_i$.
\end{defi}

Pour tout produit d'intervalle $I = \prod_{k=1}^h I_k$, où $I_k$ est un sous-intervalle de $[0,f_k g]$, on note $\overline{X}_{Iw,I} := $ Deg$^{-1} (I)$. Le lieu ordinaire-multiplicatif $\overline{X}_{Iw}^{mult}$ de $\overline{X}_{Iw}^{an}$ correspond au lieu où les degrés des $H_{i,g}$ sont maximaux, soit à $\overline{X}_{Iw,I}$ avec $I=\prod_{i=1}^h \{f_i g \}$.

\begin{prop}
Si $I$ est un produit d'intervalles compacts à bornes rationnelles, alors $\overline{X}_{Iw,I}$ est quasi-compact.
\end{prop}

\begin{proof}
Commençons tout d'abord par remarquer que, puisque l'espace $\overline{X}_{Iw}$ est propre sur $K$, l'espace rigide-analytique $\overline{X}_{Iw}^{an}$ est quasi-compact. Soit $q : A \to A/H$ le morphisme universel au-dessus de $\overline{\mathcal{A}}_{ndg,Iw_i}$, il est quasi-fini et plat (il est fini sur $\mathcal{A}_{ndg,Iw_i}$). Soit $\omega_A' = \det e^* \Omega_A^1$ et $\omega_{A/H}' = \det e'^* \Omega_{A/H}^1$ les déterminants des faisceaux conormaux associés à $A$ et $A/H$ en leurs sections unités $e$ et $e'$. Soit $\mathcal{L} = {\omega_{A/H}'}^{-1} \otimes \omega_A'$ ; c'est un faisceau inversible sur $\overline{\mathcal{A}}_{ndg,Iw_i}$. Le morphisme $q^* : \omega_{A/H}' \to \omega_A'$ donne une section $\delta_H \in H^0(\overline{\mathcal{A}}_{ndg,Iw_i},\mathcal{L})$. On en déduit un faisceau inversible que l'on notera toujours $\mathcal{L}$ sur $\overline{\mathcal{A}}_{ndg,Iw_i}^{rig}$ et une section $\delta_H \in H^0(\overline{\mathcal{A}}_{ndg,Iw_i}^{rig},\mathcal{L})$. Ce faisceau est naturellement muni d'une norme (voir \cite{Ka}), et on a pour tout $L$-point $x$ de $\overline{\mathcal{A}}_{ndg,Iw_i}^{rig}$ (où $L$ est une extension finie de $\mathbb{Q}_p$), $ | \delta_H (x) | = p^{-  n \text{deg } x}$. En effet, en reprenant les notations précédentes, $x$ provient d'un point $x'$ du schéma formel associé à $\overline{\mathcal{A}}_{ndg,Iw_i}$. Soit $A$ le schéma semi-abélien défini sur $O_L$ au-dessus de $x'$ (quitte à prendre une extension de $L$), $A$ est le quotient de $\widetilde{G}$, extension d'un tore par un schéma abélien sur $O_L$, par un réseau étale $Y$. On renvoie à l'appendice (partie $\ref{semiab}$) pour plus de détails. On a alors un isomorphisme $\omega_A \simeq \omega_{\widetilde{G}}$. De plus, si on note $\widetilde{H}$ l'intersection de $H$ avec $\widetilde{G}[p]$, alors $A/H$ est le quotient de $\widetilde{G} / \widetilde{H}$ par un réseau étale. On a alors un isomorphisme $\omega_{A/H} \simeq \omega_{\widetilde{G}/\widetilde{H}}$. Au-dessus de $x'$, le faisceau $\mathcal{L}$ est isomorphe à $(\det \omega_{\widetilde{G} / \widetilde{H}})^{-1} \otimes \det \omega_{\widetilde{G}}$. Comme $\widetilde{H}$ est un schéma en groupes fini et plat sur $O_L$, on a bien $| \delta_H (x) | = p^{-  n \text{deg } x}$. Cela prouve que la fonction degré, définie a priori point par point, est en fait la valuation d'une fonction analytique. \\
Soit $\mathcal{L}_i := \mathcal{P}_i^* \mathcal{L}$, et $\delta_{H_{i,g}} := \mathcal{P}_i^* \delta_H$. Alors $\delta_{H_{i,g}} \in H^0 (\overline{X}_{Iw}^{an},\mathcal{L}_i)$, et la norme définie pour $\delta_H$ sur $\overline{\mathcal{A}}_{ndg,Iw_i}^{rig}$ donne naturellement une norme pour $\delta_{H_{i,g}}$, et on a $|\delta_{H_{i,g}}(x)| = p^{-n Deg_i(x)}$ pour tout $L$-point $x$ de $\overline{X}_{Iw}^{an}$. Cela permet de conclure la proposition.
\end{proof}

\begin{defi}
L'espace des formes modulaires surconvergentes est défini par
$$H^0(\overline{X}_{Iw}^{an},\omega^\kappa)^\dagger := \text{colim}_\mathcal{V} H^0 (\mathcal{V},\omega^\kappa)$$
où la colimite est prise sur les voisinages stricts $\mathcal{V}$ de $\overline{X}_{Iw}^{mult}$ dans $\overline{X}_{Iw}^{an}$.
\end{defi}

\begin{rema}
D'après le théorème $\ref{extension}$, on a 
$$H^0(\overline{X}_{Iw}^{an},\omega^\kappa)^\dagger = \text{colim}_\mathcal{V} H^0 (\mathcal{V},\omega^\kappa)^b$$
où la colimite est prise sur les voisinages stricts du lieu ordinaire-multiplcatif dans $X_{Iw}^{an}$, et où $H^0 (\mathcal{V},\omega^\kappa)^b$ désigne les fonctions bornées sur $\mathcal{V}$ (au sens de la norme que nous définirons dans le prochain paragraphe). La définition des formes surconvergentes est donc indépendante du choix combinatoire effectué pour la compactification. 
\end{rema}

\begin{rema}
Il s'agit d'une définition forte des formes surconvergentes. En effet, l'espace $X_{Iw}$ a un modèle entier $X_{Iw,0}$ définie sur l'anneau des entiers d'une extension finie de $\mathbb{Q}_p$. Une définition faible pour les formes surconvergentes est alors une section de $\omega^\kappa$ sur un voisinage strict du lieu ordinaire-multiplicatif dans $X_{Iw,0}^{rig}$, ce dernier espace étant la fibre générique de la complétion formelle de $X_{Iw,0}$ le long de sa fibre spéciale. 
\end{rema}

Une forme modulaire surconvergente est donc définie sur un espace du type $\overline{X}_{Iw,I}$ avec \\
$I=~\prod_{i=1}^h [f_i g -~ \varepsilon, f_i g]$, pour un certain $\varepsilon >0$.

Les fonctions degré se comportent relativement bien par rapport aux opérateurs de Hecke.

\begin{prop} \label{augmente}
Soit $1 \leq i \leq h$, $x \in \overline{X}_{Iw}^{an}$ et $y \in U_{\pi_i} (x)$. Soit $x_j= Deg_j(x)$, et $y_j=Deg_j(y)$ pour $1 \leq j \leq h$. Alors
\begin{itemize}
\item $y_j=x_j$ pour $j \neq i$.
\item $y_i \geq x_i$
\end{itemize}
De plus, s'il existe $y \in U_{\pi_i}^{2e_i} (x)$ avec $Deg_i(y)=Deg_i(x)$, alors $x_i \in \frac{1}{e_i} \mathbb{Z}$.
\end{prop}

\begin{proof}
Au-dessus du point $x$, on dispose d'une variété semi-abélienne $A$ munie d'une action de $O_B$, définie sur une extension finie $M$ de $\mathbb{Q}_p$, et d'un drapeau complet $H_{j,1} \subset \dots \subset H_{j,g}$ de $A[\pi_j]$ pour tout $j$. De même, au-dessus de $y$, on a une variété semi-abélienne $A'$ et des sous-groupes $H_{j,k}'$, tel que ceux-ci sont obtenus à partir des données précédentes en quotientant par un sous-groupe $L$ de $A[\pi_i]$, qui est un supplémentaire de $H_{i,g}$. De plus, quitte à remplacer $M$ par une de ses extensions finies, il existe un schéma semi-abélien $A_0$ défini sur $O_M$, tel que $A=A_0 \otimes_{O_L} L$. Par extension des sous-objets, les sous-groupes $H_{j,k}$ et $L$ de $A[p]$ s'étendent en des sous-groupes $H_{j,k,0}$ et $L_0$ de $A_0[p]$. De même, l'action de $O_B$ se relève à $A_0$, et les sous-groupes $H_{k,l,0}$ sont dans $A_0[\pi_k]$. Enfin, il existe un schéma semi-abélien $\tilde{G}$ sur $O_M$, extension d'un tore par un schéma abélien, telle que $A_0$ soit obtenu par la construction de Mumford en quotientant $\widetilde{G}$ par un réseau étale (on renvoie encore à l'appendice pour plus de détails). On a alors une inclusion $\widetilde{G}[p] \subset A_0 [p]$ ; soit $\widetilde{H}_{j,k} = \widetilde{G}[p] \cap H_{j,k,0}$ et $\widetilde{L} = \widetilde{G}[p] \cap L_0$. Soit $H_{j,k,0}'$ l'image de $H_{j,k,0}$ dans $A_0 / L$. Comme $A_0[p] = \oplus_{k=1}^h A_0[\pi_k^{e_k}]$ et que $\widetilde{L}$ est inclus dans $A_0[\pi_i]$, si $j \neq i$, les groupes quasi-finis et plats $H_{j,g,0}$ et $H_{j,g,0}'$ sont isomorphes, et donc $y_j=x_j$. L'élément $x_i$ est égal au degré de $\widetilde{H}_{i,g}$ divisé par $n$, et comme $L$ est un supplémentaire de $H_{i,g}$ dans $A[\pi_i]$, l'élément $y_i$ est égal au degré de $\widetilde{G}[\pi_i] / \widetilde{L}$ divisé par $n$. Or par les propriétés de la fonction degré, on a
$$\text{deg } \widetilde{H}_{i,g} + \text{ deg } \widetilde{L} \leq \text{ deg } (\widetilde{H}_{i,g} + \widetilde{L})  \leq \text{ deg } \widetilde{G}[\pi_i]$$
ce qui donne $x_i \leq y_i$. \\
Pour prouver le second point, supposons qu'il existe $y \in U_{\pi_i^{2e_i}}(x)$ avec $Deg_i(y) =Deg_i(x)$. On dispose au-dessus de  $x$ d'un couple $(A,\lambda,\iota,\eta,H_{k,l})$, et le schéma semi-abélien au-dessus de $y$ est obtenu en quotientant par un sous-groupe $L$ de $A[\pi_i^{2e_i}]$. De plus, comme $A[\pi_i^\infty]$ est muni d'une action de $M_n(F_i)$, il existe un groupe de Barsotti-Tate principalement polarisé $G_i$ muni d'une action de $O_{F_i}$ tel que $A[\pi_i^\infty] = (\mathbb{Q}_p / \mathbb{Z}_p)^n \otimes_{\mathbb{Z}_p} G_i$. De même le sous-groupe $H_{i,g}$ s'écrit $(\mathbb{Z} / p\mathbb{Z})^n \otimes H_{i,g}^0$, où $H_{i,g}^0$ est un sous-groupe de $G_i[\pi_i]$. On voit donc que quitte à travailler avec $G_i$ et $H_{i,g}^0$, on peut se ramener au cas où $n=1$. \\
On note $Fil_k = L[\pi_i^k]$, pour $0 \leq k \leq 2 e_i$, et $Gr_k = Fil_k / Fil_{k-1}$ pour $1 \leq k \leq 2e_i$. De même que précédemment, il existe un schéma semi-abélien $A_0$ défini sur l'anneau des entiers d'une extension finie $M$ de $\mathbb{Q}_p$, étendant le schéma semi-abélien $A$. L'action de $O_B$ se relève à $A_0$, de même que les sous-groupes $H_{i,g}$ et $L$. On notera $H_{i,g,0}$ et $L_0$ les sous-groupes de $A_0[\pi_i^{2e_i}]$ étendant respectivement $H_{i,g}$ et $L$. De même, il existe un schéma semi-abélien $\widetilde{G}$, extension d'un tore par un schéma abélien sur $O_M$, telle que $A_0$ soit obtenue en quotientant $\widetilde{G}$ par un réseau étale à l'aide de la construction de Mumford. On note $\widetilde{H} = H_{i,g,0} \cap \widetilde{G}[p]$ ; comme on a supposé $n=1$, le degré de $\widetilde{H}$ est précisément $x_i$. De même, on note $\widetilde{Fil}_k=L_0[\pi_i^k] \cap \widetilde{G}[p^2]$ et $\widetilde{Gr}_k = \widetilde{Fil}_k / \widetilde{Fil}_{k-1}$. Notons $\widetilde{H}^{(k)} = (\widetilde{G}/ \widetilde{Fil}_{k-1})[\pi_i] / \widetilde{Gr}_k$, nous avons une chaîne de morphismes
$$\widetilde{H} \to \widetilde{H}^{(1)} \to \widetilde{H}^{(2)} \to \dots \to \widetilde{H}^{(2e_i)}$$
Dans $(\widetilde{G}/\widetilde{Fil}_k)[\pi_i]$, on a deux sous-groupes disjoints : $\widetilde{H}^{(k)}$ et $\widetilde{Gr}_{k+1}$. On a alors
$$\text{deg } \widetilde{H}^{(k)} + \text{ deg } \widetilde{Gr}_{k+1} \leq \text{ deg } (\widetilde{H}^{(k)} + \widetilde{Gr}_{k+1})  \leq \text{ deg } (\widetilde{G}/\widetilde{Fil}_k)[\pi_i]$$
d'où deg $\widetilde{H}^{(k)} \leq \widetilde{H}^{(k+1)}$ pour tout $k$. On a donc deg $\widetilde{H} \leq $ deg $\widetilde{H}^{(1)} \leq \dots \leq $ deg $\widetilde{H}^{(2e_i)} $. Or $Deg_i(x) = $ deg $\widetilde{H}$, et $Deg_i(y) = $ deg $\widetilde{H}^{(2e_i)}$, et par hypothèse $Deg_i(x) = Deg_i(y)$. On en déduit que les inégalités précédentes sont en fait des égalités, et que l'on a deg $\widetilde{H}^{(k)} = Deg_i(x)$ pour tout $0 \leq k \leq 2e_i$. \\
L'égalité entre degrés montre que l'on a $(\widetilde{G}/ \widetilde{Fil}_k )[\pi_i] = \widetilde{H}^{(k)} \oplus \widetilde{Gr}_{k+1}$. On voit que le degré de $\widetilde{Gr}_{k}$ est constant, et que deg $\widetilde{Fil}_{k} = k$ deg $\widetilde{Fil}_{1}$. En particulier pour $k$ et $l$ compris entre $0$ et $e_i$, on a deg $\widetilde{Fil}_{k+l} =$ deg $\widetilde{Fil}_k + $ deg $\widetilde{Fil}_l$ ; d'après les propriétés de la fonction degré, la suite
$$0 \to \widetilde{Fil}_k \to \widetilde{Fil}_{k+l} \overset{\pi_i^k}{\to} \widetilde{Fil}_l \to 0$$
est exacte. En appliquant cette relation pour $k=l=e_i$, on voit que $\widetilde{L}$ est un Barsotti-Tate tronqué d'échelon $2$. En particulier, son degré et celui de $\widetilde{Fil}_{e_i}$ sont entiers. La proposition découle de la relation $Deg_i(x) = f_i g - \frac{1}{e_i} $ deg $\widetilde{Fil}_{e_i}$.
\end{proof}

L'opérateur de Hecke $U_{\pi_i}$ augmente donc la $i$-ième fonction degré, et ne modifie pas les autres. De plus, il augmente strictement la fonction $Deg_i$, sauf éventuellement aux points où cette fonction est un multiple de $1/e_i$. Nous avons même un résultat plus fort.

\begin{prop} 
Soit $1 \leq i \leq h$, $k$ un entier compris entre $0$ et $f_i e_ig-1$ et $0 < \alpha < \beta <1$ deux rationnels. Alors il existe $\varepsilon > 0$ tel que $Deg_i (y) \geq Deg_i(x) + \varepsilon$, pour tout $x \in Deg_i^{-1} ([\frac{k+\alpha}{e_i} , \frac{k+ \beta}{e_i} ])$ et $y \in U_{\pi_i}^{2e_i} (x)$.
\end{prop}

\begin{proof}
Définissons $C_i$ comme l'espace de modules sur $K$ paramétrant les $(x,L)$ avec $x=(A,\lambda,\iota,\eta,H_{j,k}) \in X_{Iw}$ et $L$ un sous-groupe fini et plat de $A[\pi_i^{2e_i}]$, stable par $O_B$, totalement isotrope et supplémentaire de $H_{i,g}$ dans $A[\pi_i^{2e_i}]$. Notons $\overline{C_i}$ une compactification de $C_i$ compatible avec $\overline{X}_{Iw}^{an}$, et $\overline{C_i}^{an}$ l'espace analytique associé. On dispose d'un morphisme d'oubli $p_1 : \overline{C_i}^{an} \to \overline{X}_{Iw}^{an}$. \\
On veut définir les degrés de $H_{i,g}$ et $ H_{i,g}' := $Im $(H_{i,g} \to A/L)$ sur $\overline{C_i}^{an}$ comme valuations d'une fonction analytique. Il nous faut pour cela utiliser encore l'espace $\mathcal{A}_{ndg,Iw_i}$. On définit deux morphismes $f_1,f_2 : C_i \to \mathcal{A}_{ndg,Iw_i} \times K$ par $f_1(A,\lambda,\iota,\eta,H_{j,k},L) = (A,\lambda,\eta,H_{i,g})$ et $f_2(A,\lambda,\iota,\eta,H_{j,k},L) = (A/L,\lambda',\eta',H_{i,g}')$. On peut supposer que ces morphismes s'étendent aux compactifications et induisent des morphismes $\overline{C_i}^{an} \to \overline{\mathcal{A}}_{ndg,Iw_i}^{an}$. On démontre comme précédemment qu'il existe des faisceaux inversibles $\mathcal{L}_{H_{i,g}}$, $\mathcal{L}_{H_{i,g}'}$ sur $\overline{C_i}^{an}$, munis d'une norme canonique, et des sections $\delta_{H_{i,g}} \in H^0(\overline{C_i}^{an} ,\mathcal{L}_{H_{i,g}})$,  $\delta_{H_{i,g}'} \in H^0(\overline{C_i}^{an} ,\mathcal{L}_{H_{i,g}'})$, tels que les degrés de $H_{i,g}$ et $H_{i,g}'$ sont égaux (à un facteur près) à la valuation de la norme de ces sections. \\
 D'après la proposition précédente, le degré de $H_{i,g}'$ est strictement supérieur à celui de $H_{i,g}$ sur $p_1^{-1} (Deg_i^{-1} ([\frac{k+\alpha}{e_i} , \frac{k+ \beta}{e_i} ]))$. Le principe du maximum montre qu'il existe $\varepsilon >0$ tel que deg $H_{i,g}' \geq$~deg~$H_{i,g} +~\varepsilon$ sur ce dernier espace car il est quasi-compact.
\end{proof}

\begin{coro} \label{dyna}
Plaçons-nous sous les hypothèses de la proposition précédente. Alors il existe un entier $N$ tel que
$$U_{\pi_i}^N \left( Deg_i^{-1} ( [\frac{k+\alpha}{e_i},f_i g] ) \right) \subset Deg_i^{-1} ( [\frac{k+\beta}{e_i}, f_i g ] )$$
\end{coro}

\begin{proof}
Supposons par l'absurde que ce ne soit pas le cas. Alors pour tout entier $n$, il existe $x_n$ avec $Deg_i(x_n) \in [\frac{k+\alpha}{e_i},f_i g]$, et $y_n \in U_{\pi_i}^n (x_n)$ avec $Deg_i (y_n) \leq \frac{k+\beta}{e_i}$. Comme l'opérateur $U_{\pi_i}$ augmente la fonction $Deg_i$, on a $\frac{k+\alpha}{e_i} \leq Deg_i(x_n) \leq \frac{k+\beta}{e_i}$. Or d'après la proposition précédente, il existe $\varepsilon >0$ tel que l'opérateur $U_{\pi_i}^{2e_i}$ augmente la fonction $Deg_i$ d'au moins $\varepsilon$ sur $Deg_i^{-1} ([\frac{k+\alpha}{e_i} , \frac{k+ \beta}{e_i} ])$. On en déduit que $Deg_i(y_{2e_i n}) \geq n \varepsilon + Deg_i(x_{2e_i n}) \geq n \varepsilon$ ce qui est impossible.
\end{proof}

\subsection{Normes} \label{norme}

Nous souhaitons définir une norme sur l'espace des formes modulaires définies sur un ouvert $\mathcal{U}$ quasi-compact de $\overline{X}_{Iw}^{an}$, c'est-à-dire sur l'espace $H^0(\mathcal{U},\omega^\kappa)$. Comme l'espace $\overline{X}_{Iw}^{an}$ ne provient pas (canoniquement) d'un schéma formel défini sur l'anneau des entiers d'une extension de $\mathbb{Q}_p$, on ne peut appliquer directement \cite{Ka}. Bien sûr, il est possible de définir de manière non canonique une norme sur l'espace des sections d'un faisceau localement libre sur un espace rigide, mais il sera difficile de prouver certaines propriétés. (Si $Y = $ Spm $A$ est un espace affinoïde, et $f \in A$, alors la norme de $f(y)$ est définie canoniquement pour $y \in Y$. En revanche, si $\mathcal{F}$ est un faisceau inversible, on peut définir une norme sur $H^0(Y,\mathcal{F})$ qui dépendra de la trivialisation de $\mathcal{F}$.) \\
Soit $\mathcal{A}_{ndg}$ le schéma sur $\mathbb{Z}_p$ paramétrant les schémas abéliens de dimension $ndg$, avec une polarisation de degré premier à $p$, et une structure de niveau $N$. Soit également $\overline{\mathcal{A}}_{ndg}$ une compactification toroïdale de $\mathcal{A}_{ndg}$ (construite dans \cite{F-C}), avec un choix combinatoire compatible avec celui de $\overline{X}_{Iw}$. On notera $A$ le schéma semi-abélien universel sur $\overline{\mathcal{A}}_{ndg}$. Pour définir le schéma suivant, nous nous inspirons de \cite{SaHilbert}.

\begin{defi}
Soit $\widetilde{\mathcal{A}}_{ndg}$ l'espace de modules sur $\mathbb{Z}_p$ dont les $S$-points sont :
\begin{itemize}
\item un point $x \in \overline{\mathcal{A}}_{ndg}(S)$.
\item une filtration $0 = \omega_{A,0} \subset \omega_{A,1} \subset \dots \subset \omega_{A,nd} = \omega_A$ telle que pour tout $1 \leq i \leq nd$, $\omega_{A,i} / \omega_{A,i-1}$ est localement un $\mathcal{O}_S$-facteur direct de $\omega_A / \omega_{A,i-1}$ de rang $g$.
\end{itemize}
\end{defi}

L'espace $\widetilde{\mathcal{A}}_{ndg}$ est donc un schéma sur $\mathbb{Z}_p$, et est égal à la fibration de $\overline{\mathcal{A}}_{ndg}$ par une grassmanienne. Comme $\overline{\mathcal{A}}_{ndg}$ est propre sur $\mathbb{Z}_p$, $\widetilde{\mathcal{A}}_{ndg}$ l'est également. \\
Soit $\mathcal{T}_i = Isom_{\mathcal{O}_{\widetilde{\mathcal{A}}_{ndg}}} (\omega_{A,i} / \omega_{A,i-1} , \mathcal{O}_{\widetilde{\mathcal{A}}_{ndg}}^g)$ pour $1 \leq i \leq d$. On note $\phi_i : \mathcal{T}_i \to \widetilde{\mathcal{A}}_{ndg}$ la projection. L'espace $\mathcal{T}_i$ est un torseur sur $\widetilde{\mathcal{A}}_{ndg}$ pour le groupe $GL_g$. Si $\kappa_i=(k_j)$ est un élément de $\mathbb{Z}^g$, on note $\omega_i^{\kappa_i} = {\phi_i}_* \mathcal{O}_{\mathcal{T}_i} [- \kappa_i']$, où $\kappa'=(k_{g+1-j})$, et où ${\phi_i}_* \mathcal{O}_{\mathcal{T}_i} [- \kappa_i']$ est le sous-faisceau de ${\phi_i}_* \mathcal{O}_{\mathcal{T}_i}$ où le tore de $GL_g$ agit par $-\kappa_i'$, et où le radical unipotent agit trivialement. \\
Rappelons que nous avons défini le poids d'une forme modulaire comme un couple $(k_{i,\sigma})_{1\leq i \leq g, \sigma \in \Sigma}$, où $\Sigma$ est l'ensemble des plongements de $F$ dans $\overline{\mathbb{Q}}$, vérifiant $k_{1,\sigma} \geq \dots \geq k_{g,\sigma}$, pour tout $\sigma \in \Sigma$. De plus, $\Sigma$ est l'union disjointe des $\Sigma_i$, où $\Sigma_i$ est l'ensemble des plongements qui se factorisent par $F_i$. Chaque $\Sigma_i$ est de cardinal $e_i f_i$. On fixe une numérotation sur chaque $\Sigma_i$, c'est-à-dire une bijection entre $\Sigma_i$ et l'ensemble $\{1, \dots, e_i f_i \}$. Ces choix donnent une bijection entre $\Sigma$ et $\{1, \dots, d \}$. Un poids est donc un couple $(\kappa_i)_{1 \leq i \leq d}$, où chaque $\kappa_i$ est un élément dominant de $\mathbb{Z}^g$. On note alors $\omega_0^\kappa$ le faisceau défini sur $\widetilde{\mathcal{A}}_{ndg}$ par $\omega_0^\kappa := \otimes_{i=1}^d \omega_i^{\kappa_i}$. \\
$ $\\
Soit $\widetilde{\mathcal{A}}_{ndg}^{rig}$ l'espace rigide associé à $\widetilde{\mathcal{A}}_{ndg} \times_{\mathbb{Z}_p} O_K$. Puisque ce dernier est propre sur $\mathbb{Z}_p$, on a $\widetilde{\mathcal{A}}_{ndg}^{rig} = \widetilde{\mathcal{A}}_{ndg}^{an}$. On notera encore $\omega_0^\kappa$ le faisceau induit sur cet espace. D'après \cite{Ka}, on peut définir canoniquement une norme sur l'espace $H^0(\mathcal{U},\omega_0^\kappa)$, pour tout ouvert quasi-compact $\mathcal{U}$ de $\widetilde{\mathcal{A}}_{ndg}^{rig}$. On notera $\widetilde{\omega}_0^\kappa$ le sous-faisceau des fonctions de norme plus petite que $1$. \\
$ $\\
Le faisceau $\omega_A$ défini sur $\overline{X}_{Iw}$ est muni d'une action de $B \otimes_{\mathbb{Q}} \mathbb{Q}_p = \prod_{i=1}^h $M$_n(F_i)$. Par équivalence de Morita, la catégorie des $\prod_{i=1}^h $M$_n(F_i)$-modules et celle des $\prod_{i=1}^h F_i$-modules sont équivalentes. L'équivalence de catégorie est simplement donnée par $R \to E \cdot R$. On rappelle que $E = \prod_{i=1}^h E_{1,1}$ où $E_{1,1}$ est la matrice dont tous les coefficients sont nuls sauf celui en position $(1,1)$. Soit donc $\omega_{A,d} = E \cdot \omega_{A}$. C'est un $\prod_{i=1}^h F_i$-module, et on a
$$\omega_A = \bigoplus_{j=1}^n  ( \prod_{i=1}^h E_{j,1}) \omega_{A,d}$$
Le faisceau $\omega_A$ est donc isomorphe à $n$ copies de $\omega_{A,d}$. De plus, ce dernier faisceau est localement libre de rang $dg$, et muni d'une action de $\prod_{i=1}^h F_i$. On peut donc écrire $\omega_{A,d} = \oplus_{i=1}^h \omega_{A,d,i}$, où $\omega_{A,d,i}$ est un faisceau localement libre de rang $e_i f_i g$ muni d'une action de $F_i$. On peut alors décomposer ce faisceau suivant les plongements de $F_i$ dans $\overline{\mathbb{Q}}_p$ (c'est-à-dire suivant les éléments de $\Sigma_i$): $\omega_{A,d,i} = \oplus_\sigma \omega_{A,d,i}^{(\sigma)}$, où $\sigma$ parcourt $\Sigma_i$, et où $\omega_{A,d,i}^{(\sigma)}$ est un faisceau localement libre de rang $g$. \\
Ainsi, en utilisant la bijection entre $\Sigma$ et $\{1, \dots, d\}$ fixée précédemment, on peut décomposer le faisceau $\omega_{A,d}$ en
$$\omega_{A,d} =  \oplus_{j=1}^d \omega_{A,d}^{(j)}   $$
où les faisceaux $\omega_{A,d}^{(j)}$ sont localement libres de rang $g$. Cela permet donc d'écrire $\omega_A$ comme somme directe de $nd$ faisceaux localement libre de rang $g$.

\begin{defi}
On définit un morphisme $\psi : \overline{X}_{Iw} \to \widetilde{\mathcal{A}}_{ndg} \times K$ par la formule $x \to (\mathcal{P}(x), (\omega_{A,\bullet}))$, où $\mathcal{P}$ est le morphisme d'oubli de l'action de $O_B$ et de la structure Iwahorique, et où la filtration $(\omega_{A,\bullet})$ de $\omega_A$ est déduite de ce qui précède.
\end{defi}

Nous avons défini un faisceau $\omega^\kappa$ sur $\overline{X}_{Iw}$ et un faisceau $\omega_0^\kappa$ sur $\widetilde{\mathcal{A}}_{ndg}$. 

\begin{prop}
On a $\psi^* \omega_0^\kappa = \omega^\kappa$.
\end{prop}

\begin{proof}
Le faisceau $\omega^\kappa$ est défini à l'aide du torseur $\mathcal{T} = $Isom$_{B \otimes \mathcal{O}_{\overline{X}_{Iw}}} (St \otimes \mathcal{O}_{\overline{X}_{Iw}}, \omega_A)$ sur $\overline{X}_{Iw}$, où on rappelle que $St = \oplus_{i=1}^h (F_i^n)^g$. On a donc, par l'équivalence de Morita, et la décomposition $\omega_{A,d} = \oplus_{i=1}^h \omega_{A,d,i}$,
$$\mathcal{T} = \prod_{i=1}^h \text{Isom}_{F_i \otimes \mathcal{O}_{\overline{X}_{Iw}}} (F_i^g \otimes \mathcal{O}_{\overline{X}_{Iw}}, \omega_{A,d,i}) \simeq
\prod_{j=1}^d \text{Isom}_{\mathcal{O}_{\overline{X}_{Iw}}} (\mathcal{O}_{\overline{X}_{Iw}}^g, \omega_{A,d}^{(j)}) $$
Le résultat en découle.
\end{proof}

On notera encore $\psi : \overline{X}_{Iw}^{an} \to \widetilde{\mathcal{A}}_{ndg}^{rig}$ le morphisme obtenu au niveau des espaces analytiques.
On peut donc définir une semi-norme sur l'espace $H^0(\mathcal{U},\omega^\kappa)$, pour tout ouvert $\mathcal{U}$ quasi-compact de $\overline{X}_{Iw}^{an}$. Soit $f \in H^0(\mathcal{U},\omega^\kappa)$, et $x \in \mathcal{U}$. Si on note $L$ le corps résiduel de $x$, on a donc un morphisme $x : $ Spec~$L \to \mathcal{U}$. Alors
$$x^*f \in H^0( \text{Spec } L,x^* \omega^\kappa) = H^0 (\text{Spec } L, x^* \psi^* \omega_0^\kappa) = H^0 (\text{Spec } L, (\psi x)^* \omega_0^\kappa) $$ 
Le morphisme $\psi x : $ Spec $L \to \widetilde{\mathcal{A}}_{ndg}^{rig}$ donne un $L$-point de $\widetilde{\mathcal{A}}_{ndg}^{rig}$, qui provient d'un unique $O_L$-point du schéma formel $\widetilde{\mathcal{A}}_{ndg}^{for}$ associé à $\widetilde{\mathcal{A}}_{ndg}$. On note $\psi_0$ le morphisme Spf $O_L \to \widetilde{\mathcal{A}}_{ndg}^{for}$ correspondant. On a alors
$$H^0 (\text{Spec } L, (\psi x)^* \omega_0^\kappa) = H^0 (\text{Spf } O_L, \psi_0^* \omega_0^\kappa) \otimes_{O_L} L$$
où $O_L$ est l'anneau des entiers de $L$ (on note encore $\omega_0^\kappa$ le faisceau induit sur $\widetilde{\mathcal{A}}_{ndg}^{for}$). On définit donc une norme sur $H^0( $Spec $L,x^* \omega^\kappa)$ en identifiant $H^0( $Spf $O_L, \psi_0^* \omega_0^\kappa)$ et les éléments de norme plus petite que $1$.

\begin{defi}
Soit $\mathcal{U}$ un ouvert de $\overline{X}_{Iw}^{an}$, $f \in H^0(\mathcal{U},\omega^\kappa)$, et $x \in \mathcal{U}$. On définit la norme de $f$ en $x$ par $|f(x)| := |x^* f|$, et la norme de $f$ sur $\mathcal{U}$ par $|f|_\mathcal{U} := \sup_{x \in \mathcal{U}} |f(x)|$.
\end{defi}

\begin{rema}
L'élément $|f|_\mathcal{U}$ peut éventuellement être infini, mais est fini si $\mathcal{U}$ est quasi-compact. Dans ce cas, cette définition donne en général une semi-norme, mais est une norme si l'espace est réduit.
\end{rema}

\begin{defi}
On notera encore $\widetilde{\omega}^\kappa$ le sous-faisceau des fonctions de norme plus petite que $1$.
\end{defi}

Donnons une autre définition du faisceau $\widetilde{\omega}^\kappa$. Pour tout espace rigide $Y$, on note $\widetilde{\mathcal{O}_Y}$ le faisceau des fonctions de norme inférieure ou égale à $1$. Alors
$$\widetilde{\omega}^\kappa = \psi^{-1} \widetilde{\omega}_0^\kappa \otimes_{\psi^{-1} \widetilde{\mathcal{O}_{\widetilde{\mathcal{A}}_{ndg}^{rig}}}} \widetilde{\mathcal{O}_{\overline{X}_{Iw}^{an}}}$$
Nous dirons que $\widetilde{\omega}^\kappa$ définit une structure entière pour $\omega^\kappa$. Si $\mathcal{F}$ est un faisceau localement libre de rang $r$ sur un espace rigide $Y$, on appelle structure entière pour $\mathcal{F}$ un sous-faisceau $\widetilde{\mathcal{F}}$, tel que $\widetilde{\mathcal{F}}$ soit localement isomorphe à $\widetilde{\mathcal{O}_Y}^r$. \\

Rappelons un \og gluing lemma \fg $ $ dû à Kassaei (\cite{Ka}). On rappelle que nous avons fait les choix combinatoires de telle sorte que l'espace $\overline{X}_{Iw}$ est lisse.

\begin{lemm} \label{glue}
Soit $\mathcal{U}$ un ouvert quasi-compact de $\overline{X}_{Iw}^{an}$. On a :
$$H^0(\mathcal{U},\omega^\kappa) \simeq H^0(\mathcal{U},\widetilde{\omega}^\kappa) \otimes_{\mathbb{Z}_p} \mathbb{Q}_p \simeq \left( \underset{\leftarrow}{\lim} \text{ } H^0(\mathcal{U},\widetilde{\omega}^\kappa / p^n) \right) \otimes_{\mathbb{Z}_p} \mathbb{Q}_p$$
\end{lemm}

\section{Décomposition des opérateurs de Hecke}

\subsection{Décomposition} \label{decompo}

Soit $\mathcal{U}$ un ouvert quasi-compact de $X_{Iw}^{an}$. Fixons un élément $i$ compris entre $1$ et $g$, et un élément rationnel $r \in [0,f_ig]$. On note $X_{i,\leq r} := \{ x \in X_{Iw}^{an}, Deg_i (x) \leq r \}$. Nous voulons découper notre ouvert $\mathcal{U}$ suivant le nombre de points de $U_{\pi_i}(x) \cap X_{i,\leq r}$. Pour tout $x=(A,\lambda,\iota,\eta,H_{i,j}) \in X_{Iw}^{an}$, soit $N(x,r)$ le nombre de points de $U_{\pi_i}(x) \cap X_{i,\leq r}$. Définissons
$$\mathcal{U}_j := \{ x \in \mathcal{U}, N(x,r) \geq j \}$$

\begin{prop}
Les $(\mathcal{U}_j)$ forment une suite décroissante d'ouverts quasi-compacts, vide à partir d'un certain rang.
\end{prop}

\begin{proof}
Voir \cite{Bi} lemme $4.3.6$.
\end{proof}

Sur $\mathcal{U}_j \backslash \mathcal{U}_{j+1}$, on a $N(x,r)=j$. On peut alors décomposer l'opérateur $U_{\pi_i}$ en $U_{\pi_i}^{good} \coprod U_{\pi_i}^{bad}$, où $U_{\pi_i}^{bad}$ correspond aux $j$ points de $X_{i,\leq r}$, et $U_{\pi_i}^{good}$ aux autres. Remarquons que $U_{\pi_i}^{bad}$ paramètre les supplémentaires $L$ de $H_i$ avec deg $L \geq f_i -r$. De plus, il est possible de faire surconverger ces ouverts.

\begin{prop}
Soit $r'>r$ un nombre rationnel, et $\mathcal{U}_j' := \{ x \in \mathcal{U}, N(x,r') \geq~j \}$. Alors $\mathcal{U}_j'$ est un voisinage strict de $\mathcal{U}_j$ dans $\mathcal{U}$, c'est-à-dire que le recouvrement $(\mathcal{U}_j',\mathcal{U} \backslash \mathcal{U}_j)$ de $\mathcal{U}$ est admissible.
\end{prop}

\begin{proof}
Voir \cite{Bi} proposition $4.3.10$.
\end{proof}

Pour $r'>r$, on dispose donc de la décomposition de $U_{\pi_i}$ sur $\mathcal{U}_j \backslash \mathcal{U}_{j+1}$, ainsi que sur $\mathcal{U}_j' \backslash \mathcal{U}_{j+1}'$. Ces décompositions coïncident sur l'intersection des deux ensembles. \\
Il est possible de généraliser cette décomposition à $U_{\pi_i}^N$ pour tout entier $N$.

\begin{theo} \label{bigtheo}
Soit $N \geq 1$ et $r \in [0,f_ig]$ un rationnel. Il existe un ensemble fini totalement ordonné $S_N$ et une suite décroissante d'ouverts quasi-compacts $(\mathcal{U}_j (N))_{i \in S_N}$ de $\mathcal{U}$ de longueur $L=L(N)$ indépendante de $\mathcal{U}$, tels que pour tout $j\geq 0$, on peut décomposer la correspondance $U_{\pi_i}^N$ sur $\mathcal{U}_{j}(N) \backslash \mathcal{U}_{j+1}(N)$ en 
$$ U_{\pi_i}^N = \left ( \coprod_{k=0}^{N-1} U_{\pi_i}^{N-1-k} \circ T_k  \right ) \coprod T_N$$
avec $T_0 = U_{\pi_i,j,N}^{good}$, pour $0 < k < N$
$$T_k = \coprod_{j_1 \in S_{N-1}, \dots, j_k \in S_{N-k}}  U_{\pi_i,j_k,N}^{good} U_{\pi_i,j_{k-1},j_k,N}^{bad} \dots U_{\pi_i,j,j_1,N}^{bad}$$
et
$$T_N = \coprod_{j_1 \in S_{N-1}, \dots, j_{N-1} \in S_1} U_{\pi_i,j_{N-1},N}^{bad} U_{\pi_i,j_{N-2},j_{N-1},N}^{bad} \dots U_{\pi_i,j,j_1,N}^{bad}$$  
avec
\begin{itemize}
\item les images des opérateurs $U_{\pi_i,j,N}^{good}$ ($j \in S_k$) sont incluses dans $X_{i,> r}=\{ x \in~X_{rig}, Deg_i (x) >~r \}$.
\item les opérateurs $U_{\pi_i,j,l,N}^{bad}$ ($j \in S_k$, $l \in S_{k-1})$ et $U_{\pi_i,j,N}^{bad}$ ($j \in S_1$) sont incluses dans $X_{i,\leq r}$.
\end{itemize}
Enfin, si $(\mathcal{U}_j'(N))$ est la suite d'ouverts de $\mathcal{U}$ obtenue pour $r'>r$, alors $\mathcal{U}_j'(N)$ est un voisinage strict de $\mathcal{U}_j(N)$ dans $\mathcal{U}$ pour tout $j$.
\end{theo}

\begin{proof}
C'est le théorème $4.4.1$ de \cite{Bi}.
\end{proof}

\subsection{Norme des opérateurs de Hecke}

Pour démontrer le théorème de classicité, nous aurons besoin d'un calcul de normes de ces opérateurs de Hecke. Rappelons que la norme d'un opérateur $T : H^0(T(\mathcal{U}),\mathcal{F}) \to H^0(\mathcal{U},\mathcal{F})$ est défini par
$$ \left\| T \right\|_{\mathcal{U}} := \inf \left\{ \lambda \in \mathbb{R}_{>0}, \text{ } |Tf|_\mathcal{U} \leq \lambda |f|_{T(\mathcal{U})} \text{ } \forall f \in H^0(T(\mathcal{U}),\mathcal{F}) \right\}$$

\begin{prop} \label{lemnorm}
Soit $T$ un opérateur défini sur un ouvert $\mathcal{U}$, égal à $U_{\pi_i}$, $U_{\pi_i}^{good}$ ou $U_{\pi_i}^{bad}$. On suppose que l'image de cet opérateur est incluse dans $X_{i,\leq f_i g - c}$ pour un certain $c \geq 0$. Alors
$$ \Vert T \Vert_\mathcal{U} \leq p^{f_i g(g+1)/2-c \inf_{\tau \in \Sigma_i} k_{g,\tau}}$$
\end{prop}

\begin{proof}
Avec les notations de $\ref{heckedef}$, nous allons majorer la norme du morphisme $q^*(\kappa) : p_2^* \omega^\kappa \to p_1^*\omega^\kappa$, chacun de ces deux faisceaux étant muni de la structure entière induite par celle de $\omega^\kappa$ via les morphismes $p_1$ et $p_2$ respectivement. \\
Soit $x=(A,i,\phi,H,\omega_{A,\sigma,j}) \in X_{Iw}^{an}(\overline{\mathbb{Q}}_p) $ et $L \subset A[\pi_i]$ un supplémentaire de $H[\pi_i]$ stable par $O_B$. Alors $\psi(x) \in \widetilde{\mathcal{A}}_{ndg}^{rig}(\overline{\mathbb{Q}}_p)$, et on a une variété semi-abélienne $A_0$ définie sur $\overline{\mathbb{Z}_p}$ au-dessus de $\psi(x)$, qui étend la variété abélienne $A$. L'action de $O_B$ s'étend à $A_0$, et le sous-groupe $L$ s'étend en un sous-groupe $L_0$ de $A_0[\pi_i]$. De même que précédemment, il existe un schéma semi-abélien $\widetilde{G}$ sur $\overline{\mathbb{Z}_p}$, globalement extension d'un tore par un schéma abélien, tel que $A_0$ soit le quotient de $\widetilde{G}$ par un réseau étale. On se référera à l'annexe (partie $\ref{semiab}$) pour plus de détails. On a $\widetilde{G}[p] \subset A_0[p]$ ; soit $\widetilde{L} = L_0 \cap \widetilde{G} [p]$. C'est un schéma en groupes fini et plat sur $\overline{\mathbb{Z}_p}$. On a alors des isomorphismes $\omega_{A_0} \simeq \omega_{\widetilde{G}}$ et $\omega_{A_0/L_0} \simeq \omega_{\widetilde{G} / \widetilde{L}}$. \\
Soit $\kappa_1$ le poids défini par $(k_{g,\sigma}, \dots, k_{g,\sigma})_{\sigma \in \Sigma}$, et $\kappa_2 = \kappa - \kappa_1$. On a alors $\omega_{\widetilde{G}}^\kappa = \omega_{\widetilde{G}}^{\kappa_1} \otimes \omega_{\widetilde{G}}^{\kappa_2}$. Puisque les coefficients de $\kappa_2$ sont positifs, le morphisme $\omega_{\widetilde{G}/\widetilde{L}}^{\kappa_2} \to \omega_{\widetilde{G}}^{\kappa_2}$ a une norme inférieure ou égale \`a $1$. Il nous reste \`a calculer la norme du morphisme $\omega_{\widetilde{G}/\widetilde{L}}^{\kappa_1} \to \omega_{\widetilde{G}}^{\kappa_1}$. \\
Le morphisme $q : A_0 \to A_0/L_0$ donne une suite exacte de $\overline{\mathbb{Z}_p} \otimes_{\mathbb{Z}} O_B$-modules
$$0 \to \omega_{A_0/L_0} \to \omega_{A_0} \to \omega_{L_0} \to 0$$
qui s'identifie à
$$0 \to \omega_{\widetilde{G}/\widetilde{L}} \to \omega_{\widetilde{G}} \to \omega_{\widetilde{L}} \to 0$$
En utilisant l'équivalence de Morita, on en déduit une suite exacte de $(\prod_{i=1}^h O_{F_i} ) \otimes_{\mathbb{Z}_p} \overline{\mathbb{Z}_p}$-modules
$$0 \to E \cdot \omega_{\widetilde{G}/\widetilde{L}} \to E \cdot \omega_{\widetilde{G}} \to E \cdot \omega_{\widetilde{L}} \to 0$$
De plus, on sait que ces modules admettent une filtration indexée par les éléments de $\Sigma$, et la suite exacte respecte cette filtration. Remarquons que puisque l'on travaille sur $\overline{\mathbb{Z}_p}$, cette filtration est canonique, et est déduite de la décomposition en somme directe de ces modules après inversion de $p$. On en déduit que les morphismes de la suite exacte sont stricts pour la filtration. Pour tout $1 \leq j \leq d$, on obtient donc une suite exacte
$$0 \to \omega_{\widetilde{G}/\widetilde{L},j} / \omega_{\widetilde{G}/\widetilde{L},j-1} \overset{f_j}{\to} \omega_{\widetilde{G},j} / \omega_{\widetilde{G},j-1} \to \omega_{\widetilde{L},j} / \omega_{\widetilde{L},j-1} \to 0$$ 
où $(\omega_{\widetilde{G},j})_{1 \leq j \leq d}$ est la filtration de $E \cdot \omega_{\widetilde{G}}$, et similairement pour $E \cdot \omega_{\widetilde{G}/\widetilde{L}}$ et $E \cdot \omega_{\widetilde{L}}$. On rappelle qu'on a ordonné les éléments de $\Sigma = \{\sigma_1, \dots, \sigma_d \}$. De plus, on a 
$$\omega_{\widetilde{G}}^{\kappa_1} = \bigotimes_{j=1}^d (\det \omega_{\widetilde{G},j} / \omega_{\widetilde{G},j-1})^{k_{g,\sigma_j}}$$
Si $\sigma_j \notin \Sigma_i$, alors $f_j$ est un isomorphisme puisque $L \subset A[\pi_i]$ donc $\omega_{\widetilde{L},j} / \omega_{\widetilde{L},j-1} = 0$. Soit $\lambda_j = v(\det f_j)$ ; alors
$$\Vert q^*(\kappa_1) \Vert_x = \prod_{j, \sigma_j \in \Sigma_i} p^{ - \lambda_j k_{g,\sigma_j } } \leq  p^{ - \inf_{\tau \in \Sigma_i} k_{g,\tau} \sum_{j,\sigma_j \in \Sigma_i} \lambda_{j} }$$
La proposition découle alors du fait que deg $\widetilde{L} = \sum_{j, \sigma_j \in \Sigma_i} \lambda_j \geq c$.
\end{proof}

\section{Classicité}

Un élément $\kappa \in~X(T_M)^+$ est une famille d'entiers
$$\prod_{i=1}^{h}  \prod_{j=1}^{d_i} (k_{1,j,i} \geq \dots \geq k_{g,j,i} ) $$
Le théorème suivant dit qu'une forme surconvergente propre pour les opérateurs de Hecke de poids suffisamment grand est classique. On rappelle que $d_i = e_i f_i$.

\begin{theo} \label{theogen}
Soit $f$ une forme surconvergente de poids $\kappa \in X(T_M)^+$ sur $X_{Iw}$, propre pour la famille d'opérateurs de Hecke $U_{\pi_i}$ de valeurs propres $\alpha_i$. Supposons que pour tout $1\leq i \leq h$
$$\frac{d_i g(g+1)}{2} + e_i v(\alpha_i) < \inf_{1 \leq j \leq d_i} k_{g,j,i} $$
Alors $f$ est classique.
\end{theo}
 
\noindent Le reste de cette section est consacrée à le preuve de ce théorème. \\
Une forme modulaire surconvergente est définie sur un espace du type Deg$^{-1} ([f_1g - \varepsilon, f_1 g] \times \dots \times [f_h g - \varepsilon, f_h g])$, pour un certain $\varepsilon >0$. Pour montrer que $f$ est classique, nous allons tout d'abord prolonger $f$ à tout $X_{Iw}^{an}$. Le prolongement se fera direction par direction, c'est-à-dire que l'on prolongera $f$ à 
$$Deg^{-1} ([0,f_1g] \times~[f_2g - \varepsilon,f_2g] \times \dots \times [f_hg - \varepsilon,f_hg]) \cap X_{Iw}^{an}$$
puis à
$$Deg^{-1} ([0,f_1g] \times [0,f_2g] \times [f_3g - \varepsilon,f_3g] \times \dots \times~[f_hg -~\varepsilon,f_hg]) \cap X_{Iw}^{an} $$
et ainsi de suite. \\
Chacune de ses étapes se démontrant de manière analogue, nous ne détaillerons que la première, c'est-à-dire le prolongement à 
$$Deg^{-1} ([0,f_1g] \times [f_2g - \varepsilon,f_2g] \times \dots \times [f_hg - \varepsilon,f_hg]) \cap X_{Iw}^{an}$$
Pour conclure, nous utiliserons le théorème $\ref{extension}$, qui permettra d'étendre la forme $f$ à $\overline{X}_{Iw}^{an}$. Un théorème de type GAGA permet ensuite de prouver que $f$ est algébrique, c'est-à-dire que $f$ est une forme classique.

\subsection{Prolongement automatique}

Soit $f$ une forme modulaire surconvergente vérifiant les hypothèses du théorème $\ref{theogen}$. Elle est donc définie sur Deg$^{-1} ([f_1g - \varepsilon, f_1 g] \times \dots \times [f_h g - \varepsilon, f_h g])$, pour un certain $\varepsilon >0$. Pour tout intervalle $I$, notons $\mathcal{U}_I :=Deg^{-1} (I \times [f_2g - \varepsilon,f_2g] \times \dots \times [f_hg - \varepsilon,f_hg]) \cap X_{Iw}^{an} $. La forme $f$ est donc définie sur $\mathcal{U}_{[f_1g-\varepsilon,f_1g]}$. Nous allons prolonger $f$ à $\mathcal{U}_{]f_1g-\frac{1}{e_1},f_1g]}$. 

\begin{prop}
Il est possible de prolonger $f$ à $\mathcal{U}_{]f_1-\frac{1}{e_1},f_1]}$.
\end{prop}

\begin{proof}
Soit $\beta$ un rationnel avec $0<\beta<\frac{1}{e_1}$. D'après le corollaire $\ref{dyna}$, il existe un entier $N$ tel que 
$$U_{\pi_1}^N (\mathcal{U}_{[f_1g-\beta,f_1g]}) \subset \mathcal{U}_{[f_1g-\varepsilon,f_1g]}$$
La fonction $f_\beta = a_1^{-N} U_{\pi_1}^N f$ est donc définie sur $\mathcal{U}_{[f_1g-\beta,f_1g]}$, et est égale à $f$ sur $\mathcal{U}_{[f_1g-\varepsilon,f_1g]}$. Nous noterons donc encore $f$ cette fonction. De plus, les $(\mathcal{U}_{[f_1g-\beta,f_1g]})$ pour $0<\beta<\frac{1}{e_1}$ forment un recouvrement admissible de $\mathcal{U}_{]f_1g-\frac{1}{e_1},f_1g]}$. On peut donc étendre $f$ à ce dernier intervalle.
\end{proof}

\begin{rema}
Pour démontrer cette proposition, nous avons seulement utilisé le fait que la valeur propre $\alpha_1$ était non nulle.
\end{rema}

\subsection{Séries de Kassaei}

Dans cette partie, nous prolongeons la forme $f$ à $\mathcal{U}_{[0,f_1g]}$. Comme les itérés de l'opérateur $U_{\pi_1}$ n'augmentent pas strictement le degré de $H_{1,g}$ sur cet ouvert, la méthode de la partie précédente ne s'applique pas. Nous allons construire des séries $f_n$, analogues de celles introduites par Kassaei dans \cite{Ka}, qui convergeront vers $f$. Pour cela, nous utiliserons la décomposition de l'opérateur $U_{\pi_1}$ réalisée dans la partie $\ref{decompo}$. \\
Soit $\varepsilon$ un réel strictement positif tel que $v(\alpha_1) + f_1 g(g+1)/2 < (\frac{1}{e_1}-\varepsilon) \inf_{1 \leq j \leq d_1} k_{g,j,1}$. Cela est possible d'après les hypothèses du théorème $\ref{theogen}$. Soit $r$ un nombre rationnel avec $f_1g-\frac{1}{e_1} < r < f_1g-\frac{1}{e_1}+\varepsilon$, et $\mathcal{U}:=\mathcal{U}_{[0,r]}$. \\
Soit $N \geq 1$ un entier ; d'après le théorème $\ref{bigtheo}$, on peut trouver une suite d'ouverts $(\mathcal{U}_j)_{j \in S_N}$ de $\mathcal{U}$, et une décomposition de $U_{\pi_1}^N$ sur chaque cran $\mathcal{U}_j \backslash \mathcal{U}_{j+1}$. De plus, il est possible de faire surconverger arbitrairement cette suite d'ouverts. En effet, soit $(r^{(k)})$ une suite strictement croissante de rationnels avec $r^{(0)}=r$, $r^{(k)} < f_1g-\frac{1}{e_1}+\varepsilon$ pour tout $k$, et $(\mathcal{U}_j^{(k)})$ la suite d'ouverts correspondante à $r^{(k)}$. Alors $\mathcal{U}_j^{(k+1)}$ est un voisinage strict de $\mathcal{U}_j^{(k)}$ dans $\mathcal{U}$ pour tout $j,k$. \\
Notons $\mathcal{V}_j = \mathcal{U}_j^{(j-1)}$ pour tout $j \geq 1$, et $\mathcal{V}_j'=\mathcal{U}_j^{(j)}$ pour tout $i \geq 0$. Alors $\mathcal{V}_j'$ est un voisinage strict de $\mathcal{V}_j$ dans $\mathcal{U}$. Nous avons décomposé l'opérateur $U_{\pi_1}^N$ sur $\mathcal{V}_j' \backslash \mathcal{V}_{j+1}$ en 
$$U_{\pi_1}^N = \coprod_{k=0}^{N-1} U_{\pi_1}^{N-1-k} T_k \coprod T_N$$
avec $T_0 = U_{\pi_1,j}^{good}$ et pour $0 < k < N$
$$T_k = \coprod_{j_1 \in S_{N-1}, \dots, j_k \in S_{N-k}}  U_{\pi_1,j_k}^{good} U_{\pi_1,j_{k-1},j_k}^{bad} \dots U_{\pi_1,j,j_1}^{bad}$$
et
$$T_N = \coprod_{j_1 \in S_{N-1}, \dots, j_{N-1} \in S_{1}} U_{\pi_1,j_{N-1}}^{bad} U_{\pi_1,j_{N-2},j_{N-1}}^{bad} \dots U_{\pi_1,j,j_1}^{bad}$$  
Les images de $U_{\pi_1,j}^{good}$ et de $U_{\pi_1,j_k}^{good}$ ($j_k \in S_{N-k}$) sont incluses dans $\mathcal{U}_{[r^{(j)},f_1g]} \subset \mathcal{U}_{[r,f_1g]}$, et les opérateurs $U_{\pi_1,i,j}^{bad}$, $U_{\pi_1,j}^{bad}$ ne font intervenir que des supplémentaires $L$ de degré supérieur à $f_1g -~r^{(j)} > \frac{1}{e_1} - \varepsilon$. 

\begin{defi}
Les séries de Kassaei sur $\mathcal{V}_j' \backslash \mathcal{V}_{j+1}$ sont définies par
$$f_{N,j} := \alpha_1^{-1} U_{\pi_1,j}^{good} f + \sum_{k=1}^{N-1} \sum_{j_1 \in S_{N-1}, \dots, j_k \in S_{N-k}} \alpha_1^{-k-1} U_{\pi_1,j,j_1}^{bad} \dots U_{\pi_1,j_{k-1},j_k}^{bad} U_{\pi_1,j_k}^{good} f$$
\end{defi}

Cette fonction est bien définie, puisque les opérateurs $U_{\pi_1,j}^{good}$ sont soit nuls, auquel cas leur action sur $f$ donne $0$, soit à valeurs dans $\mathcal{U}_{[r,f_1g]}$ et $f$ est définie sur cet espace. Ce dernier espace étant quasi-compact, $f$ y est bornée, disons par $M$. \\
La proposition $\ref{lemnorm}$ permet de majorer la norme des opérateurs $\alpha_1^{-1} U_{p,j,k}^{bad}$ : la norme de ces opérateurs est inférieure à
$$u_0 = p^{f_1g(g+1)/2+v(\alpha_1) - (\frac{1}{e_1}-\varepsilon) \inf_{1 \leq j \leq d_1} k_{g,j,1}} < 1$$

\begin{lemm}
Les fonctions $f_{N,i}$ sont uniformément bornées.
\end{lemm}

\begin{proof}
On a 
$$ |\alpha_1^{-k-1} U_{\pi_1,j,j_1}^{bad} \dots U_{\pi_1,j_{k-1},j_k}^{bad} U_{\pi_1,j_k}^{good} f |_{\mathcal{V}_j' \backslash \mathcal{V}_{j+1}} \leq u_0^k |\alpha_1^{-1} U_{\pi_1,j_k}^{good} f |_{U_{\pi_1,j_{k-1},j_k}^{bad} \dots U_{\pi_1,j,j_1}^{bad} (\mathcal{V}_j' \backslash \mathcal{V}_{j+1} )} \leq |\alpha_1^{-1} | p^{f_1g(g+1)/2}  M$$
car la norme de $U_{\pi_1,j_k}^{good}$ est majorée par $p^{f_1g(g+1)/2}$.
On peut donc majorer la fonction $f_{N,j}$ par
$$ |f_{N,j} |_{\mathcal{V}_j' \backslash \mathcal{V}_{j+1}} \leq |\alpha_1^{-1} | p^{f_1g(g+1)/2} M$$
ce qui prouve que les fonctions $f_{N,j}$ sont uniformément bornées. 
\end{proof}

Puisque ces fonctions sont bornées, nous pouvons supposer qu'elles sont de norme inférieure à $1$, quitte à multiplier $f$ par une constante. Nous allons maintenant recoller ces fonctions sur $\mathcal{U}$.

\begin{lemm}
Soient $j,k \in S_N$ et $x \in (\mathcal{V}_j' \backslash \mathcal{V}_{j+1}) \cap (\mathcal{V}_k' \backslash \mathcal{V}_{k+1})$. Alors
$$ | (f_{N,j} -  f_{N,k}) (x) | \leq u_0^N M $$
\end{lemm}

\begin{proof}
Il existe une autre mani\`ere de d\'efinir les s\'eries de Kassaei $f_{N,j}$. En effet, on aurait pu d\'ecomposer l'op\'erateur $U_{\pi_1}^N$ en $U_{\pi_1}^{N,good} + U_{\pi_1}^{N,bad}$ suivant les degr\'es des points de $U_{\pi_1}^N$ : l'image de $U_{\pi_1}^{N,good}$ est incluse dans $Deg_1^{-1} (]s , f_1 g])$, et celle de $U_{\pi_1}^{N,bad}$ dans $Deg_1^{-1} ([0,s])$, pour un certain rationnel $s$ compris entre $r$ et $f_1 g -1/e_1 + \varepsilon$. La s\'erie de Kassaei est alors d\'efinie comme $\alpha_1^{-N} U_{\pi_1}^{N,good} f$. La s\'erie de Kassaei $f_{N,j}$ est d\'efinie \`a l'aide d'un rationnel $s$ comme pr\'ec\'edemment ; au-dessus du point $x$ on peut donc d\'ecomposer l'op\'erateur $U_{\pi_1}^N$ en $U_{\pi_1}^{N,good} + U_{\pi_1}^{N,bad}$, avec $f_{N,j} (x) = \alpha_1^{-N} U_{\pi_1}^{N,good} f(x)$. De m\^eme, la s\'erie $f_{N,k}$ est d\'efinie \`a l'aide d'un rationnel $s'$, et au-dessus de $x$, on a la d\'ecomposition $U_{\pi_1}^N = {U_{\pi_1}^{N,good}}' + {U_{\pi_1}^{N,bad}}'$, avec $f_{N,k} (x) = \alpha_1^{-N} {U_{\pi_1}^{N,good}}' f(x)$. \\
Supposons par exemple que $k<j$. On a alors $s' < s$, et au-dessus de $x$, l'op\'erateur $U_{\pi_1}^{N,bad}$ se d\'ecompose en ${U_{\pi_1}^{N,bad}}' + {U_{\pi_1}^{N,bad}}''$, l'op\'erateur ${U_{\pi_1}^{N,bad}}''$ ayant son image incluse dans $Deg_1^{-1} (]s' , s])$. On a alors $f_{N,k} (x) - f_{N,j} (x) = \alpha_1^{-N} {U_{\pi_1}^{N,bad}}'' f(x)$. De plus, la norme de de l'op\'erateur $\alpha_1^{-N} {U_{\pi_1}^{N,bad}}''$ est inf\'erieure \`a $u_0^N$ d'apr\`es les calculs sur les normes des op\'erateurs de Hecke. D'o\`u
$$ | (f_{N,j} -  f_{N,k}) (x) | \leq u_0^N |f|_{{U_{\pi_1}^{N,bad}}''(x)} $$
De plus, l'ensemble ${U_i^{N,bad}}'' (x)$ \'etant inclus dans $Deg_1^{-1} ([r , f_1 g])$, on a $|f|_{{U_i^{N,bad}}''(x)} \leq M$ ce qui donne la majoration. 
\end{proof}

\begin{prop}
Il existe un entier $A_N$ telle que les fonctions $(f_{N,j})_{j\in S_N}$ se recollent en une fonction $g_N \in H^0(\mathcal{U}, \tilde{\omega}^\kappa / p^{A_N})$. De plus, la suite $(A_N)_{N \geq 1}$ tend vers l'infini.
\end{prop}

\begin{proof}
La décomposition de l'ouvert $\mathcal{U}$ étant finie, soit $L$ tel que $\mathcal{V}_{L+1}$ soit vide. La fonction $f_{N,L}$ est donc définie sur $\mathcal{V}_L'$. La fonction $f_{N,L-1}$ est définie sur $\mathcal{V}_{L-1}' \backslash \mathcal{V}_L$. \\
De plus, d'après le lemme précédent, on a
$$ | f_{N,L-1} - f_{N,L} |_{(\mathcal{V}_L' \cap \mathcal{V}_{L-1}') \backslash \mathcal{V}_L} \leq u_0^N M$$
Soit $A_N$ le plus grand entier tel que $u_0^N M \leq p^{-A_N}$ ; comme $u_0 < 1$, la suite $(A_N)_{N \geq 1}$ tend vers l'infini. \\
Les fonctions $f_{N,L-1}$ et $f_{N,L}$ sont donc égales modulo $p^{A_N}$ sur $(\mathcal{V}_L' \cap \mathcal{V}_{L-1}') \backslash \mathcal{V}_L$. Comme $(\mathcal{V}_L' \cap~\mathcal{V}_{L-1}' , \mathcal{V}_{L-1}' \backslash  \mathcal{V}_L)$ est un recouvrement admissible de $\mathcal{V}_{L-1}'$ , celles-ci se recollent en une fonction $g_{N,L-1} \in H^0 ( \mathcal{V}_{L-1}', \tilde{\omega}^\kappa / p^{A_N})$. \\
De même, $g_{N,L-1}$ et $f_{N,L-2}$ sont égales (modulo $p^{A_N}$) sur $(\mathcal{V}_{L-2}' \cap \mathcal{V}_{L-1} ') \backslash \mathcal{V}_{L-1}$, et donc se recollent en $g_{N,L-2} \in H^0 (\mathcal{V}_{L-2}' , \tilde{\omega}^\kappa / p^{A_N})$. \\
En répétant ce processus, on voit que les fonctions $f_{N,j}$ se recollent toutes modulo $p^{A_N}$ sur $\mathcal{V}_0' = \mathcal{U}$, et définissent donc une fonction $g_N \in H^0(\mathcal{U}, \tilde{\omega}^\kappa / p^{A_N})$.
\end{proof}

\begin{prop}
Les fonctions $(g_N)$ définissent un système projectif dans $\underset{\leftarrow}\lim $ $ H^0 ( \mathcal{U}, \tilde{\omega}^\kappa / p^m)$.
\end{prop}

\begin{proof}
Nous allons prouver que $g_{N+1}$ et $g_N$ sont égales modulo $p^{A_N}$. Soit $ x \in \mathcal{U}$ ; nous avons construit en $x$ les séries de Kassaei $f_{N,j}$ et $f_{N+1,k}$. Or le terme $f_{N+1,k}$ provient d'une décomposition de $U_{\pi_1}^{N+1}$ du type
$$U_{\pi_1}^{N+1} = \sum_{l=0}^{N} U_{\pi_1}^{N-l} T_N + T_{N+1}$$
Nous pouvons donc écrire $f_{N+1,k} = h_1 + h_2$, la fonction $h_1$ étant associée à l'opérateur $\sum_{l=0}^{N-1} U_{\pi_1}^{N-1-l} T_N$  
et $h_2$ à $T_N$. \\
Or la fonction $h_1$ est en réalité une série de Kassaei pour une certaine décomposition de $U_{\pi_1}^N$ : le lemme précédent donne donc 
$$ | (f_{N,j} - h_1) (x) | \leq p^{-A_N}$$
De plus, on a
$$h_2 = \sum_{j_1 \in S_N, \dots, j_N \in S_1} \alpha_1^{-N-1} U_{\pi_1,j,j_1}^{bad} \dots U_{\pi_1,j_{N-1},j_N}^{bad} U_{\pi_1,j_N}^{good} f $$
donc comme les opérateurs $\alpha_{1}^{-1} U_{\pi_1,i,l}^{bad}$ ont une norme inférieure à $u_0$, 
$$ | h_2(x) | \leq u_0^N p^{f_1} |\alpha_1^{-1} | M \leq p^{-A_N'}$$
avec $A_N'=A_N - f_1 - v(\alpha_1)$. Quitte à remplacer $A_N$ par $A_N'$, on voit donc que la réduction de $g_{N+1}$ modulo $p^{A_N}$ est égal à $g_N$.
\end{proof}

En utilisant le gluing lemma (lemme $\ref{glue}$), on voit donc que les fonctions $g_N$ définissent une fonction $g \in H^0(\mathcal{U}_0,\omega^\kappa)$ pour tout ouvert quasi-compact $\mathcal{U}_0$ inclus dans $\mathcal{U}$, donc un élément de $H^0(\mathcal{U}_0,\omega^\kappa)$. Bien sûr, $g$ coïncide avec $f$ sur $\mathcal{U}_{]f_1g-\frac{1}{e_1},f_1g]}$. \\
En effet, si $x \in \mathcal{U}_{]f_1g-\frac{1}{e_1},f_1g]}$, il existe $N_0$ tel que $U_{\pi_1}^{N} (x) \subset \mathcal{U}_{[f_1g-\varepsilon,f_1g]}$ pour $N \geq N_0$, et la série de Kassaei est alors stationnaire égale à 
$$ \alpha_1^{-N_0} U_{\pi_1}^{N_0} f=f$$
Nous pouvons donc étendre $f$ à $\mathcal{U}_{[0,f_1g]}$.

\subsection{Fin de la démonstration}

Nous avons étendu $f$ à $\mathcal{U}_{[0,f_1g]} = Deg^{-1} ([0,f_1g] \times [f_2g-\varepsilon,f_2g] \times \dots \times [f_hg-\varepsilon,f_hg]) \cap X_{Iw}^{an}$. En utilisant le fait que $f$ soit propre pour $U_{\pi_2}$, et en utilisant la relation vérifiée par la valeur propre $\alpha_2$, la même méthode montre que l'on peut étendre $f$ à $Deg^{-1} ([0,f_1g] \times [0,f_2g] \times [f_3g-\varepsilon,f_3g] \times \dots \times [f_hg-\varepsilon,f_hg]) \cap X_{Iw}^{an}$. En répétant ce processus, on voit donc que l'on peut étendre $f$ à tout $X_{Iw}^{an}$.  \\
Nous avons donc étendu à la fonction $f$ en un élément de $H^0 (X_{Iw}^{an}, \omega^\kappa)$. Il nous reste encore à montrer que $f$ s'étend au bord. Dans le cas où $dg>1$, et dans le cas algébrique, le principe de Koecher nous assure que l'on peut négliger les pointes dans la définition des formes modulaires (voir \cite{LanKoecher}). Il doit sans doute être possible de déduire un analogue analytique de ce résultat, c'est-à-dire démontrer que toute forme modulaire définie sur $X_{Iw}^{an}$ s'étend à $\overline{X}_{Iw}^{an}$. Néanmoins, dans notre cas, il est possible de raisonner plus simplement, et en ne supposant pas que $dg>1$. Nous allons pour cela utiliser le théorème $\ref{extension}$ : il nous suffit de prouver que la forme $f$ est bornée sur $X_{Iw}^{an}$ pour prouver qu'elle s'étend au bord.

\begin{prop}
La forme $f$ est bornée sur $X_{Iw}^{an}$.
\end{prop} 

\begin{proof}
Rappelons que nous étions partis d'une section définie sur Deg$^{-1} ([f_1g - \varepsilon, f_1 g] \times \dots \times [f_h g - \varepsilon, f_h g])$, pour un certain $\varepsilon >0$. Comme cet espace est quasi-compact, $f$ y est automatiquement bornée. Nous allons maintenant démontrer, qu'à chaque étape du prolongement, $f$ reste bornée. \\
Regardons par exemple l'extension de $f$ à $\mathcal{U}_{[0,f_1 g]}$. La forme $f$ est obtenue en recollant les séries de Kassaei $g$ définies sur $\mathcal{U}_{[0,r]}$, et une forme $f_0$ définie sur $\mathcal{U}_{[r',f_1 g]}$ par la formule $f_0 = \alpha_1^{-N} U_{\pi_1}^N f$, où $r$ et $r'$ sont des rationnels vérifiant 
$$f_1 g - \frac{1}{e_1} < r' < r < f_1 g$$
Comme l'opérateur $U_{\pi_1}$ est borné, $f_0$ est bornée. De plus, comme cela a été vu dans le paragraphe précédent, les séries définissant $g$ sont uniformément bornées. On en déduit que le prolongement de $f$ à $\mathcal{U}_{[0,f_1 g]}$ est borné.
\end{proof}

Puisque $f$ est bornée sur $X_{Iw}^{an}$, on en déduit d'après le théorème $\ref{extension}$ qu'elle s'étend à $\overline{X}_{Iw}^{an}$. Comme $\overline{X}_{Iw}$ est propre, on en déduit par GAGA que $f$ provient d'un élément de $H^0( \overline{X}_{Iw}, \omega^\kappa)$, soit que $f$ est classique.

\section{Cas des variétés de type (A)} \label{partie_A}

\subsection{Données et variétés de Shimura}

Rappelons les données paramétrant les variétés de Shimura PEL de type (A) (voir \cite{Ko}). Soit $B$ une $\mathbb{Q}$-algèbre simple munie d'une involution positive $\star$. Soit $F$ le centre de $B$ et $F_0$ le sous-corps de $F$ fixé par $\star$. Le corps $F_0$ est une extension totalement réelle de $\mathbb{Q}$, soit $d$ son degré. Faisons les hypothèses suivantes :

\begin{itemize}
\item $[F:F_0]=2$.
\item Pour tout plongement $F_0 \to \mathbb{R}$, $B \otimes_{F_0} \mathbb{R} \simeq $M$_n(\mathbb{C})$, et l'involution $\star$ est donnée par $A \to \overline{A}^t$.
\end{itemize}

Soit également $(U_{\mathbb{Q}},\langle,\rangle)$ un $B$-module hermitien non dégénéré, l'accouplement étant alterné. Soit $G$ le groupe des automorphismes du $B$-module hermitien $U_{\mathbb{Q}}$ ; pour toute $\mathbb{Q}$-algèbre $R$, on a donc
$$G(R) = \left\{ (g,c) \in GL_{B} (U_{\mathbb{Q}} \otimes_{\mathbb{Q}} R) \times R^* , \langle gx,gy \rangle =c \langle x,y\rangle \text{ pour tout } x,y \in U_{\mathbb{Q}} \otimes_{\mathbb{Q}} R \right\} $$

Soient $\tau_1, \dots, \tau_d$ les plongements de $F_0$ dans $\mathbb{R}$ ; soit également $\sigma_i$ et $\overline{\sigma_i}$ les deux plongements de $F$ dans $\mathbb{C}$ étendant $\tau_i$. Le choix de $\sigma_i$ donne un isomorphisme $F \otimes_{F_0} \mathbb{R} \simeq \mathbb{C}$. On a également $B_i=B \otimes_{F_0,\tau_i} \mathbb{R} \simeq $M$_n (\mathbb{C})$. Notons $U_i = U_\mathbb{Q} \otimes_{F_0,\tau_i} \mathbb{R}$. D'après l'équivalence de Morita, $U_i \simeq~\mathbb{C}^n \otimes W_i$, où $B_i$ agit sur le premier facteur et $W_i$ est un $\mathbb{C}$-espace vectoriel. La structure anti-hermitienne sur $U_i$ en induit une sur $W_i$, et on note $(a_i,b_i)$ sa signature.
 Alors $G_\mathbb{R}$ est isomorphe à
$$\text{G} \left( \prod_{i=1}^d \text{U}(a_i,b_i) \right)$$
où $a_i + b_i$ est indépendant de $i$ et vaut $\frac{1}{2nd} $dim$_\mathbb{Q} U_\mathbb{Q}$. \\
Donnons-nous également un morphisme de $\mathbb{R}$-algèbres $h : \mathbb{C} \to $End$_B U_\mathbb{R}$ tel que $\langle h(z)v,w\rangle ~=~\langle v,h(\overline{z})w\rangle$ et $(v,w) \to \langle v,h(i)w\rangle$ est définie positive. Ce morphisme définit donc une structure complexe sur $U_\mathbb{R}$ : soit $U^{1,0}_{\mathbb{C}}$ le sous-espace de $U_\mathbb{C}$ pour lequel $h(z)$ agit par la multiplication par $z$. \\
On a alors $U^{1,0}_{\mathbb{C}} \simeq \prod_{i=1}^d (\mathbb{C}^n)^{a_i} \oplus \overline{(\mathbb{C}^n)}^{b_i}$ en tant que $B \otimes_{\mathbb{Q}} \mathbb{R} \simeq \oplus_{i=1}^d $M$_n( \mathbb{C})$-module (l'action de M$_n(\mathbb{C})$ sur $(\mathbb{C}^n)^{a_i} \oplus \overline{(\mathbb{C}^n)}^{b_i}$ est l'action standard sur le premier facteur et l'action conjuguée sur le second) . \\
Soient également un ordre $O_B$ de $B$ stable par $\star$, et un réseau $U$ de $U_\mathbb{Q}$ tel que l'accouplement $\langle,\rangle$ restreint à $U\times U$ soit à valeurs dans $\mathbb{Z}$. Nous ferons également les hypothèses suivantes : 
\begin{itemize}
\item $B \otimes_{\mathbb{Q}} \mathbb{Q}_p$ est isomorphe à un produit d'algèbres de matrices à coefficients dans une extension finie de $\mathbb{Q}_p$.
\item $O_B$ est un ordre maximal en $p$.
\item L'accouplement $U \times U \to \mathbb{Z}$ est parfait en $p$.
\end{itemize}
$ $\\
Soit $\mathbb{Z}_{(p)}$ le localisé de $\mathbb{Z}$ en $p$ ; $O_B$ est un $\mathbb{Z}_{(p)}$-module libre. Soit $\alpha_1, \dots, \alpha_t$ une base de ce module, et 
$$\text{det}_{U^{1,0}} = f(X_1, \dots, X_t) = \det (X_1 \alpha_1 + \dots + X_t \alpha_t ;U^{1,0}_{\mathbb{C}} \otimes_{\mathbb{C}} \mathbb{C} [X_1, \dots, X_t])$$
On montre (\cite{Ko}) que $f$ est un polynôme à coefficients algébriques. Le corps de nombres $E$ engendré par ses coefficients est appelé le corps réflexe. \\
Soit $p = \prod_{i=1}^h \pi_i ^{e_i}$ la décomposition de $p$ dans $F_0$, et soit $f_i$ le degré résiduel de chacune de ces places. On notera $\Sigma$ l'ensemble des plongements de $F_0$ dans $\overline{\mathbb{Q}}_p$, et $\Sigma_i$ le sous-ensemble des plongements envoyant $\pi_i$ dans l'idéal maximal de $\overline{\mathbb{Z}_p}$. On notera également $F_{0,i}$ la complétion de $F_0$ en $\pi_i$, et $O_{F_{0,i}}$ son anneau des entiers. Alors $B \otimes_{\mathbb{Q}} \mathbb{Q}_p \simeq \prod_{i=1}^{h} B_i$, où $B_i$ est la complétion de $B$ en $\pi_i$. Pour déterminer la structure de $B_i$ ,on peut distinguer $3$ cas.
\begin{itemize}
\item Cas $1$ : $\pi_i$ est décomposé dans $F$. Alors $B_i \simeq $ M$_n (F_{0,i}) \oplus $M$_n (F_{0,i})$.
\item Cas $2$ : $\pi_i$ est inerte dans $F$. Alors $B_i \simeq $ M$_n (F_i)$, où $F_i$ est la complétion de $F$ en $\pi_i$.
\item Cas $3$ : $\pi_i$ est ramifié dans $F$, $\pi_i = \varpi_i^2$. Alors $B_i \simeq $ M$_n (F_i)$, où $F_i$ est la complétion de $F$ en $\varpi_i$.
\end{itemize}
$ $\\
Définissons maintenant la variété de Shimura PEL de type (A) associée à $G$. Soit $K$ une extension de $\mathbb{Q}_p$ contenant les images de tous les plongements possibles $F \hookrightarrow \overline{\mathbb{Q}}_p$ et $E \hookrightarrow \overline{\mathbb{Q}}_p$.

\begin{defi}
Soit $X$ l'espace de modules sur Spec$(K)$ dont les $S$-points sont les classes d'isomorphismes des $(A,\lambda,\iota,\eta)$ où
\begin{itemize}
\item $A \to S$ est un schéma abélien
\item $\lambda : A \to A^t$ est une polarisation de degré premier à $p$.
\item $\iota : O_B \to $ End $A$ est compatible avec les involutions $\star$ et de Rosati, et les polynômes $\det_{U^{1,0}}$ et $\det_{Lie (A)}$ sont égaux.
\item $\eta : A[N] \to U/NU$ est une similitude symplectique $O_B$-linéaire, qui se relève localement pour la topologie étale en une similitude symplectique $O_B$-linéaire
$$H_1 (A,\mathbb{A}_f^p) \to U \otimes_{\mathbb{Z}} \mathbb{A}_f^p$$
\end{itemize}
\end{defi}

La condition du déterminant est explicite : si $S = $ Spec$(R)$, cela signifie que le faisceau conormal $\omega_A$ est isomorphe à $St \otimes_{K} R$ comme $B \otimes_\mathbb{Q} R$-module, où $St$ est défini par $St = \oplus_{i=1}^h St_i$, et $St_i$ est le $B_i$-module égal à
$$\bigoplus_{\tau \in \Sigma_i} (K^{a_{\tau}})^n \oplus (K^{b_{\tau}})^n$$
où $\Sigma_i = $Hom$(F_{0,i},\overline{\mathbb{Q}}_p)$, et où $B_i = $ M$_n (F_{0,i}) \oplus $M$_n (F_{0,i})$ agit par l'action standard donnée par $\tau$ sur chacun des facteurs dans le cas $1$ ; et
$$\bigoplus_{\tau \in \Sigma_i} (K^{a_{\tau}})^n \oplus (K^{b_{\tau}})^n$$
où l'action standard de  M$_n (F_i)$ est donnée par $\sigma$ sur le premier facteur, et par $\overline{\sigma}$ sur le deuxième (avec $\sigma$ et $\overline{\sigma}$ les deux plongements de $F_i$ au-dessus de $\tau$) dans les cas $2$ et $3$.

\begin{rema} 
Le schéma $X$ est en fait défini sur le corps réflexe, mais nous aurons besoin d'élargir le corps de base pour définir certains faisceaux ultérieurement.
\end{rema}

Pour définir les formes surconvergentes, nous aurons besoin de supposer que le lieu ordinaire de la variété de Shimura est non-vide. Dans le cas où $p$ est non ramifié dans $F$, un résultat de Wedhorn (\cite{We}) dit que cela est le cas si et seulement si $p$ est totalement décomposé dans le corps réflexe $E$. Si le corps $F$ est fixé, cela donne une condition sur les nombres $(a_\sigma, b_\sigma)$. Ainsi, si on se place dans le cas le plus simple où $F_0 = \mathbb{Q}$, $B=F$, on considère la variété associée au groupe $GU(a,b)$. Le corps réflexe est égal à $\mathbb{Q}$ si $a=b$ et $F$ sinon. L'existence du lieu ordinaire est alors équivalente à $p$ décomposé dans $F$ ou $a=b$. Nous allons obtenir des conditions nécessaires sur les couples $(a_\sigma, b_\sigma)$ dans le cas général. 

\begin{prop}
Supposons que le lieu ordinaire soit non vide. Soit $1 \leq i \leq h$, et supposons que $\pi_i$ soit dans le cas $1$. Alors il existe des entiers $a_i$ et $b_i$ tels que $(a_\sigma,b_\sigma) = (a_i,b_i)$ pour tout $\sigma \in \Sigma_i$. Si $\pi_i$ est dans le cas $2$ ou $3$, alors $a_\sigma = b_\sigma = (a+b)/2$ pour tout $\sigma \in \Sigma_i$.
\end{prop}

\begin{proof}
Supposons l'existence du lieu ordinaire. Cela veut dire qu'il existe une variété abélienne $A$ sur une extension $L$ de $\mathbb{Q}_p$, qui s'étend en un schéma semi-abélien $A_0$ (sur l'anneau des entiers d'une extension finie de $L$), tel que $A_0[p^\infty]$ soit un groupe de Barsotti-Tate ordinaire (c'est-à-dire extension d'une partie multiplicative et d'une partie étale). En particulier, pour tout $1 \leq i \leq h$, $A_0[\pi_i^\infty]$ est ordinaire. Supposons que $\pi_i$ est dans le cas $1$, et soit $\pi_i^+$ une place de $F$ au-dessus de $\pi_i$. Le groupe de Barsotti-Tate $A_0[(\pi_i^+)^\infty]$ est ordinaire, et munie d'une action de $O_{F_{0,i}}$. On en déduit qu'il existe des entiers $N_1$ et $N_2$ tels que 
$$A_0[(\pi_i^+)^\infty] = (\mu_{p^\infty} \otimes_{\mathbb{Z}_p} O_{F_{0,i}} ) ^{N_1} \times (\mathbb{Q}_p / \mathbb{Z}_p \otimes_{\mathbb{Z}_p} O_{F_{0,i}} ) ^{N_2}$$
En particulier, si on ordonne le couple $(a_\sigma, b_\sigma)$ tels que $a_\sigma$ corresponde au plongement au-dessus de $\pi_i^+$, on voit que $a_\sigma$ ne dépend pas de $\sigma \in \Sigma_i$. De même pour $b_\sigma$. Il existe donc un couple d'entiers $(a_i,b_i)$ tels que $(a_\sigma, b_\sigma) = (a_i, b_i)$ pour tout $\sigma \in \Sigma_i$.
Dans le cas $2$ ou $3$, si $O_{F_i}$ désigne l'anneau des entiers de $F_i$, on voit de même qu'il existe des entiers $N_1$ et $N_2$ tels que 
$$A_0[\pi_i^\infty] = (\mu_{p^\infty} \otimes_{\mathbb{Z}_p} O_{F_i} ) ^{N_1} \times (\mathbb{Q}_p / \mathbb{Z}_p \otimes_{\mathbb{Z}_p} O_{F_i} ) ^{N_2}$$
(en fait $N_1 = N_2$ par auto-dualité). On en déduit que pour tout $\sigma \in \Sigma_i$, $a_\sigma = b_\sigma = (a+b)/2$.
\end{proof}

Définissons maintenant la structure de niveau Iwahorique. Si $\pi_i$ est dans le cas $1$, on notera $\pi_i^+$ et $\pi_i^-$ les idéaux de $F$ au-dessus de $\pi_i$, et $(a_i,b_i)$ le couple $(a_\sigma,b_\sigma)$. Si $\pi_i$ est dans le cas $2$ ou $3$, on notera également $a_i =b_i= (a+b)/2$, de telle sorte que l'on ait encore $a_\sigma = a_i$ et $b_\sigma = b_i$ pour tout $\sigma \in \Sigma_i$.

\begin{defi}
Soit $X_{Iw}$ l'espace de modules sur $K$ dont les $S$-points sont les $(A,\lambda,\iota,\eta,H_{i,j})$ où $(A,\lambda,\iota,\eta) \in X(S)$ et
\begin{itemize}
\item Si $\pi_i$ est dans le cas $1$, $\left (0=H_{i,0} \subset H_{i,1} \subset \dots \subset H_{i,a+b} =A[\pi_i^+] \right )$ est un drapeau de sous-groupes finis et plats de $A[\pi_i^+]$, stables par $O_B$, chaque $H_{i,j}$ étant de hauteur $nf_i j$.
\item Si $\pi_i$ est dans le cas $2$, $\left (0=H_{i,0} \subset H_{i,1} \subset \dots \subset H_{i,(a+b)/2} \right )$ est un drapeau de sous-groupes finis et plats de $A[\pi_i]$, stables par $O_B$ et totalement isotropes, chaque $H_{i,j}$ étant de hauteur $2nf_i j$.
\item Si $\pi_i$ est dans le cas $3$, $\left (0=H_{i,0} \subset H_{i,1} \subset \dots \subset H_{i,(a+b)/2} \right )$ est un drapeau de sous-groupes finis et plats de $A[\pi_i]$, stables par $O_B$ et totalement isotropes, chaque $H_{i,j}$ isomorphe localement pour la topologie étale à $(O_B / \pi_i O_B)^j$.
\end{itemize}
\end{defi}

\begin{rema}
Dans le cas $1$, les sous-groupes $A[\pi_i^+]$ et $A[\pi_i^-]$ de $A[\pi_i]$ sont duaux pour l'accouplement de Weil. Si $H$ est un sous-groupe de $A[\pi_i^+]$, alors son orthogonal $H^\bot$ est un sous-groupe de $A[\pi_i^-]$ isomorphe au dual de Cartier de $A[\pi_i^+] / H$.
\end{rema}

\begin{rema}
Dans le cas $3$, le groupe $A[\varpi_i]$ est totalement isotrope, ce qui justifie le fait que l'on travaille avec $A[\pi_i]$ plutôt que $A[\varpi_i]$.
\end{rema}

On notera $\overline{X}$ et $\overline{X}_{Iw}$ des compactifications toroïdales de $X$ et $X_{Iw}$ construites par exemples dans \cite{Pin}. Soient également $X^{an}$, $X_{Iw}^{an}$, $\overline{X}^{an}$ et $\overline{X}_{Iw}^{an}$ les espaces analytiques associés.

\subsection{Formes modulaires et opérateurs de Hecke}

Soit $A$ le schéma semi-abélien universel sur $\overline{X}$, et soit $\omega_A = e^* \Omega_{A/\overline{X}}^1$ le faisceau conormal relatif à la section unité de $A$ ; il est localement pour la topologie de Zariski isomorphe à $St \otimes \mathcal{O}_{\overline{X}}$ comme $B \otimes \mathcal{O}_{\overline{X}}$-module. Rappelons que $St = \oplus_{i=1}^h St_i$, et $St_i$ est le $B_i \otimes_{\mathbb{Q}} K$-module égal à $\oplus_{\tau \in \Sigma_i} (K^{a_{\tau}})^n \oplus (K^{b_{\tau}})^n$. \\
Soit $\mathcal{T} = $Isom$_{B \otimes \mathcal{O}_{\overline{X}}} (St \otimes \mathcal{O}_{\overline{X}}, \omega_A)$. C'est un torseur sur $\overline{X}$ sous le groupe
$$M=\left ( \prod_{i=1}^{h} \prod_{\tau \in \Sigma_i}  {\text{GL}_{a_\tau}} {\times \text{GL}_{b_\tau}} \right ) \times_{\mathbb{Q}_p} K$$

Soit $T_M$ le tore diagonal de $M$, $B_M$ son Borel supérieur, et $U_M$ le radical unipotent. Soit $X(T_M)$ le groupe des caractères de $T_M$, et $X(T_M)^+$ le cône des poids dominants pour $B_M$. Si $\kappa \in X(T_M)^+$, on note $\kappa'=- w_0 \kappa \in X(T_M)^+$, où $w_0$ est l'élément le plus long du groupe de Weyl de $M$ relativement à $T_M$. \\
Soit $\phi : \mathcal{T} \to \overline{X}$ le morphisme de projection.

\begin{defi}
Soit $\kappa \in X(T_M)^+$. Le faisceau des formes modulaires de poids $\kappa$ est $\omega^\kappa =~\phi_* O_\mathcal{T}[\kappa']$, où $\phi_* O_\mathcal{T}[\kappa']$ est le sous-faisceau de $\phi_* O_\mathcal{T}$ où $B_M=T_M U_M$ agit par $\kappa'$ sur $T_M$ et trivialement sur $U_M$.
\end{defi} 

Une forme modulaire de poids $\kappa$ sur $\overline{X}$ est donc une section globale de $\omega^\kappa$, soit un élément de $H^0(\overline{X} , \omega^\kappa)$. En utilisant la projection $\overline{X}_{Iw} \to \overline{X}$, on définit de même le faisceau $\omega^\kappa$ sur $\overline{X}_{Iw}$, ainsi que les formes modulaires sur $\overline{X}_{Iw}$. On notera encore $\omega^\kappa$ le faisceau analytifié sur $\overline{X}_{Iw}^{an}$. \\
$ $\\
Définissons maintenant les opérateurs de Hecke. Soit $1 \leq i \leq h$, et $C_i$ l'espace de modules sur $K$ dont les $S$-points sont les $(A,\lambda,\iota,\eta,H_{j,k},L)$ avec $(A,\lambda,\iota,\eta,H_{j,k}) \in X_{Iw}(S)$ et $L$ un sous-groupe fini et plat stable par $O_B$ de $A[\pi_i]$ vérifiant
\begin{itemize}
\item $L=L_0 \oplus L_0^\bot$, où $L_0$ est un supplémentaire de $H_{i,a_i}$ dans $A[\pi_i^+]$ dans le cas $1$. 
\item $L$ est totalement isotrope, et est un supplémentaire de $H_{i,a_i}$ dans $A[\pi_i]$ dans le cas $2$ ou $3$.
\end{itemize}
Nous avons deux morphismes finis étales $p_1, p_2 : C_i \to X_{Iw}$ : $p_1$ est l'oubli de $L$, et $p_2$ est le quotient par $L$. Plus précisément :
\begin{itemize}
\item dans le cas $1$, l'image de de $(A,\lambda,\iota,\eta,H_{j,k},L=L_0 \oplus L_0^\bot)$ est $(A/L,\lambda',\iota',\eta',H_{j,k}')$, où $H_{j,k}'$ est l'image de $H_{j,k}$ dans $A/L$ si $j \neq i$ ou $j=i$ et $k \leq a_i$, et $H_{i,k}'$ est l'image de $(\pi_i^+)^{-1} (H_{i,k} \cap L_0)$ dans $A/L$ si $k > a_i$.
\item dans le cas $2$ ou $3$, l'image de $(A,\lambda,\iota,\eta,H_{j,k},L)$ est $(A/L,\lambda',\iota',\eta',H_{j,k}')$ où $H_{j,k}'$ est l'image de $H_{j,k}$ dans $A/L$.
\end{itemize}
Comme dans le cas C, une ambigu\"ité existe pour la polarisation sur $A/L$. Soit $\xi_i$ un élément totalement positif de l'anneau des entiers de $F_0$, avec $v_{\pi_i} (\xi_i) = 1$ et $v_{\pi_j} (\xi_i) = 0$ si $j \neq i$. On définit la polarisation sur $A/L$ comme la polarisation descendue $\xi_i \cdot \lambda$. Il s'agit bien d'une polarisation de degré premier à $p$. Le morphisme $p_2$ dépend donc du choix d'un tel élément $\xi_i$, mais ce choix n'est pas important en pratique (voir \cite{Bi} remarque $2.3.2$). Remarquons que si $p$ est inerte dans $F_0$, on peut choisir $\xi_i = p$. \\
Soit $C_i^{an}$ l'espace analytique associé à $C_i$ ; on note encore $p_1$, $p_2 : C_i^{an} \to X_{Iw}^{an}$ les morphismes (finis étales) induits. Il existe une compactification toroïdale $\overline{C}_i$ de $C_i$, et on peut supposer (par le théorème $\ref{morphcompact}$) que les morphismes $p_1, p_2 : C_i \to X_{Iw}$ s'étendent en des morphismes $\overline{C}_i \to \overline{X}_{Iw}$. Si on note $\overline{C}_i^{an}$ l'espace rigide analytique associé à $\overline{C}_i$, on obtient des morphismes $p_1, p_2 : \overline{C}_i^{an} \to \overline{X}_{Iw}^{an}$. 

\begin{defi}
L'opérateur de Hecke géométrique agissant sur $\overline{X}_{Iw}^{an}$ est défini par $U_{\pi_i} (S) := p_2(p_1^{-1}(S))$ pour toute partie $S$ de $\overline{X}_{Iw}^{an}$. 
\end{defi} 

Cet opérateur respecte les ouverts de $X_{Iw}^{an}$, mais pas ceux de $\overline{X}_{Iw}^{an}$ en général. Notons $q : A \to A/L$ l'isogénie universelle au-dessus de $C_i$. Celle-ci induit un isomorphisme $q^* : \omega_{A/L} \to \omega_{A}$, et donc un morphisme $q^* (\kappa) : p_2^* \omega^\kappa \to~p_1^* \omega^\kappa$. Pour tout ouvert $\mathcal{U}$ de $X_{Iw}^{an}$, nous pouvons donc former le morphisme composé
\begin{displaymath}
\widetilde{U}_{\pi_i} :   H^0(U_{\pi_i}(\mathcal{U}),\omega^\kappa) \to H^0 ( p_1^{-1} (\mathcal{U}), p_2^* \omega^\kappa) \overset{q^*(\kappa)}{\to} H^0(p_1^{-1}(\mathcal{U}) , p_1^* \omega^\kappa) \overset{Tr_{p_1}}{\to}  H^0(\mathcal{U},\omega^\kappa)
\end{displaymath}

\begin{defi}
L'opérateur de Hecke agissant sur les formes modulaires est alors défini par $U_{\pi_i} = \frac{1}{p^{N_i}} \widetilde{U}_{\pi_i}$ avec $N_i=f_i a_i b_i$. 
\end{defi}

Nous avons donc défini $h$ opérateurs agissant sur les formes modulaires définies sur un ouvert de $X_{Iw}^{an}$. De même que dans le cas précédent, comme ces opérateurs sont bornés, ils agissent sur l'espace $H^0 (\overline{X}_{Iw}^{an}, \omega^\kappa)$. En effet, l'image par $U_{\pi_i}$ d'une telle section sera bornée, et s'étendra automatiquement au bord d'après le théorème $\ref{extension}$.

\subsection{Structures entières} \label{hecke_cas_A}

Nous allons maintenant définir les structures entières sur $\overline{X}_{Iw}^{an}$, c'est-à-dire les fonctions degré, et une norme pour l'espace des formes modulaires. Nous gardons les notations des parties précédentes. \\
Soit $h_i$ l'entier tel que $nh_i$ soit la hauteur de $H_{i,a_i}$ ; on a donc $h_i = f_i a_i$ dans le cas $1$ et $h_i = 2f_i a_i$ dans les cas $2$ et $3$. Soit $\mathcal{A}_{nd(a+b),h_i}$ l'espace de Siegel analogue à celui de la définition \ref{siegel}, mais en demandant que la hauteur de $H$ soit égale à $nh_i$. On dispose d'un morphisme  $\mathcal{P}_i : X_{Iw} \to \mathcal{A}_{nd(a+b),h_i} \times K $ par la formule $(A,\lambda,\iota,\eta,H_{j,k}) \to (A,\lambda,\eta,H_{i,a_i})$. De plus, il existe une compactification toroïdale $\overline{\mathcal{A}}_{nd(a+b),h_i}$ de $\mathcal{A}_{nd(a+b),h_i}$, et quitte à restreindre les décompositions polyhédrales utilisées pour construire les compactifications toroïdales, ces morphismes s'étendent en $\mathcal{P}_i : \overline{X}_{Iw} \to \overline{\mathcal{A}}_{nd(a+b),h_i} \times K$ par le théorème $\ref{morphcompact}$. \\
Soit $\overline{\mathcal{A}}_{nd(a+b),h_i}^{rig}$ l'espace rigide associé à $\overline{\mathcal{A}}_{nd(a+b),h_i} \times_{\mathbb{Z}_p} O_K$, où $O_K$ est l'anneau des entiers de $K$. Comme ce dernier schéma est propre, cet espace rigide est égal à l'espace analytique associé à $\overline{\mathcal{A}}_{nd(a+b),h_i} \times K$. Les morphismes $\mathcal{P}_i$ induisent des morphismes $\overline{X}_{Iw}^{an} \to \overline{\mathcal{A}}_{nd(a+b),h_i}^{rig}$. Rappelons que nous avons défini une fonction deg $: \overline{\mathcal{A}}_{nd(a+b),h_i}^{rig} \to [0,h_i]$.

\begin{defi}
On définit la fonction Deg$_i : \overline{X}_{Iw}^{an} \to [0,f_i a_i]$ par la formule $ x \to $ deg $\mathcal{P}_i ( x)$ dans le cas $1$ et $ x \to \frac{1}{2}$ deg $\mathcal{P}_i ( x)$ dans les cas $2$ et $3$. La fonction degré Deg : $\overline{X}_{Iw}^{an} \to \prod_{i=1}^h [0,f_i a_i]$ est définie par $x \to ($Deg$_i (x))_i$.
\end{defi}

\begin{rema} \label{part}
Le fait de diviser par $2$ le degré de $H_{i,a_i}$ dans les cas $2$ et $3$ est justifié par le fait suivant. Supposons que $A$ soit une variété abélienne avec bonne réduction sur $O_{L}$, avec $L$ une extension finie de $\mathbb{Q}_p$. Supposons pour simplifier que $B=F$, et que $\pi_i$ est dans le cas $2$. Alors le groupe $H_{i,a_i}$ est un schéma en groupe fini et plat sur $O_L$, muni d'une action de $O_{F_i^{nr}}$ où $F_i^{nr}$ est l'extension maximale non ramifiée contenue dans $F_i$. Soit $F_{0,i}$ la complétion de $F_0$ en $\pi_i$, et $F_{0,i}^{nr}$ l'extension maximale non ramifiée contenue dans $F_{0,i}$. Par hypothèse, $F_i^{nr}$ est une extension de degré $2$ de $F_{0,i}^{nr}$. Soit $S_0$ (resp. $S$) l'ensemble des plongements de $F_{0,i}^{nr}$ (resp. $F_i^{nr}$) dans $\overline{\mathbb{Q}}_p$. Nous avons défini dans la partie $\ref{partial}$ les degrés partiels deg$_s H_{i,a_i}$ pour tout $s \in S$. Le groupe $H_{i,a_i}$ étant égal à son orthogonal dans $A[\pi_i]$, on a $H_{i,a_i} \simeq (A[\pi_i] / H_{i,a_i})^{D,c}$, où $^D$ signifie le dual de Cartier, et $^c$ que l'action de $O_{F}$ est obtenue par conjugaison (cela résulte de la compatibilité entre l'action de $O_F$ et l'involution de Rosati). Si $s_0$ est un élément de $S_0$, et si $s$ et $\overline{s}$ sont les deux éléments de $S$ au-dessus de $s_0$, alors la proposition $\ref{dual}$ montre que 
$$\text{deg}_s H_{i,a_i} = \text{deg}_{\overline{s}} H_{i,a_i}$$
La quantité pertinente pour étudier le sous-groupe $H_{i,a_i}$ n'est donc pas son degré, mais la moitié de celui-ci. \\
Remarquons enfin que dans le cas $1$, le groupe $H_{i,a_i}^\bot$, qui est un sous-groupe de $A[\pi_i^-]$ a pour degré $f_i(b_i - a_i) + $ deg $H_{i,a_i}$. Dans le cas où $a_i=b_i$, le groupe $H_{i,a_i} \oplus H_{i,a_i}^\bot$, qui est totalement isotrope, a pour degré $2$ deg $H_{i,a_i}$. \\
Une autre justification pour cette définition est de considérer le cas $F_0 = \mathbb{Q}$ et $B=F$ est centrale ($F$ est donc un corps quadratique imaginaire). On considère donc le groupe $GU(a,b)$. Le cas $a=b=1$ correspond au cas de la courbe modulaire ; au niveau des espaces de modules, cela s'interprète par le fait que tout schéma abélien apparaissant dans la variété unitaire s'écrit comme $E \otimes_{\mathbb{Z}} O_F$, où $E$ est une courbe elliptique. Or on sait que la quantité pertinente pour étudier la courbe modulaire de niveau Iwahorique en $p$ est le degré du sous-groupe universel, qui est compris entre $0$ et $1$. Nous devons donc retrouver cette quantité pour la variété unitaire, ce qui justifie notre définition. 
\end{rema}

Si $I= \prod_{k=1}^h I_k$ est un produit d'intervalles, on note $\overline{X}_{Iw,I} = $ Deg$^{-1} (I)$. Le lieu ordinaire-multiplicatif $\overline{X}_{Iw}^{mult}$ correspond au lieu où tous les degrés sont maximaux, c'est à dire à $\overline{X}_{Iw,I}$ avec $I= \prod_{i=1}^h \{f_i a_i\}$. Par hypothèse, nous nous plaçons dans le cas où ce lieu est non vide.

\begin{defi}
L'espace des formes modulaires surconvergentes est défini par
$$H^0(\overline{X}_{Iw}^{an},\omega^\kappa)^\dagger := \text{colim}_\mathcal{V} H^0 (\mathcal{V},\omega^\kappa)$$
où la colimite est prise sur les voisinages stricts $\mathcal{V}$ de $\overline{X}_{Iw}^{mult}$ dans $\overline{X}_{Iw}^{an}$.
\end{defi}

Une forme modulaire surconvergente est donc définie sur un espace du type $\overline{X}_{Iw,I}$ avec $I=~\prod_{i=1}^h [f_i a_i -~ \varepsilon, f_i a_i]$, pour un certain $\varepsilon >0$. \\
Nous avons des propriétés analogues quant au comportement des opérateurs de Hecke relativement à la fonction Degré. 

\begin{prop}
Soit $1 \leq i \leq h$, $x \in \overline{X}_{Iw}^{an}$ et $y \in U_{\pi_i} (x)$. Soit $x_j= Deg_j(x)$, et $y_j=Deg_j(y)$ pour $1 \leq j \leq h$. Alors
\begin{itemize}
\item $y_j=x_j$ pour $j \neq i$.
\item $y_i \geq x_i$
\end{itemize}
De plus, s'il existe $y \in U_{\pi_i}^{2e_i} (x)$ avec $Deg_i(y)=Deg_i(x)$, alors $x_i \in \frac{1}{e_i} \mathbb{Z}$ dans les cas $1$ et $2$, et $x_i \in \frac{1}{2e_i} \mathbb{Z}$ dans le cas $3$.
\end{prop}

\begin{proof}
Nous raisonnons comme dans la démonstration de la proposition $\ref{augmente}$. Le premier point est identique. Pour le deuxième point, supposons qu'il existe $y \in U_{\pi_i}^{2 e_i} (x)$ avec $Deg_i(y) = Deg_i(x)$. Soit $A$ la variété semi-abélienne associée à $x$, définie sur une extension finie $M$ de $\mathbb{Q}_p$. Le point $y$ correspond à un sous-groupe $L$ de $A[\pi_i^{2e_i}]$. De plus, quitte à élargir $M$, on peut supposer que $A$ s'étend en un schéma semi-abélien $A_0$ sur $O_M$, et $L$ en un sous-groupe quasi-fini et plat $L_0$ de $A_0 [\pi_i^{2e_i}]$. En gardant les mêmes notations que la démonstration $\ref{augmente}$, on note $\widetilde{G}$ un schéma semi-abélien de rang torique constant tel que $A_0$ soit obtenu comme un quotient de $\widetilde{G}$ par la construction de Mumford (voir l'annexe). Le même raisonnement que pour la proposition $\ref{augmente}$ montre que $L_1 := L_0 [\pi_i^{e_i}] \cap \widetilde{G} [p]$ est un groupe de Barsotti-Tate tronqué d'échelon $1$. On en déduit immédiatement que son degré est entier, et donc que $Deg_i (x) \in \frac{1}{2e_i} \mathbb{Z}$ avec notre définition de la fonction $Deg_i$. On veut montrer que $Deg_i (x) \in \frac{1}{e_i} \mathbb{Z} $ dans les cas $1$ et $2$.  On se ramène facilement au cas où $n=1$. Etudions alors les deux cas possibles. \\
Dans le premier cas, on peut décomposer $L_1$ en $L_1 = L_1^+ \oplus L_1^-$, où $L_1^+ = L_1 [(\pi_i^+)^{e_i}]$, et $L_1^-$ est l'orthogonal de $L_1^+$. Le fait que $L_1$ soit un groupe de Barsotti-Tate tronqué d'échelon $1$ montre alors que les degrés de $L_1^+$ et $L_1^-$ sont entiers, donc que $Deg_i(x)$ est un multiple de $1/e_i$. \\
Dans le cas $2$, soit $F_i^{nr}$ est l'extension maximale non ramifiée contenue dans $F_i$, et de même pour $F_{0,i}^{nr}$. Alors le sous-groupe $L_1$ est muni d'une action de $O_{F_i^{nr}}$, et on peut donc définir les degrés partiels de $L_1$ pour cette action. Si $S$ est l'ensemble des plongements de $F_i^{nr}$ dans $\overline{\mathbb{Q}}_p$, alors le degré de $L_1$ est la somme des degrés partiels deg$_s L_1$, pour tout $s \in S$. Comme $L_1$ est un groupe de Barsotti-Tate tronqué d'échelon $1$, les degrés partiels sont tous entiers. De plus, si $s_0$ est un plongement de $F_{0,i}^{nr}$, $s$ et $\overline{s}$ les deux éléments de $S$ au-dessus de $s_0$, alors en raisonnant comme dans la remarque $\ref{part}$ on a par dualité
$$\text{deg}_s L_1 = \text{deg}_{\overline{s}} L_1$$
Le degré de $L_1$ est donc pair. Avec notre définition de la fonction degré, cela prouve que $Deg_i(x)$ appartient à $\frac{1}{e_i} \mathbb{Z}$. 
\end{proof}

\begin{rema} \label{mauvais}
La situation est plus compliquée dans le cas $3$, ce qui explique la différence dans le résultat. Cela impliquera une borne moins forte dans le résultat de classicité. Néanmoins, une analyse plus détaillée de ce cas pourrait peut-être permettre d'obtenir un résultat équivalent aux cas $1$ et $2$.
\end{rema}

Comme dans le cas précédent, on en déduit la proposition suivante. On note $e_i' = e_i$ dans les cas $1$ et $2$, et $e_i' = 2 e_i$ dans le cas $3$.

\begin{prop} \label{dynaA}
Soit $1 \leq i \leq h$, $k$ un entier compris entre $0$ et $e_i' f_i a_i-1$ et $0 < \alpha < \beta <1$ deux rationnels. Alors il existe un entier $N$ tel que
$$U_{\pi_i}^N \left( Deg_i^{-1} ( [\frac{k+\alpha}{e_i'},f_i a_i] ) \right) \subset Deg_i^{-1} ( [\frac{k+\beta}{e_i'}, f_i a_i ] )$$
\end{prop}

Nous allons maintenant définir une norme sur l'espace des formes modulaires. Rappelons que nous avons noté $\mathcal{A}_{nd(a+b)}$ le schéma sur $\mathbb{Z}_p$ paramétrant les variétés abéliennes de dimension $nd(a+b)$ munies d'une polarisation de degré premier à $p$ avec une structure de niveau $N$, et $\overline{\mathcal{A}}_{nd(a+b)}$ une compactification toroïdale de ce schéma. On se donne une identification de $\Sigma$ avec $\{1, \dots, d\}$. En particulier, on a des couples $(a_1,b_1), \dots, (a_d,b_d)$.

\begin{defi}
Soit $\widetilde{\mathcal{A}}_{nd(a+b)}$ l'espace de modules sur $\mathbb{Z}_p$ dont les $S$-points sont :
\begin{itemize}
\item un point $x \in \overline{\mathcal{A}}_{nd(a+b)}(S)$.
\item une filtration $0 = \omega_{A,0} \subset \omega_{A,1} \subset \dots \subset \omega_{A,2nd} = \omega_A$ telle que pour tout $1 \leq i \leq nd$, $\omega_{A,2i-1}/\omega_{A,2i-2}$ et $\omega_{A,2i}/\omega_{A,2i-1}$ sont localement des $\mathcal{O}_S$-facteurs directs de $\omega_A$ de rangs respectifs $a_i$ et $b_i$ ($a_i$ est égal à $a_j$ est $j$ est l'unique entier compris entre $1$ et $d$ et congru à $i$ modulo $d$, et de même pour $b_i$).
\end{itemize}
\end{defi}

L'espace $\widetilde{\mathcal{A}}_{nd(a+b)}$ est un schéma propre sur $\mathbb{Z}_p$. Soit 
$$\mathcal{T}_i = \text{Isom}_{\mathcal{O}_{\widetilde{\mathcal{A}}_{nd(a+b)}}} (\omega_{A,2i-1} / \omega_{A,2i-2} , \mathcal{O}_{\widetilde{\mathcal{A}}_{nd(a+b)}}^{a_i}) \oplus \text{Isom}_{\mathcal{O}_{\widetilde{\mathcal{A}}_{nd(a+b)}}} (\omega_{A,2i} / \omega_{A,2i-1} , \mathcal{O}_{\widetilde{\mathcal{A}}_{nd(a+b)}}^{b_i})$$
pour $1 \leq i \leq d$ (avec par convention $\omega_{A,0}=0$). On note $\phi_i : \mathcal{T}_i \to \widetilde{\mathcal{A}}_{nd(a+b)}$ la projection. L'espace $\mathcal{T}_i$ est un torseur sur $\widetilde{\mathcal{A}}_{ndg}$ pour le groupe GL$_{a_i} \times $GL$_{b_i}$. Si $\kappa_i=(k_j,l_j)$ est un élément de $\mathbb{Z}^{a_i} \times \mathbb{Z}^{b_i}$, on note $\omega_i^{\kappa_i} = {\phi_i}_* \mathcal{O}_{\mathcal{T}_i} [- \kappa_i']$, où $\kappa_i'=(k_{a_i+1-j}, l_{b_i+1-j})$, et où ${\phi_i}_* \mathcal{O}_{\mathcal{T}_i} [- \kappa_i']$ est le sous-faisceau de ${\phi_i}_* \mathcal{O}_{\mathcal{T}_i}$ où le tore de GL$_{a_i} \times $GL$_{b_i}$ agit par $-\kappa_i'$, et où le radical unipotent agit trivialement. \\
Rappelons que nous avons défini le poids d'une forme modulaire comme des couples $(k_{i,\sigma})_{1\leq i \leq a_\sigma, \sigma \in \Sigma}$ et $(l_{i,\sigma})_{1\leq i \leq b_\sigma, \sigma \in \Sigma}$, où $\Sigma$ est l'ensemble des plongements de $F$ dans $\overline{\mathbb{Q}}$, vérifiant $k_{1,\sigma} \geq \dots \geq k_{a_\sigma,\sigma}$ et $l_{1,\sigma} \geq \dots \geq l_{b_\sigma,\sigma}$, pour tout $\sigma \in \Sigma$. On note $\kappa_\sigma = ((k_{i,\sigma}),(l_{i,\sigma}))$ ; avec la numérotation faite sur $\Sigma$, on a donc des éléments $\kappa_i$ de $\mathbb{Z}^{a_i} \times \mathbb{Z}^{b_i}$. On note alors $\omega_0^\kappa$ le faisceau défini sur $\widetilde{\mathcal{A}}_{ndg}$ par $\omega_0^\kappa := \otimes_{i=1}^d \omega_i^{\kappa_i}$.\\
Soit $\widetilde{\mathcal{A}}_{nd(a+b)}^{rig}$ l'espace rigide associé à $\widetilde{\mathcal{A}}_{nd(a+b)} \times_{\mathbb{Z}_p} O_K$. Comme le schéma $\widetilde{\mathcal{A}}_{nd(a+b)}$ est propre, cet espace rigide est égal à $(\widetilde{\mathcal{A}}_{nd(a+b)} \times_{\mathbb{Z}_p} K)^{an}$. \\
$ $\\
Nous allons maintenant définir un morphisme de $\overline{X}_{Iw}$ vers $\widetilde{\mathcal{A}}_{nd(a+b)} \times_{\mathbb{Z}_p} K$. Le faisceau $\omega_A$ universel sur $\overline{X}_{Iw}$ est muni d'une action de $O_B$, et se décompose en $\omega_A = \oplus_{i=1}^h \omega_{A,i}$, où $\omega_{A,i}$ est un $O_{B,i}$-module. Rappelons que $O_{B,i}$ est égal à M$_n(O_{F_{0,i}}) \oplus $M$_n(O_{F_{0,i}})$ dans le cas $1$ et à M$_n(O_{F_i})$ dans les cas $2$ et $3$. Par équivalence de Morita, le faisceau $\omega_{A,i}$ est la somme de $n$ copies de $\omega_{A,i,0}$, om $\omega_{A,i,0}$ est un faisceau localement libre de rang $e_i f_i (a+b)$ muni d'une action de $O_{F_{0,i}} \oplus O_{F_{0,i}}$ dans le cas $1$, et de $O_{F_i}$ dans les cas $2$ et $3$. De plus, on peut décomposer ce dernier faisceau suivant les éléments de $\Sigma_i$ : 
$$\omega_{A,i,0} = \oplus_{\sigma \in \Sigma_i} (\omega_{A,i,0,\sigma,1} \oplus \omega_{A,i,0,\sigma,2} )$$
où $\omega_{A,i,0,\sigma,1}$  et $\omega_{A,i,0,\sigma,2}$ sont des faisceaux localement libres de rang $a_\sigma$ et $b_\sigma$. On obtient de cette manière une filtration du faisceau $\omega_A$.

\begin{defi}
On définit un morphisme $\psi : \overline{X}_{Iw} \to \widetilde{\mathcal{A}}_{nd(a+b)} \times_{\mathbb{Z}_p} K$ par la formule $x \to (\mathcal{P}(x), (\omega_{A,\bullet}))$, où $\mathcal{P}$ est le morphisme d'oubli de l'action de $O_B$ et de la structure Iwahorique, et où la filtration $(\omega_{A,\bullet})$ de $\omega_A$ est déduite de ce qui précède.
\end{defi}

On en déduit un morphisme $\psi : \overline{X}_{Iw}^{an} \to \widetilde{\mathcal{A}}_{nd(a+b)}^{rig}$. Au poids $\kappa$ nous avons associé un faisceau $\omega_0^\kappa$ sur $\widetilde{\mathcal{A}}_{nd(a+b)}$, à l'aide d'un choix de numérotation des places que l'on supposera compatible avec la filtration de $\omega_A$ construite précédemment. On a alors $\omega^\kappa = \psi^* \omega_0^\kappa$ comme faisceaux sur $\overline{X}_{Iw}$ et $\overline{X}_{Iw}^{an}$. Cela permet de disposer d'une structure entière pour le faisceau $\omega^\kappa$ (sur $\overline{X}_{Iw}^{an}$), et d'une norme sur l'espace $H^0(\mathcal{U},\omega^\kappa)$ pour tout ouvert $\mathcal{U}$ de $\overline{X}_{Iw}^{an}$.

\subsection{Classicité}

Nous prouvons dans cette partie le théorème de classicité. Enonçons tout d'abord le théorème. Nous avons noté $\Sigma$ l'ensemble des plongements de $F_0$ dans $\overline{\mathbb{Q}}$, $\Sigma_i$ l'ensemble des plongements de $F_{0,i}$ dans $\overline{\mathbb{Q}}_p$, de telle sorte que $\Sigma$ soit l'union disjointe des $\Sigma_i$. Notons également $F_{0,i}^{nr}$ l'extension maximale non ramifiée contenue dans $F_{0,i}$, et $S_i$ l'ensemble des plongements de $F_{0,i}^{nr}$ dans $\overline{\mathbb{Q}}_p$. Pour tout $s \in S_i$, on note $\Sigma_s$ l'ensemble des éléments de $\Sigma_i$ égaux à $s$ en restriction à $F_{0,i}^{nr}$. Ainsi, $\Sigma_i$ est égal à l'union disjointe des $\Sigma_s$ pour $s \in S_i$.

\begin{theo} \label{theogen_A}
Soit $f$ une forme surconvergente de poids $\kappa = (k_\sigma,l_\sigma)_{\sigma \in \Sigma}$, $k_\sigma = (k_{1,\sigma} \geq \dots k_{a_\sigma,\sigma})$ et $l_\sigma = (l_{1,\sigma} \geq \dots l_{b_\sigma,\sigma})$. On suppose que $f$ est propre pour les opérateurs $U_{\pi_i}$ de valeur propre $\alpha_i$ pour tout $i$ entre $1$ et $h$. Supposons que
$$d_i a_i b_i + e_i v(\alpha_i) < \inf_{s \in S_i} (\inf_{\sigma \in \Sigma_s} k_{a_\sigma,\sigma} + \inf_{\sigma \in \Sigma_s} l_{b_\sigma,\sigma})$$
dans les cas $1$ et $2$, et que 
$$d_i a_i b_i + e_i v(\alpha_i) < \inf_{\sigma \in \Sigma_i} (k_{a_\sigma,\sigma},l_{b_\sigma,\sigma})$$
dans le cas $3$. Alors $f$ est classique.
\end{theo}

Dans le cas $2$, la condition se réécrit donc
$$d_i \frac{(a+b)^2}{4} + e_i v(\alpha_i) < \inf_{s \in S_i} (\inf_{\sigma \in \Sigma_s} k_{a_\sigma,\sigma} + \inf_{\sigma \in \Sigma_s} l_{b_\sigma,\sigma})$$
puisque $a_i = b_i = (a+b)/2$. Remarquons également que la condition du théorème dans les cas $1$ et $2$ est impliquée par la condition suivante, plus forte :
$$d_i a_i b_i + e_i v(\alpha_i) < \inf_{\sigma \in \Sigma_i} k_{a_\sigma,\sigma} + \inf_{\sigma \in \Sigma_i} l_{b_\sigma,\sigma}$$
$ $\\
Passons maintenant à la démonstration du théorème. Soit $f$ une forme surconvergente vérifiant les hypothèses du théorème. Par définition, $f$ est une section de $\omega^\kappa$ sur $\overline{X}_{Iw,J}$ avec $J=~\prod_{i=1}^h [f_i a_i -~ \varepsilon, f_i a_i]$, pour un certain $\varepsilon >0$. Nous allons prolonger $f$ à $X_{Iw}^{an}$. la méthode est analogue à celle du cas des variétés de type $C$ : on va prolonger $f$ dans chaque direction, successivement. \\
Plus précisément, nous allons prolonger $f$ à Deg$^{-1}([0,f_1 a_1] \times \dots \times [f_h a_h-\varepsilon,f_h a_h]) \cap X_{Iw}^{an}$ en utilisant le fait que $f$ est propre pour $U_{\pi_1}$ et la relation vérifiée par la valeur propre $\alpha_{1}$. En répétant ce processus, on prolongera donc $f$ à tout $X_{Iw}^{an}$. \\
Soit donc $\mathcal{U}_I := $Deg$^{-1} ( I \times [f_2 a_2-\varepsilon,f_2 a_2] \times \dots \times [f_h a_h-\varepsilon, f_h a_h]) \cap X_{Iw}^{an}$, pour tout intervalle $I$ de $[0,f_1 a_1]$. La forme $f$ est définie sur $\mathcal{U}_{[f_1 a_1-\varepsilon, f_1 a_1]}$. En utilisant la relation $f = \alpha_{1}^{-m} U_{\pi_1}^m f$ pour tout $m \geq 1$ et la proposition \ref{dynaA}, on peut donc prolonger $f$ à $\mathcal{U}_{]f_1 a_1-1/e_1',f_1 a_1]}$. \\
Soit $\mathcal{U}:=\mathcal{U}_{[0,f_1 a_1-1/e_1'+\beta]}$, avec $\beta$ un rationnel strictement positif, que l'on prendra arbitrairement petit. On peut décomposer les opérateurs de Hecke sur cet espace, et obtenir des opérateurs $U_{\pi_1,j}^{good}$ et $U_{\pi_1,j}^{bad}$. On peut alors former les séries de Kassaei attachées à cette décomposition. Le fait que ces séries se recolleront découle alors de la proposition suivante.

\begin{prop} \label{lemnorm_A}
Soit $T$ un opérateur égal à un certain $U_{\pi_1,j}^{bad}$. On suppose que l'image de cet opérateur est incluse dans $\mathcal{U}_{[0,f_1 a_1 - c]}$ pour un certain $c \geq 0$. Alors
$$ \Vert T \Vert_\mathcal{U} \leq p^{N_i-c M}$$
avec $M=\inf_{s \in S_1} (\inf_{\sigma \in \Sigma_s} k_{a_\sigma,\sigma} + \inf_{\sigma \in \Sigma_s} l_{b_\sigma,\sigma})$ dans les cas $1$ et $2$, et $M=2 \inf_{\sigma \in \Sigma_1} (k_{a_\sigma,\sigma},l_{b_\sigma,\sigma})$ dans le cas $3$.
\end{prop}

\begin{proof}
En raisonnant comme dans la proposition \ref{lemnorm}, il suffit de majorer la norme du morphisme $\omega_{A/L}^\kappa \to \omega_{A}^\kappa$ pour un schéma semi-abélien $A$ définie sur $O_M$ ($M$ est une extension finie de $\mathbb{Q}_p$), et $L$ un sous-groupe de $A[\pi_1]$ totalement isotrope maximal stable par $O_B$ pouvant intervenir dans la définition de l'opérateur de Hecke. On se ramène au cas où $n=1$, $A$ est un schéma abélien, et $\pi_1$ est la seule place au-dessus de $p$. On a un morphisme naturel $q : \omega_{A/L} \to \omega_A$. Distinguons maintenant les différents cas. \\
Dans le cas $1$, la place $\pi_1$ est décomposée dans $F$. Le $O_M \otimes_{\mathbb{Z}} O_F$-module $\omega$ se décompose donc en $\omega_A = \omega_A^+ \oplus \omega_A^-$, où $\omega_A^+$ et $\omega_A^-$ sont des $O_M \otimes_{\mathbb{Z}_p} O_{F_{0,1}}$-modules. De même pour $\omega_{A/L}$. \\
De plus, le module $\omega_A^+$ se décompose en somme directe par les éléments de $S_1$ : $\omega_A^+ = \oplus_{s \in S_1} \omega_{A,s}^+$. On obtient décomposition analogue pour $\omega_{A/L}^+$. Le morphisme $q$ respecte cette filtration ; soit $\lambda_s$ le déterminant du morphisme $\omega_{A/L,s}^{+} \to \omega_{A,s}^{+}$. La valuation de $\lambda_s$ est égale au degré partiel deg$_s L^+$, où $L^+ = L[\pi_1^+]$. Si $\omega_A^{\kappa,+}$ désigne le module associé au poids $((k_{ij}),(0))$, et $q^{\kappa,+} : \omega_{A/L}^{\kappa,+} \to \omega_A^{\kappa,+}$ le morphisme induit  alors on a
$$ || q^{\kappa,+}|| \leq p^{- \sum_{s \in S_1} (\text{deg}_s L^+ \inf_{\sigma \in \Sigma_s} k_{a_\sigma,\sigma} )     }    $$
De même, par dualité, si on note $\omega_A^{\kappa,-}$ désigne le module associé au poids $((0),(l_{ij}))$, et $q^{\kappa,-} : \omega_{A/L}^{\kappa,-} \to \omega_A^{\kappa,-}$ le morphisme induit, on a
$$ || q^{\kappa,-}|| \leq p^{- \sum_{s \in S_1} (\text{deg}_s L^- \inf_{\sigma \in \Sigma_s} l_{b_\sigma,\sigma} )     }    $$
où $L^- = L[\pi_1^-]$. Or par dualité, les degrés partiels de $L^+$ et $L^-$ son égaux. De plus, comme par hypothèse le degré de $L$ est supérieur à $2c$, on a $\sum_{s \in S_1}$ deg$_s L^+  \geq c$. On obtient que la norme du morphisme $q^\kappa : \omega_{A/L}^\kappa \to \omega_{A}^\kappa$ est majorée par
$$ || q^{\kappa}|| \leq p^{- \sum_{s \in S_1} \text{deg}_s L^+ (\inf_{\sigma \in \Sigma_s} k_{a_\sigma,\sigma} +  \inf_{\sigma \in \Sigma_s} l_{b_\sigma,\sigma} )} \leq p^{ -c M   }    $$
Dans le cas $2$, soit $F_{0,1}^{nr}$ l'extension maximale non ramifiée contenue dans $F_{0,1}$, et de même pour $F_1^{nr}$. Soit $S_1'$ l'ensemble des plongements de $F_1^{nr}$ dans $\overline{\mathbb{Q}}_p$. Le $O_M \otimes_{\mathbb{Z}_p} O_{F_1}$-module $\omega_A$ se décompose donc en somme directe suivant les éléments de $S_1$. On se ramène au cas où $S_{0,1}$ n'a qu'un seul élément, i.e. $p$ est totalement ramifié dans $F_{0,1}$. On peut alors décomposer le module $\omega_A$ en $\omega_A^+ \oplus \omega_A^-$. La majoration est alors identique au calcul précédent. \\
Dans le cas $3$, puisque le déterminant de $q : \omega_{A/L} \to \omega_A$ est de valuation égale au degré de $L$, qui est supérieur à $2c$, on obtient directement
$$ || q^{\kappa}|| \leq   p^{ -2c \inf_{\sigma \in \Sigma_1} (k_{a_\sigma,\sigma},l_{b_\sigma,\sigma}) }    $$
\end{proof}

\begin{rema}
Les majorations de norme dans le cas $3$ sont plus difficiles, mais nous pensons qu'il devrait être possible d'améliorer la borne obtenue dans ce cas $3$ par celle (meilleure) obtenue dans les autres cas. Avec la remarque $\ref{mauvais}$, cela permettrait d'améliorer la borne dans le résultat de classicité, et au final avoir un critère uniforme dans chacun des cas $1$, $2$ et $3$.
\end{rema}

Comme les opérateurs $U_{\pi_1,j}^{bad}$ sont à valeurs dans $\mathcal{U}_{[f_1 a_1 - (1/e_1 - \beta ),f_1 a_1]}$, cela montre que si $\beta$ est choisi suffisamment petit, les opérateurs $\alpha_{1}^{-1} U_{\pi_1,j}^{bad}$ sont tous de norme strictement inférieure à $1$. Les séries définies se recolleront donc, et permettent d'étendre $f$ à $\mathcal{U}_{[0,f_1 a_1]}$. \\
En itérant ce raisonnement, on prolonge $f$ à $X_{Iw}^{an}$, c'est-à-dire un élément de $H^0(X_{Iw}^{an},\omega^\kappa)$. De plus, on démontre que cette fonction est bornée sur $X_{Iw}^{an}$. Le théorème d'extension $\ref{extension}$ montre que $f$ s'étend à la compactification, soit $f \in H^0(\overline{X}_{Iw}^{an},\omega^\kappa)$. Un principe GAGA (voir \cite{EGA2} partie $5.1$) montre alors que $f$ est une forme modulaire algébrique, c'est-à-dire provient d'un élément de l'espace $H^0(\overline{X}_{Iw},\omega^\kappa)$.

\section{Cas d'un niveau arbitraire en $p$} \label{partie_gen}

Nous montrons dans cette section que les résultats obtenus se généralisent à des variétés de Shimura avec des structures de niveau en $p$ plus générales. Remarquons que, même dans le cas où $p$ est non ramifié dans le corps $F$, on ne sait pas construire de bons modèles entiers pour les variétés. La situation est donc plus compliquée que le cas précédent, où le problème venait simplement de l'abscence de modèle entier pour les compactifications.

\subsection{Définitions}

Soit $(F,F_0,B,\star)$ une donnée de Shimura de type (A) ou (C) comme définie précédemment, et $(U_{\mathbb{Q}},\langle,\rangle)$ un $B$-module hermitien non dégénéré. Soient également un ordre $O_B$ de $B$ stable par $\star$, et un réseau $U$ de $U_\mathbb{Q}$ tel que l'accouplement $\langle,\rangle$ restreint à $U\times U$ soit à valeurs dans $\mathbb{Z}$. On supposera que les hypothèses faites dans le cas (A) ou (C) sur $O_B$ et $U$ sont vérifiées. Introduisons maintenant la variété de niveau plus général en $p$. On rappelle que $X$ désigne la variété de Shimura sans niveau en $p$. Nous avons introduit un corps $K$ suffisamment grand dans les cas $(A)$ et $(C)$, qui est une extension finie de $\mathbb{Q}_p$. Soit $m \geq 1$ un entier. On suppose que $K$ contient suffisamment de racines $p^r$-ièmes de l'unité, de telle sorte que les caractères de $(O_F / \pi_i^m)$ dans $\overline{\mathbb{Q}}_p^\times$ sont à valeurs dans $K^\times$.

\begin{defi}
Soit $X_{0,m}$ l'espace de modules sur Spec$(K)$ dont les $S$-points sont les classes d'isomorphismes des $(A,\lambda,\iota,\eta,H_{\bullet})$ où
\begin{itemize}
\item $(A,\lambda,\iota,\eta) \in X(S)$
\item Dans le cas (C), pour tout $1 \leq i \leq h$, $0 \subset H_{i,1} \subset \dots \subset H_{i,g}$ est un drapeau de $A[\pi_i^m]$, chaque $H_{i,j}$ étant totalement isotrope, stable par $O_B$ et isomorphe localement pour la topologie étale à $(O_F / \pi_i^m O_F)^{nj}$.
\item Dans le cas (A) et si $\pi_i$ est dans le cas $1$, $0 \subset H_{i,1} \subset \dots \subset H_{i,a+b}$ est un drapeau de $A[(\pi_i^+)^m]$, chaque $H_{i,j}$ étant stable par $O_B$ et isomorphe localement pour la topologie étale à $(O_F / (\pi_i^+)^m O_F)^{nj}$.
\item Dans le cas (A) et si $\pi_i$ est dans le cas $2$ ou $3$, $0 \subset H_{i,1} \subset \dots \subset H_{i,(a+b)/2}$ est un drapeau de $A[\pi_i^m]$, chaque $H_{i,j}$ étant totalement isotrope, stable par $O_B$ et isomorphe localement pour la topologie étale à $(O_F / (\pi_i)^m O_F)^{nj}$.
\end{itemize}
\end{defi}

Si $A$ est un schéma abélien avec action de $O_B$, et si $\mathfrak{m}$ est un idéal de $O_F$, on dit qu'un point $P$ est d'ordre exactement $\mathfrak{m}^N$ si $\mathfrak{m}^N \cdot P=0$ et $\mathfrak{m}^{N-1} \cdot P \neq 0$.

\begin{defi} \label{defpoints}
Soit $X_{1,m}$ l'espace de modules sur Spec$(K)$ dont les $S$-points sont les classes d'isomorphismes des $(A,\lambda,\iota,\eta,P_{\bullet})$ où
\begin{itemize}
\item $(A,\lambda,\iota,\eta) \in X(S)$
\item Dans le cas (C), pour tout $1 \leq i \leq h$, $P_{i,1}, \dots, P_{i,g}$ sont des points de $A[\pi_i^m]$ d'ordre exactement $\pi_i^m$, orthogonaux entre eux, tels que le sous-groupe engendré par $P_{i,j}$ soit isomorphe localement pour la topologie étale à $(O_F / \pi_i^m O_F)^n$.
\item Dans le cas (A) et si $\pi_i$ est dans le cas $1$, $P_{i,1}, \dots, P_{i,a+b}$ sont des points de $A[(\pi_i^+)^m]$ d'ordre exactement $(\pi_i^+)^m$, tels que le sous-groupe engendré par $P_{i,j}$ est isomorphe localement pour la topologie étale à $(O_F / (\pi_i^+)^m O_F)^n$.
\item Dans le cas (A) et si $\pi_i$ est dans le cas $2$ ou $3$, $P_{i,1}, \dots, P_{i,(a+b)/2}$ sont des points de $A[\pi_i^m]$ d'ordre exactement $\pi_i^m$, orthogonaux entre eux, tels que le sous-groupe engendré par $P_{i,j}$ soit isomorphe localement pour la topologie étale à $(O_F / \pi_i^m O_F)^n$. 
\end{itemize}
\end{defi}

\begin{rema}
Plaçons-nous dans le cas (C), et soit $P$ un point de $A[\pi_i^m]$ d'ordre exactement $\pi_i^m$. La condition que le sous-groupe engendré par $P$ est isomorphe localement pour la topologie étale à $(O_F / \pi_i^m O_F)^n$ se reformule de la manière suivante. Le groupe $A[\pi_i^m]$ est muni d'une action de $M_n(O_F / \pi_i^m)$ ; soit $E_{j,k}$ la base traditionnelle de cet anneau comme $(O_F / \pi_i^m)$-module. La condition précédente est alors équivalente aux relations $E_{j,k} \cdot P = E_{j,l} \cdot P$ pour tout $j,k,l$. De même dans le cas (A).
\end{rema}

On dispose d'une application naturelle $F : X_{1,m} \to X_{0,m}$. Dans le cas (C), $F$ envoie $(A,\lambda,\iota,\eta,P_{\bullet})$ sur $(A,\lambda,\iota,\eta,H_{\bullet})$, où $H_{i,j}$ est le sous-groupe de $A[\pi_i^m]$ engendré par $P_{i,1}, \dots, P_{i,j}$, pour tout $1 \leq i \leq h$ et $1 \leq j \leq g$. L'application $F$ est un revêtement étale de groupe $\mathcal{G}=\prod_{i=1}^h \mathcal{G}_i$, avec $\mathcal{G}_i = B_g(O_F /\pi_i^m)$, où $B_g$ désigne le Borel supérieur de $GL_g$. De même dans le cas (A) : $F$ est alors un revêtement étale de groupe $\mathcal{G}=\prod_{i=1}^h G_i$, où $\mathcal{G}_i = B_{a+b}(O_F / (\pi_i^+)^m)$ dans le cas $1$, $\mathcal{G}_i = B_{(a+b)/2} ( O_F / \pi_i^m)$ dans le cas $2$ ou $3$. \\
Soit $\chi$ un caractère du tore de $\mathcal{G}$ que l'on voit comme un caractère de $\mathcal{G}$ ; le faisceau $F_* \mathcal{O}_{X_{1,m}}$ est muni d'une action de $\mathcal{G}$, et on note $\mathcal{O}_{X_{1,m}} (\chi) = F_* \mathcal{O}_{X_{1,m}} [\chi]$ le sous-faisceau où $\mathcal{G}$ agit par $\chi$. C'est un faisceau inversible sur $X_{0,m}$. \\
Le faisceau des formes modulaires de poids $\kappa$ et de nebentypus $\chi$ est le faisceau 
$$\omega^\kappa (\chi) := \omega^\kappa \otimes_{\mathcal{O}_{X_{0,m}}} \mathcal{O}_{X_{1,m}} (\chi)$$

Soient $\overline{X}_{1,m}$ et $\overline{X}_{0,m}$ des compactifications toroïdales respectivement de $X_{1,m}$ et $X_{0,m}$. On suppose que les compactifications sont construites de telle sorte que le morphisme $F$ s'étende en $F : \overline{X}_{1,m} \to \overline{X}_{0,m}$ (ce qui est possible d'après le théorème $\ref{morphcompact}$) ; le faisceau $\omega^\kappa (\chi)$ s'étend donc sur $\overline{X}_{0,m}$. L'espace des formes modulaires de poids $\kappa$ et de nebentypus $\chi$ est donc l'espace H$^0(\overline{X}_{0,m},\omega^\kappa (\chi))$. On notera $\overline{X}_{1,m}^{an}$, $\overline{X}_{0,m}^{an}$, $X_{1,m}^{an}$ et $X_{0,m}^{an}$ les espaces analytiques associés respectivement à $\overline{X}_{1,m}$, $\overline{X}_{0,m}$, $X_{1,m}$ et $X_{0,m}$. \\

Définissons maintenant les opérateurs de Hecke $U_{\pi_i}$, pour $1 \leq i \leq h$. Soit $C_i$ l'espace de modules sur $K$ paramétrant un point $x=(A,\lambda,\iota,\eta,H_{\bullet})$ de $X_{0,m}$, et un sous-groupe $L$, supplémentaire générique de $H_{i,D}[p]$ dans $A[p]$ avec $D$ qui vaut $g$,$a_i$, ou $(a+b)/2$ suivant les cas (dans le cas (A) et $\pi_i$ dans le cas $1$, $L= L_0 \oplus L_0^\bot$, avec $L_0$ un supplémentaire de $H_{i,a_i}$ dans $A[\pi_i^+]$). On dispose de deux applications $p_1,p_2 : C_i \to X_{0,m}$, où $p_1$ est l'oubli de $L$, et $p_2$ est le quotient par $L$. Soit $\overline{C}_i$ une compactification toroïdale de $C_i$ telle que les morphismes $p_1$ et $p_2$ s'étendent en des morphismes $\overline{C}_i \to \overline{X}_{0,m}$, et $\overline{C}_i^{an}$ l'espace analytique associé.

\begin{defi}
L'opérateur géométrique agissant sur les parties de $\overline{X}_{0,m}^{an}$ est défini par $U_{\pi_i} (S) = p_2(p_1^{-1} (S))$, pour toute partie $S$ de $\overline{X}_{0,m}^{an}$.
\end{defi}

Cet opérateur respecte les ouverts de $X_{0,m}^{an}$ (puisque les morphismes $p_1$ et $p_2$ sont finis étales sur $X_{0,m}^{an}$), mais pas ceux de $\overline{X}_{0,m}^{an}$ en général. Définissons maintenant l'opérateur de Hecke agissant sur les formes modulaires avec nebentypus. Nous devons pour cela définir un morphisme $p_2^* \omega^\kappa (\chi) \to p_1^* \omega^\kappa (\chi)$. Nous avons déjà défini un morphisme $p_2^* \omega^\kappa \to p_1^* \omega^\kappa$ à l'aide de l'isogénie universelle $A \to A/L$ sur $C_i$. Nous allons donc définir un morphisme $p_2^* \mathcal{O}_{X_{1,m}} (\chi) \to p_1^* \mathcal{O}_{X_{1,m}} (\chi)$. \\
Soit $C_{i,1}$ l'espace de modules paramétrant un point $(A,\lambda,\iota,\eta,H_{\bullet},L)$ de $C_i$ et des points $P_\bullet$ de $A[p^\infty]$ comme dans la définition $\ref{defpoints}$ tel que $F(A,\lambda,\iota,\eta,P_{\bullet}) = (A,\lambda,\iota,\eta,H_{\bullet})$. Soit $C_{i,1}'$ l'espace de modules paramétrant un point $(A,\lambda,\iota,\eta,H_{\bullet},L)$ de $C_i$ et des points $P_\bullet'$ de $(A/L)[p^\infty]$ comme dans la définition $\ref{defpoints}$ tel que $F(A/L,\lambda',\iota',\eta',P_{\bullet}') = (A/L,\lambda',\iota',\eta',H_{\bullet}')$. On dispose des morphismes de projection $q : C_{i,1} \to C_i$ et $q' : C_{i,1}' \to C_i$, qui consistent à oublier les points $P_\bullet$. Les morphismes $q$ et $q'$ sont des revêtements étales de groupe $\mathcal{G}_i$, et on a $p_1^* \mathcal{O}_{X_{1,m}} (\chi) = q_* \mathcal{O}_{C_{i,1}} (\chi)$ et $p_2^* \mathcal{O}_{X_{1,m}} (\chi) = q'_* \mathcal{O}_{C_{i,1}'} (\chi)$. \\
De plus, on a un isomorphisme naturel $C_{i,1} \simeq C_{i,1}'$, défini par $(A,\lambda,\iota,\eta,H_{\bullet},L,P_\bullet) \to (A,\lambda,\iota,\eta,H_{\bullet},L,P_\bullet')$, où les points $P_\bullet'$ sont les images des points $P_\bullet$ dans $A/L$. On en déduit donc un isomorphisme naturel entre $q_* \mathcal{O}_{C_{i,1}} (\chi)$ et $q'_* \mathcal{O}_{C_{i,1}'} (\chi)$, donc entre $p_1^* \mathcal{O}_{X_{1,m}} (\chi)$ et $p_2^* \mathcal{O}_{X_{1,m}} (\chi)$. \\
On a donc un morphisme $q^*(\kappa) (\chi) : p_2^* \omega^\kappa (\chi) \to p_1^* \omega^\kappa (\chi)$. Pour tout ouvert $\mathcal{U}$ de $X_{0,m}^{an}$, nous pouvons donc former le morphisme composé
\begin{displaymath}
\widetilde{U}_{\pi_i} :   H^0(U_{\pi_i}(\mathcal{U}),\omega^\kappa (\chi)) \to H^0 ( p_1^{-1} (\mathcal{U}), p_2^* \omega^\kappa (\chi)) \overset{q^*(\kappa) (\chi)}{\to} H^0(p_1^{-1}(\mathcal{U}) , p_1^* \omega^\kappa (\chi)) \overset{Tr_{p_1}}{\to}  H^0(\mathcal{U},\omega^\kappa (\chi))
\end{displaymath}

\begin{defi}
L'opérateur de Hecke agissant sur les formes modulaires est alors défini par $U_{\pi_i} = \frac{1}{p^{N_i}} \widetilde{U}_{\pi_i}$ avec $N_i$ le facteur de normalisation défini dans les parties précédentes. 
\end{defi}

\subsection{Degré et normes}

Nous allons maintenant définir la fonction degré sur $X_{0,m}^{an}$. Heureusement, nous allons utiliser la fonction degré que l'on a définie précédemment sur l'espace de niveau Iwahorique. Si $x=(A,\lambda,\iota,\eta,H_{\bullet})$ est un point de $X_{0,m}$, on rappelle que $H_{i,j}$ est un sous-groupe de $A[\pi_i^m]$ dans le cas (C), et de $A[(\pi_i^+)^m]$ ou $A[\pi_i^m]$ dans le cas (A), suivant les cas. On note $H_{i}^{(m-1)}$ le sous-groupe de $A[p^\infty]$ égal à $H_{i,g}[\pi_i^{m-1}]$ dans le cas (C), à $H_{i,a_i}[(\pi_i^+)^{m-1}] \oplus H_{i,a_i}^\bot[(\pi_i^-)^{m-1}]$ dans le cas (A)-$1$, et à $H_{i,(a+b)/2}[\pi_i^{m-1}]$ dans les cas (A)-$2$ et (A)-$3$. On note alors $H^{(m-1)}$ le sous-groupe de $A[p^\infty]$ engendré par les $H_{i}^{(m-1)}$, pour tout $i$. \\
On définit alors un morphisme $G : X_{0,m} \to X_{Iw}$ par $(A,\lambda,\iota,\eta,H_{\bullet}) \to (A / H^{(m-1)},\lambda',\iota',\eta',H_{\bullet}')$, où $H_{i,j}'$ est l'image de $H_{i,j}$ dans $A / H^{(m-1)}$, sauf éventuellement dans le cas (A)-$1$, où $H_{i,j}'$ est égal à l'image de $H_{i,j} \cap (\pi_i^+)^{-1} H_i^{(m-1)}$ dans $A/ H^{(m-1)}$ pour $j > a_i$. On vérifie alors que $H_{i,j}'$ est un sous-groupe de $A[\pi_i]$ ou $A[\pi_i^+]$ suivant les cas, et que $(A / H^{(m-1)},\lambda,\iota,\eta,H_{\bullet}')$ définit bien un point de $X_{Iw}$. De plus, on peut supposer que les compactifications toroïdales sont construites de telle sorte que le morphisme $G$ s'étende en $G : \overline{X}_{0,m} \to \overline{X}_{Iw}$. On notera encore $G$ le morphisme analytifié $\overline{X}_{0,m}^{an} \to \overline{X}_{Iw}^{an}$.

\begin{defi}
On définit la fonction degré sur $\overline{X}_{0,m}^{an}$ par Deg$ : \overline{X}_{0,m}^{an} \to \prod_{i=1}^h [0 , f_i a_i]$, $x \to $ Deg$ (G(x))$.
\end{defi}

L'entier $a_i$ est défini dans la partie \ref{hecke_cas_A} dans le cas (A), et on note $a_i =  g$ dans le cas (C). On notera également Deg$_i$ la $i$-ième composante de la fonction Deg. Pour démontrer les propriétés de la fonction Deg, nous allons nous ramener à l'espace $X_{Iw}^{an}$ à l'aide de l'application $G$. Les opérateurs $U_{\pi_i}$ ne commutent pas avec $G$, mais nous avons néanmoins la propriété suivante.

\begin{prop}
Soit $1 \leq i \leq h$, et $x \in X_{0,m}^{an}$. Alors $G(U_{\pi_i}(x)) \subset U_{\pi_i} ( G(x))$.
\end{prop}

\begin{proof}
Pour simplifier les notations, plaçons-nous dans le cas (C). Soit $x=(A,\lambda,\iota,\eta,H_{\bullet})$ le point de $X_{0,m}^{an}$ et $L$ un sous-groupe totalement isotrope de $A[\pi_i]$, supplémentaire générique de $H_{i,g}$. Soit $y =(A/L,\lambda',\iota',\eta',H_{\bullet}')$ le point de $U_{\pi_i}(x)$ correspondant à $L$. Le sous-groupe $H_{i,j}'$ est donc égal à $(H_{i,j} + L) / L$, pour tout $1 \leq j \leq g$. Le point $z=G(y) \in X_{Iw}^{an}$ est obtenu en quotientant la variété abélienne $A/L$ par $H'^{(m-1)} = \oplus_{k=1}^h H_k'^{(m-1)}$, avec $H_k'^{(m-1)} = H_{k,g}' [\pi_k^{m-1}]$ pour tout $k$. Comme $H'^{(m-1)}$ est l'image de $H^{(m-1)}$ dans $A/L$, on a $z=(A / (L + H^{(m-1)}) ,\lambda'',\iota'',\eta'',H_{\bullet}'')$, avec $H_{j,k}''$ égal à l'image de $H_{j,k}$ dans $A / (L + H^{(m-1)})$. \\
D'un autre côté, le point $G(x) \in X_{Iw}^{an}$ est égal à $(A / H^{(m-1)},\lambda^0,\iota^0,\eta^0,H_{\bullet}^0)$, avec $H_{j,k}^0$ égal à l'image de $H_{j,k}$ dans $A/ H^{(m-1)}$. En particulier, $H_{i,g}^0 = H_{i,g} / H_{i,g}[\pi_i^{m-1}]$. Soit $L_0$ l'image de $L$ dans $A/ H^{(m-1)}$. On voit alors facilement que $L$ est un supplémentaire générique de $H_{i,g}^0$ : cela découle de l'égalité $(L + H_{i,g}[\pi_i^{m-1}]) \cap H_{i,g} = H_{i,g}[\pi_i^{m-1}]$. On a donc défini un point de $U_{\pi_i} (G(x))$. Or ce point est égal à $z$, donc on a bien $G(y) \in U_{\pi_i}(G(x))$. \\
La démonstration est analogue dans les autres cas.
\end{proof}

Cette proposition nous permet d'en déduire les propriétés de la fonction degré à partir de celles démontrées sur $X_{Iw}^{an}$.

\begin{coro}
Soit $1 \leq i \leq h$, $x \in X_{0,m}^{an}$ et $y \in U_{\pi_i} (x)$. Soit $x_j= Deg_j(x)$, et $y_j=Deg_j(y)$ pour $1 \leq j \leq h$. Alors
\begin{itemize}
\item $y_j=x_j$ pour $j \neq i$.
\item $y_i \geq x_i$
\end{itemize}
De plus, s'il existe $y \in U_{\pi_i}^{2e_i} (x)$ avec $Deg_i(y)=Deg_i(x)$, alors $x_i \in \frac{1}{e_i'} \mathbb{Z}$. \\
Soit $k$ un entier compris entre $0$ et $e_i' f_i a_i-1$ et $0 < \alpha < \beta <1$ deux rationnels. Alors il existe $\varepsilon > 0$ tel que $Deg_i (y) \geq Deg_i(x) + \varepsilon$, pour tout $x \in Deg_i^{-1} ([\frac{k+\alpha}{e_i'} , \frac{k+ \beta}{e_i'} ])$ et $y \in U_{\pi_i}^{2e_i} (x)$.
\end{coro}

La fonction degré nous permet de définir les formes modulaires surconvergentes sur $X_{0,m}$. On définit le lieu ordinaire-multiplicatif comme $\overline{X}_{0,m}^{mult} := $ Deg$^{-1} (\{f_1 a_1\} \times \dots \times \{f_h a_h \})$.

\begin{defi}
L'espace des formes modulaires surconvergentes est défini par
$$H^0(\overline{X}_{0,m}^{an},\omega^\kappa (\chi))^\dagger := \text{colim}_\mathcal{V} H^0 (\mathcal{V},\omega^\kappa (\chi))$$
où la colimite est prise sur les voisinages stricts $\mathcal{V}$ de $\overline{X}_{0,m}^{mult}$ dans $\overline{X}_{0,m}^{an}$.
\end{defi}

Une forme modulaire surconvergente sur $X_{0,m}$ est donc définie sur un espace du type Deg$^{-1} ([f_1 a_1 - \varepsilon, f_1 a_1] \times \dots \times [f_h a_h - \varepsilon, f_h a_h])$. 

\begin{rema}
Plaçons-nous sur le lieu de bonne réduction de $X_{0,m}^{an}$, et dans le cas (C) par exemple. On a donc un schéma abélien $A$ défini sur $O_L$, l'anneau des entiers d'une extension finie de $\mathbb{Q}_p$. Le sous-groupe $H_{i,g}$ de $A[\pi_i^m]$ est totalement isotrope, et est isomorphe localement pour la topologie étale à $(O_B / \pi_i^m O_B)^g$. Pour tout $1 \leq r \leq m-1$, on a un morphisme
$$H_{i,g} / H_{i,g} [\pi_i^{m-1}] \overset{\pi_i^r}{\longrightarrow} H_{i,g} [\pi_i^{m-r}] / H_{i,g} [\pi_i^{m-r-1}]$$
qui est un isomorphisme en fibre générique. On en déduit par les propriétés de la fonction degré (voir \cite{Fa}), que deg $H_{i,g} / H_{i,g} [\pi_i^{m-1}] \leq $ deg $H_{i,g} [\pi_i^{m-r}] / H_{i,g} [\pi_i^{m-r-1}]$ pour tout $1 \leq r \leq m-1$. On voit donc que si le degré de $H_{i,g} / H_{i,g} [\pi_i^{m-1}]$ est maximal, il en est de même de $H_{i,g} [\pi_i^{m-r}] / H_{i,g} [\pi_i^{m-r-1}]$ pour tout $1 \leq r \leq m-1$, et donc le degré de $H_{i,g}$ est maximal. Cela justifie notre définition du lieu ordinaire-multiplicatif et des formes modulaires surconvergentes.
\end{rema}

$ $\\
Nous allons maintenant définir une norme sur l'espace $H^0(\mathcal{U},\omega^\kappa (\chi))$, pour tout ouvert $\mathcal{U}$ de $\overline{X}_{0,m}^{an}$. Cela revient à trouver un modèle entier pour le faisceau $\omega^\kappa (\chi)$. Puisque nous avons déjà défini un modèle entier $\widetilde{\omega}^{\kappa}$ du faisceau $\omega^\kappa$, il nous suffit de définir un modèle entier du faisceau inversible $\mathcal{O}_{\overline{X}_{1,m}^{an}} (\chi)$. On rappelle que $\mathcal{O}_{\overline{X}_{1,m}^{an}} (\chi) = F_* \mathcal{O}_{\overline{X}_{1,m}^{an}} [\chi]$, où $F$ est le morphisme $\overline{X}_{1,m}^{an} \to \overline{X}_{0,m}^{an}$. Le faisceau structural $\mathcal{O}_{\overline{X}_{1,m}^{an}}$ est canoniquement muni d'un modèle entier $\widetilde{\mathcal{O}}_{\overline{X}_{1,m}^{an}}$. On définit alors

$$\widetilde{\mathcal{O}}_{\overline{X}_{1,m}^{an}} (\chi) := F_* \widetilde{\mathcal{O}}_{\overline{X}_{1,m}^{an}} [\chi]$$

Le sous-faisceau $\widetilde{\mathcal{O}}_{\overline{X}_{1,m}^{an}} (\chi)$ définit donc un modèle entier pour le faisceau $\mathcal{O}_{\overline{X}_{1,m}^{an}} (\chi)$. On définit
$$\widetilde{\omega}^\kappa (\chi) := \widetilde{\omega}^\kappa \otimes_{\widetilde{\mathcal{O}}_{\overline{X}_{0,m}^{an}}} \widetilde{\mathcal{O}}_{\overline{X}_{1,m}^{an}} (\chi)$$

C'est un sous-faisceau de $\omega^\kappa (\chi)$, et cela nous permet de définir une norme sur $H^0(\mathcal{U},\omega^\kappa (\chi))$, pour tout ouvert $\mathcal{U}$ de $\overline{X}_{0,m}^{an}$, et donc sur les opérateurs agissant sur ces espaces.

\subsection{Classicité}

Les méthodes développées dans les parties précédentes permettent de démontrer un théorème de classicité pour les formes modulaires surconvergentes dont le poids est grand devant la pente.

\begin{theo} \label{theogen_gen}
Soit $f$ une forme modulaire surconvergente sur $X_{0,m}$, de poids $\kappa$ et de nebentypus $\chi$. On suppose que $f$ est propre pour les opérateurs $U_{\pi_i}$, de valeurs propres $\alpha_i$, et que le poids $\kappa$ est grand devant les valuations des $\alpha_i$ au sens du théorème $\ref{theogen}$ dans le cas (C), et du théorème $\ref{theogen_A}$ dans le cas (A). Alors $f$ est classique.
\end{theo}

La démonstration est entièrement analogue à celle des théorèmes précédents. La forme modulaire $f$ est définie sur une partie de $\overline{X}_{0,m}^{an}$. On prolonge $f$ à $X_{0,m}^{an}$ : pour ce faire, on prolonge $f$ dans chaque direction en utilisant le fait que $f$ est propre pour $U_{\pi_i}$. On utilise tout d'abord le fait que $f$ se prolonge automatiquement à certaine zone du type Deg$_i > f_i a_i - 1/e_i'$. Ensuite, on décompose l'opérateur de Hecke sur $U_{\pi_i}$ sur la zone restante, construit les séries de Kassaei, et recolle celles-ci pour prolonger $f$ à la zone restante. On applique ensuite le principe d'extension (théorème $\ref{extension}$) pour conclure que $f$ est classique. \\
La seule proposition à prouver est le fait que les opérateurs $\alpha_i^{-1} U_{\pi_i}^{bad}$ définis sont bien de norme strictement inférieure à $1$.

\begin{prop}
Soit $T$ un opérateur égal à un certain $U_{\pi_i}^{bad}$. On suppose que l'image de cet opérateur est incluse dans Deg$_i^{-1}([0,f_i a_i - c])$ pour un certain $c \geq 0$. Alors la majoration de la norme de $U_{\pi_i}^{bad}$ obtenue dans la proposition $\ref{lemnorm}$ ou $\ref{lemnorm_A}$ reste valable.
\end{prop}

\begin{proof}
En raisonnant comme dans les propositions citées, il suffit de majorer la norme du morphisme
$$q^*(\kappa) (\chi) : p_2^* \omega^\kappa (\chi) \to p_1^* \omega^\kappa (\chi)$$
La norme du morphisme $p_2^* \omega^\kappa \to p_1^* \omega^\kappa$ a été majorée dans les propositions précédentes. Le morphisme $p_2^* \mathcal{O}_{X_{1,m}} (\chi) \to p_1^* \mathcal{O}_{X_{1,m}} (\chi)$ étant un isomorphisme naturel, il est de norme égale à $1$. Cela donne la majoration pour la norme de $q^* (\kappa) (\chi)$.
\end{proof}

\section{Appendice}

\subsection{Schémas semi-abéliens} \label{semiab}

Nous rappelons dans cette section la définition et certaines propriétés des schémas semi-abéliens. On peut se référer à \cite{F-C} pour plus de détails.

\begin{defi}
Soit $S$ un schéma. Un schéma semi-abélien $G \to S$ est un schéma en groupes commutatif qui est lisse et séparé, et tel que pout tout point $s \in S$, la fibre $G_s$ de $G$ en $s$ est l'extension d'un tore $T_s$ par une variété abélienne $A_s$ : 
$$ 0 \to T_s \to G_s \to A_s \to 0$$
\end{defi}

Nous avons un théorème de réduction semi-stable.

\begin{theo}[\cite{F-C} Théorème I.2.6]
Soit $V$ un anneau de valuation, de corps de fraction $K$, et $G_K$ une variété semi-abélienne sur $K$. Alors il existe une extension finie $V'$ de $V$, de corps des fractions $K'$, telle que $G_{K'} := G_K \otimes_K K'$ s'étende en un schéma semi-abélien sur $V'$.
\end{theo}

Nous avons également une propriété des extensions des morphismes. Si $L$ est une extension finie de $\mathbb{Q}_p$, d'anneau des entiers $O_L$, et si $G_1$ et $G_2$ sont deux schémas semi-abéliens sur $O_L$, alors tout morphisme en fibre générique $G_1 \otimes_{O_L} L \to G_2 \otimes_{O_L} L$ s'étend de manière unique en un morphisme $G_1 \to G_2$.

\begin{prop}[\cite{F-C} Proposition I.2.7]
Soit $S$ un schéma noethérien normal, et $G_1,G_2$ deux schémas semi-abéliens sur $S$. On suppose que sur un ouvert dense $U$ de $S$ il existe un morphisme $\phi_U : G_1 \times U \to G_2 \times U$. Alors $\phi_U$ s'étend de manière unique en un morphisme $\phi : G_1 \to G_2$.
\end{prop}

Soit $S$ un schéma, et $G$ un schéma semi-abélien sur $S$. Pour tout $s \in S$, on note $rg(s)$ le rang de la partie torique $T_s$.

\begin{prop} [\cite{F-C} Remarque I.2.4 Corollaire I.2.11]
La fonction $s \to rg(s)$ est semi-continue supérieurement. De plus, si cette fonction est localement constante, alors $G$ est globalement extension d'un tore par un schéma abélien. En particulier, $G$ est un tore (resp. un schéma abélien) si et seulement si pour tout $s \in S$, $G_s$ est un tore (resp. une variété abélienne).
\end{prop}

Si $G$ est un schéma semi-abélien sur $O_L$, et si on note $rg(\eta)$ et $rg(s)$ les rangs de la partie torique de $G$ respectivement en fibre générique et en fibre spéciale, alors on a $rg(\eta) \leq rg(s)$. \\
De plus, la partie torique en fibre générique de $G$ peut s'étendre en un tore.

\begin{prop} [\cite{F-C} Proposition I.2.9]
Soit $S$ un schéma noethérien normal, $G$ un schéma semi-abélien sur $S$, et $U$ un ouvert dense de $S$. Si $H_U$ est un sous-groupe fermé de $G \times U$, qui est un tore sur $U$, alors l'adhérence de $U$ dans $G$ est un tore $H \to S$ contenu dans $G$.
\end{prop}

Mumford a établi une construction pour construire certains schémas semi-abéliens, dont la fibre générique est abélienne. Cela généralise la construction de la courbe de Tate. Si $\widetilde{G}$ est globalement extension d'un tore $T$ par un schéma abélien $A$ sur $O_L$, et si $Y$ est un faisceau étale de groupes abéliens libres sur $O_L$ de rang $rg(T)$ avec un morphisme $i : Y \times L \to \widetilde{G} \times L$ vérifiant certaines conditions (voir \cite{F-C} chapitre III), alors Mumford a construit un schéma semi-abélien $G$, que l'on peut voir comme le quotient de $\widetilde{G}$ par $Y$. Cette construction est en fait une équivalence entre certaines catégories.

\begin{theo} [\cite{F-C} Corollaire III.7.2]
Soit $G$ un schéma semi-abélien sur $O_L$ dont la fibre générique est abélienne. Alors il existe un schéma en groupes $\widetilde{G}$ sur $O_L$, globalement extension d'un tore $T$ par un schéma abélien $A$, un faisceau étale $Y$ de groupes abéliens libres de rang $rg(T)$, et un morphisme $i : Y \to \widetilde{G} \times L$, tel que $G$ soit obtenu en quotientant $\widetilde{G}$ par $Y$ via la construction de Mumford. De plus, si $\omega_G$ et $\omega_{\widetilde{G}}$ désignent les faisceaux conormaux de $G$ et $\widetilde{G}$, alors on a un isomorphisme $\omega_G \simeq \omega_{\widetilde{G}}$.
\end{theo}

La construction de Mumford donne une description explicite des groupes de torsion des schémas semi-abéliens considérés.

\begin{prop} [\cite{F-C} Corollaire III.5.11]
Soit $\widetilde{G},Y$ et $G$ comme précédemment. Alors, pour tout $n \geq 1$ on a une suite exacte pour tout $s \in$ Spec$(O_L)$
$$ 0 \to \widetilde{G} [n] \times \kappa(s) \to G [n] \times \kappa(s) \to \frac{1}{n} Y_s / Y_s \to 0 $$
où $\kappa(s)$ est le corps résiduel en $s$ et $Y_s = \{ y \in Y, y \in \widetilde{G} (O_{L,s}) \}$.
\end{prop}

Ainsi, si $H_n$ désigne l'adhérence schématique de $\widetilde{G} [n] \times L$ dans $G[n]$, alors $H_n$ est isomorphe à $\widetilde{G}[n]$, et $G[n] / H_n$ est étale. Par exemple, si $\widetilde{G}= \mathbb{G}_m$, $Y=q^\mathbb{Z}$, alors $G$ est la courbe de Tate, et $H_n$ est isomorphe à $\mu_n$ pour tout $n$. La quantité $\frac{1}{n} Y_s / Y_s$ est dans ce cas isomorphe à $\mathbb{Z} / n \mathbb{Z}$ si $s$ est le point générique, et est nulle si $s$ est le point spécial. \\
En résumé, supposons que l'on dispose d'un schéma $G$ semi-abélien sur $L$. Alors, quitte à étendre le corps $L$, il s'étend en un schéma semi-abélien $G_0$ sur $O_L$. Soit $T_0$ le tore maximal contenu dans $G_0$, et $A_0=G_0 / T_0$ ; $A_0$ est un schéma semi-abélien dont la fibre générique est abélienne. Alors $A_0$ est obtenu par la construction de Mumford en quotientant un schéma semi-abélien $\widetilde{G}$ par un réseau étale $Y$, où $\widetilde{G}$ est globalement extension d'un tore $T_1$ par un schéma abélien $A_1$. On a alors pour tout $n \geq 1$ une suite exacte
$$ 0 \to T_0[n] \to G_0[n] \to A_0[n] \to 0$$
De plus, on a une injection $0 \to \widetilde{G}[n] \to A_0 [n]$, dont le quotient est étale, et une suite exacte
$$ 0 \to T_1[n] \to \widetilde{G}[n] \to A_1[n] \to 0$$
On peut donc filtrer le schéma en groupes $G_0[n]$, avec comme crans de filtration $T_0[n]$, $T_1[n]$, $A_1[n]$ et un schéma en groupes étale. Remarquons que les trois premiers schémas en groupes sont finis et plat sur $O_L$, alors que $G_0[n]$ n'est en général que quasi-fini et plat.

\subsection{Degrés partiels} \label{partial}

Dans cette section, nous définissons les degrés partiels pour les schémas en groupes finis et plats munis d'une action de l'anneau des entiers d'une extension finie non ramifiée de $\mathbb{Q}_p$. Nous appliquerons en particulier ces résultats pour les sous-groupes finis d'un groupe $p$-divisible. \\
Soit $F$ une extension finie non ramifiée de $\mathbb{Q}_p$ de degré $f$, et $O_F$ son anneau des entiers. On a donc $O_F = W(\mathbb{F}_{p^f})$ et $F=O_F[1/p]$. Soit $S$ l'ensemble des plongements de $F$ dans $\overline{\mathbb{Q}}_p$ ; on sait que $S$ est un groupe cyclique d'ordre $f$ engendré par le Frobenius. \\
Soit $K$ une extension finie de $\mathbb{Q}_p$ contenant $F$, et soit $H$ un schéma en groupes fini et plat d'ordre une puissance de $p$ sur $O_K$ muni d'une action de $O_F$ de hauteur $fh$. Soit $\omega_H$ le module des différentielles ; c'est un $O_K$-module de type fini muni d'une action de $O_F$. Alors, on a 
$$\omega_H = \bigoplus_{s \in S} \omega_{H,s}$$
où $\omega_{H,s}$ est le sous-module de $\omega_H$ où $O_{F}$ agit par $s$. 

\begin{defi}
Le degré partiel de $H$ relatif au plongement $s$ de $F$ est défini par 
$$\text{deg}_s H := v( \text{Fitt}_0 \text{ } \omega_{H,s})$$
où $Fitt_0$ désigne l'idéal de Fitting, et la valuation d'un idéal de $O_K$ est définie comme la valuation d'un de ses générateurs.
\end{defi}

On voit immédiatement que le degré de $H$ au sens de Fargues (voir \cite{Fa}) est égal à la somme des deg$_s H$ pour $s \in S$. Nous allons maintenant démontrer des propriétés analogues à la fonction degré pour les degrés partiels.

\begin{prop} 
Les fonctions deg$_s$ sont additives. Plus précisément, soient $H_1$, $H_2$ et $H_3$ trois groupes finis et plats d'ordre une puissance de $p$ munis d'une action de $O_F$ avec une suite exacte
$$ 0 \to H_1 \to H_2 \to H_3 \to 0$$
Alors pour tout $s \in S$
$$\text{deg}_s H_2 = \text{deg}_s H_1 + \text{deg}_s H_3 $$
\end{prop}

\begin{proof}
On a une suite exacte de $O_K \otimes_{\mathbb{Z}_p} O_F$-modules
$$0 \to \omega_{H_3} \to \omega_{H_2} \to \omega_{H_1} \to 0$$
En décomposant cette suite exacte suivant les éléments de $S$, on en déduit des suites exactes
$$0 \to \omega_{H_3,s} \to \omega_{H_2,s} \to \omega_{H_1,s} \to 0$$
pour tout $s \in S$. Le résultat en découle.
\end{proof}

\begin{prop} \label{dual}
Soit $H$ un schéma en groupes fini et plat de d'ordre une puissance de $p$ sur $O_K$ muni d'une action de $O_F$ de hauteur $fh$. Soit $H^D$ le dual de Cartier de $H$ ; c'est encore un schéma en groupes fini et plat sur $O_K$ muni d'une action de $O_F$. Alors pour tout $s \in S$, 
$$\text{deg}_s H^D = h - \text{deg}_s H$$
En particulier, on voit que deg$_s H \in [0 , h]$. 
\end{prop}

\begin{proof}
On se ramène au cas où $H$ est de $p$-torsion. Soit $(\mathfrak{M},\phi)$ le module de Breuil-Kisin de $H$ (voir \cite{Ki}) ; $\mathfrak{M}$ est un $k[[u]]$-module libre de rang $fh$ et $\phi$ est un endomorphisme semi-linéaire tel que $u^e \mathfrak{M}$ soit inclus dans le module engendré par l'image de $\phi$, où $k$ est le corps résiduel de $O_K$ et $e$ son indice de ramification. Le module $\mathfrak{M}$ est muni d'une action de $O_F$, donc se décompose suivant les éléments de $S$ : $\mathfrak{M} = \oplus_{s \in S} \mathfrak{M}_s$. On choisit une bijection entre $S$ et $\mathbb{Z} / f \mathbb{Z}$ de telle sorte que $\phi$ envoie $\mathfrak{M}_i$ dans $\mathfrak{M}_{i+1}$. Les $\mathfrak{M}_i$ sont donc des $k[[u]]$-modules libres de rang $h$. On note $\phi_i : \mathfrak{M}_{i-1} \to \mathfrak{M}_i$. Fixons une base pour les modules $(\mathfrak{M}_i)$, et soit $A_i$ la matrice de $\phi_i$ dans cette base. On a alors
$$ \text{deg}_i H = \frac{1}{e} v_u(\det A_i)$$
où $v_u$ est la valuation $u$-adique. De plus, le module de Breuil-Kisin de $H^D$ est $(\mathfrak{M}^*,\phi^*)$, où $\mathfrak{M}^*$ est le dual de $M$, et où $\phi^*$ peut être décrit comme suit. Le module $\mathfrak{M}^*$ se décompose en $\mathfrak{M}^* = \oplus_{i=0}^{f-1} \mathfrak{M}_i^*$. On muni chaque module $\mathfrak{M}_i^*$ de la base duale de celle des $\mathfrak{M}_i$. Alors la matrice de $\phi_i^* : \mathfrak{M}_{i-1}^* \to \mathfrak{M}_i^*$ dans cette base est $B_i = {u^e} ({}^t\! A_i) ^{-1}$. D'où
$$ \text{deg}_i H^D = \frac{1}{e} v_u(\det B_i) = \frac{1}{e} (eh - v_u(\det A_i)) = h - \text{deg}_i H $$
Une autre démonstration possible aurait été de filtrer le groupe $H$ par des groupes de Raynaud (voir \cite{Ray}), et d'utiliser l'additivité des fonctions degrés (la propriété est évidente pour les groupes de Raynaud car on a une description explicite de ces groupes et de leurs duaux).
\end{proof}

\bibliographystyle{amsalpha}

\begin{thebibliography}{99}

\bibitem[Be]{Be} P. BERTHELOT,
\emph{Cohomologie rigide et cohomologie à support propre}, 
Première partie, prépublication (1996), disponible sur perso.univ-rennes1/pierre.berthelot.

\bibitem[BPS]{Bi} S. BIJAKOWSKI, V. PILLONI et B. STROH
\emph{Classicité de formes modulaires surconvergentes},
Ann. of Math. 183 (2016), no. 3, 975-1014.

\bibitem[Bi]{Bi_Hilbert} S. BIJAKOWSKI,
\emph{Classicité de formes modulaires de Hilbert},
``Arithmetique $p$-adique des formes de Hilbert'', Astérisque 382 (2016), 49-71.

\bibitem[Bo]{Bo} S. BOSCH,
\emph{Half a century of rigid analytic spaces},
Pure Appl. Math. Q., 5(4) : 1435-1467 (2009).

\bibitem[Bu]{Bu} K. BUZZARD,
\emph{Analytic continuation of overconvergent eigenforms},
Jour. Amer. Math. Soc. 16 (2002).

\bibitem[Co]{Co} R. COLEMAN,
\emph{Classical and overconvergent modular forms},
Invent. Math. 124 (1996).

\bibitem[F-C]{F-C} G. FALTINGS et C. L. CHAI,
\emph{Degeneration of Abelian Varieties},
Ergeb. Math. Grenzgeb. (3) 22, Springer-Verlag, Berlin 1990.

\bibitem[Fa]{Fa} L. FARGUES,
\emph{La filtration de Harder-Narasimhan des schémas en groupes finis et plats},
J. Reine Angew. Math. 645 (2010)

\bibitem[Gr]{EGA2} A. GROTHENDIECK,
\emph{Elements de géométrie algébrique : III. Etude cohomologique des faisceaux cohérents, Première partie},
Publ. Math. de l'IHES, tome 11 (1961)  5-167.

\bibitem[Gr2]{EGA} A. GROTHENDIECK,
\emph{Elements de géométrie algébrique : IV. Etude locale des schémas et des morphismes de schémas, Quatrième partie},
Publ. Math. de l'IHES, tome 32 (1967)  5-361.

\bibitem[Jo]{Jo} C. JOHANSSON,
\emph{Classicality for small slope overconvergent automorphic forms on some compact PEL Shimura varieties of type (C)},
Mathematische Annalen, vol. 357(1) (2013), 51-88.

\bibitem[Ka]{Ka} P.L. KASSAEI, 
\emph{A gluing lemma and overconvergent modular forms}, 
Duke Math. J. 132 (2006), 509-529. 

\bibitem[Ki]{Ki} M. KISIN,
\emph{Moduli of finite flat group schemes and modularity},
Annals of Math. 170 (3) (2009), 1085-1180.

\bibitem[Ko]{Ko} R. KOTTWITZ,
\emph{Points on Shimura varieties over finite fields},
J. Amer. Math. Soc. 5 (1992).

\bibitem[La]{LanKoecher} K.-W. LAN,
\emph{Higher Koecher's principle},
Math. Res. Lett. 23 (2016), no. 1, pp. 163--199.

\bibitem[Lü]{Lu} W. L\"UTKEBOHMERT,
\emph{Der Satz von Remmert-Stein in der nichtarchimedischen Funktionentheorie},
Math. Z. 139 (1974) 69-84.

\bibitem[Pi]{Pi} V. PILLONI,
\emph{Prolongements analytiques sur les variétés de Siegel},
Duke Math. J. 157 (2011), 167-222.

\bibitem[PS1]{P-S1} V. PILLONI et B. STROH,
\emph{Surconvergence et classicité : le cas Hilbert},
prépublication (2011).

\bibitem[PS2]{P-S2} V. PILLONI et B. STROH,
\emph{Surconvergence et classicité : le cas déployé},
prépublication (2011).

\bibitem[Pin]{Pin} R. PINK,
\emph{Arithmetical compactification of mixed Shimura varieties},
thèse de doctorat, Université de Bonn (1989).

\bibitem[Ra]{Ra} M. RAPOPORT,
\emph{Compactifications de l'espace de modules de Hilbert-Blumenthal},
Compos. Maths. 36 (1978), 255-335.

\bibitem[Ray]{Ray} M. RAYNAUD,
\emph{Schémas en groupes de type $(p,p,\dots,p)$},
Bull. Soc. Math. de France 102 (1974), 241-280.

\bibitem[Sa]{Sa} S. SASAKI,
\emph{Analytic continuation of overconvergent Hilbert eigenforms in the totally split case},
Compositio Mathematica, 146 (2010), 541-560.

\bibitem[Sa2]{SaHilbert} S. SASAKI,
\emph{Integral models of Hilbert modular varieties in the ramified case, deformations of modular Galois representations, and weight one forms},
preprint (2014).

\bibitem[St]{St} B. STROH,
\emph{Compactifications des variétés de Siegel aux places de mauvaise réduction},
Bull. Soc. Math. France 138 (2010).

\bibitem[Ti]{Ti} Y. TIAN,
\emph{Classicality of overconvergent Hilbert modular forms: case of quadratic inert degree},
prépublication, à paraître à Rendiconti del Seminario Matematico della Universit`a di Padova (2011).

\bibitem[TX]{T-X} Y. TIAN et L. XIAO,
\emph{$p$-adic cohomology and classicality of overconvergent Hilbert modular forms},
``Arithmetique $p$-adique des formes de Hilbert'', Astérisque 382 (2016), 73-162.

\bibitem[We]{We} T. WEDHORN, 
\emph{Ordinariness in good reductions of Shimura varieties of PEL-type},
Ann. Scient. Ec. Norm. Sup. 32 (1999), 575-618

\end{thebibliography}

\end{document}